\documentclass[oneside,english]{amsart}
\usepackage[T1]{fontenc}
\usepackage{color}
\usepackage{amsthm}
\usepackage{bm}
\usepackage{amstext}
\usepackage{amssymb}
\usepackage[all]{xy}

\makeatletter
\numberwithin{equation}{section} 
\numberwithin{figure}{section} 
\theoremstyle{plain}
\newtheorem{thm}{Theorem}[section]
  \theoremstyle{definition}
  \newtheorem{defn}[thm]{Definition}
\newtheorem*{ackn}{Acknowledgement}
  \theoremstyle{plain}
  \newtheorem{lem}[thm]{Lemma}
  \theoremstyle{plain}
  \newtheorem{prop}[thm]{Proposition}
  \newtheorem{cor}[thm]{Corollary}
  \theoremstyle{remark}
  \newtheorem*{rem*}{Remark}
 \theoremstyle{definition}
  \newtheorem{example}[thm]{Example}

  \theoremstyle{remark}
  \newtheorem{rem}[thm]{Remark}

\usepackage{amsmath}
\usepackage{amsfonts}
\usepackage{amsthm}
\usepackage{amssymb}
\usepackage{amscd}
\usepackage[all]{xy}
\usepackage{enumerate}
\usepackage{mathrsfs}
\usepackage{graphicx, bm, color, ulem}

\def\ft#1{{\mathsf #1}}
\def\chow{{\mathscr X}}
\def\hcoY{{\mathscr Y}}
\def\Hes{{\mathscr H}}
\def\UU{{\mathscr U}}
\def\Zpq{{\mathscr Z}}

\def\Prt{{\mathscr P}}
\def\Lrho{\text{\large{\mbox{$\rho\hskip-2pt$}}}}
\def\tLrho{\text{\large{\mbox{$\tilde\rho\hskip-2pt$}}}}
\def\Lpi{\text{\large{\mbox{$\pi\hskip-2pt$}}}}
\def\Lp{\text{\large{\mbox{$p\hskip-1pt$}}}}

\font\eu=eusm10 at 10pt

\def\eF{\text{\eu F}}
\def\eG{\text{\eu G}}
\def\eQ{\text{\eu W}}
\def\eS{\text{\eu U}}

\def\oF#1{\overset{\,_\circ\;\;\;}{F^{#1}}}

\input xyimport.tex

\newcommand{\vcorr}[3][1]{%
  \begingroup
    \tabcolsep=.5\tabcolsep
    \sbox0{%
      \begin{tabular}[b]{@{}l}%
        #3%
         \tabularnewline
      \end{tabular}%
    }%
    \settoheight{\dimen0 }{%
      \rotatebox{#2}{%
        \copy0 %
        \kern-\tabcolsep
      }%
    }%
    \rule{0pt}{#1\dimen0}%
    \setlength{\wd0 }{1em}%
    \setlength{\ht0 }{1em}%
    \rotatebox{#2}{\usebox{0}}%
  \endgroup
}

\newenvironment{fcaption}{\begin{list}{}{
\setlength{\leftmargin}{35pt}
\setlength{\rightmargin}{35pt}
\setlength{\labelsep}{5pt}
}}{\end{list}}

\newenvironment{myitem2}{\begin{list}{}{
\setlength{\leftmargin}{0.5cm}
\setlength{\itemindent}{-0.3cm}
\setlength{\itemsep}{0cm}
}}{\end{list}}

\makeatother

\usepackage{babel}

\begin{document}
\global\long\def\sA{\mathcal{A}}
 \global\long\def\sB{\mathcal{B}}
 \global\long\def\sC{\mathcal{C}}
 \global\long\def\sD{\mathcal{D}}
 \global\long\def\sE{\mathcal{E}}
 \global\long\def\sF{\eF}
 \global\long\def\sG{\eG}
 \global\long\def\sH{\mathcal{H}}
 \global\long\def\sI{\mathcal{I}}
 \global\long\def\sJ{\mathcal{J}}
 \global\long\def\sK{\mathcal{K}}
 \global\long\def\sL{\mathcal{L}}
 \global\long\def\sN{\mathcal{N}}
 \global\long\def\sM{\mathcal{M}}
 \global\long\def\sO{\mathcal{O}}
 \global\long\def\sP{\mathcal{P}}
 \global\long\def\sS{\mathcal{S}}
 \global\long\def\sR{\mathcal{R}}
 \global\long\def\sQ{\mathcal{Q}}
 \global\long\def\sT{\mathcal{T}}
 \global\long\def\sU{\mathcal{U}}
 \global\long\def\sV{\mathcal{V}}
 \global\long\def\sW{\mathcal{W}}
 \global\long\def\sX{\mathcal{X}}
 \global\long\def\sY{\mathcal{Y}}
 \global\long\def\sZ{\mathcal{Z}}
 \global\long\def\tA{{\widetilde{A}}}
 \global\long\def\mA{\mathbb{A}}
 \global\long\def\mC{\mathbb{C}}
 \global\long\def\mF{\mathbb{F}}
 \global\long\def\mG{\mathbb{G}}
 \global\long\def\G{{\bf G}}
 \global\long\def\mN{\mathbb{N}}
 \global\long\def\mP{\mathbb{P}}
 \global\long\def\mQ{\mathbb{Q}}
\def\frQ{{\mathfrak Q}}
 \global\long\def\mZ{\mathbb{Z}}
 \global\long\def\mW{\mathbb{W}}
 \global\long\def\Ima{\mathrm{Im}\,}
 \global\long\def\Ker{\mathrm{Ker}\,}
 \global\long\def\Alb{\mathrm{Alb}\,}
 \global\long\def\ap{\mathrm{ap}}
 \global\long\def\Bs{\mathrm{Bs}\,}
 \global\long\def\Chow{\mathrm{Chow}}
 \global\long\def\CP{\mathrm{CP}}
 \global\long\def\Div{\mathrm{Div}\,}
 \global\long\def\divi{\mathrm{div}\,}
 \global\long\def\expdim{\mathrm{expdim}\,}
 \global\long\def\ord{\mathrm{ord}\,}
 \global\long\def\Aut{\mathrm{Aut}\,}
 \global\long\def\Hilb{\mathrm{Hilb}}
 \global\long\def\Hom{\mathrm{Hom}}
 \global\long\def\id{\mathrm{id}}
 \global\long\def\Ext{\mathrm{Ext}}
 \global\long\def\sHom{\mathcal{H}{\!}om\,}
 \global\long\def\Lie{\mathrm{Lie}\,}
 \global\long\def\mult{\mathrm{mult}}
 \global\long\def\opp{\mathrm{opp}}
 \global\long\def\Pic{\mathrm{Pic}\,}
 \global\long\def\Pf{{\bf Pf}}
 \global\long\def\Sec{\mathrm{Sec}}
 \global\long\def\Spec{\mathrm{Spec}\,}
 \global\long\def\Sym{\mathrm{Sym}}
 \global\long\def\sQpec{\mathcal{S}{\!}pec\,}
 \global\long\def\Proj{\mathrm{Proj}\,}
 \global\long\def\Rhom{{\mathbb{R}\mathcal{H}{\!}om}\,}
 \global\long\def\aw{\mathrm{aw}}
 \global\long\def\exc{\mathrm{exc}\,}
 \global\long\def\emb{\mathrm{emb\text{-}dim}}
 \global\long\def\codim{\mathrm{codim}\,}
 \global\long\def\OG{\mathrm{OG}}
 \global\long\def\pr{\mathrm{pr}}
 \global\long\def\Sing{\mathrm{Sing}\,}
 \global\long\def\Supp{\mathrm{Supp}\,}
 \global\long\def\SL{\mathrm{SL}\,}
 \global\long\def\Reg{\mathrm{Reg}\,}
 \global\long\def\rank{\mathrm{rank}\,}
 \global\long\def\VSP{\mathrm{VSP}\,}
 \global\long\def\B{B}
 \global\long\def\Q{Q}
 \global\long\def\rG{\mathrm{G}}
 \global\long\def\rF{\mathrm{F}}
 \global\long\def\nS{\textsc{S}}
 \global\long\def\nT{\textsc{T}}
 \global\long\def\nU{\textsc{U}}
 \global\long\def\nR{\textsc{R}}


\title{Geometry of symmetric determinantal loci}
\author{Shinobu Hosono and Hiromichi Takagi}
\begin{abstract}
We study algebro-geometric properties of determinantal loci of 
$(n+1)\times(n+1)$ symmetric matrices and also their double covers for even 
ranks. Their singularities, Fano indices and birational geometries are 
studied in general. The double covers of symmetric determinantal loci 
of rank four are studied with special interest by noting their relation to 
the Hilbert schemes of conics on Grassmannians. 
\end{abstract}
\maketitle

\section{{\bf Introduction}}

Throughout this paper,
we work over $\mC$, the complex number field,
and we fix a vector space $V$ of dimension $n+1$.

We define $\nS_r\subset \mP({\ft S}^2 V^*)$ to be the locus 
of quadrics in $\mP(V)$ of rank at most $r$.
Taking a basis of $V$, $\nS_r$ is defined by $(r+1)\times (r+1)$ minors
of the generic $(n+1)\times (n+1)$ symmetric matrix.
We call $\nS_r$ the {\it symmetric determinantal locus of rank at most $r$}.
For example, $\nS_1=v_2(\mP(V^*))$ with 
$v_2(\mP(V^*))$ being the second Veronese variety of $\mP(V^*)$ 
 and $\nS_{n+1}=\mP({\ft S}^2 V^*)$. 
There is a natural stratification of $\mP({\ft S}^2 V^*)$ by $\nS_r$:
\[
v_2(\mP(V^*))=\nS_1\subset \nS_2\subset \cdots \subset \nS_n\subset 
\nS_{n+1}=\mP({\ft S}^2 V^*).
\] 
We call a point of $\nS_r\setminus \nS_{r-1}$ {\it a rank $r$ point}.
Similarly we define the symmetric determinantal locus $\nS_r^*$ in 
the dual projective space $\mP(\ft{S}^2V)$.  It is a well-known fact that the 
stratification of $\mP(\ft{S}^2V^*)$ by $\nS_r$ and that of $\mP(\ft{S}^2V)$ 
by $\nS_r^*$ are reversed under the projective duality. 

Recently, classical projective duality is highlighted in the study of derived categories of coherent sheaves on projective varieties, where the duality is called {\it homological projective duality} (HPD) due to Kuznetsov \cite{HPD1}. HPD is a powerful framework to describe the derived category of a projective variety with its dual variety, and has been worked out in several interesting examples such as Pfaffian varieties (i.e., determinantal loci of anti-symmetric matrices) \cite{HPD2} and the second Veronese variety $\nS^*_1$  \cite{Quad}. Interestingly, it is often the case that we have interesting pairs of Calabi-Yau manifolds associated to HPDs 
\cite{BC,HPD2}.
%
%
%
%
In a series of papers \cite{HoTa1}--\cite{HoTa3},
we have studied the case $\nS_2^*$ and $\nS_4$ for $n=4$ in detail, where 
a pair of smooth Calabi-Yau threefolds $X$ and $Y$ appears, respectively, as a 
linear section of  $\nS^*_2$ and the double cover of the orthogonal linear 
section of $\nS_4$ branched along the set of rank 3 points. 
It has been shown in \cite{HoTa3} that these $X$ and $Y$ are derived-equivalent, indicating that $\nS_2^*$ and the double cover $\nT_4$ of $\nS_4$ (called {\it double quintic symmetroids}) are HPD to each other. 
Also, for $n=3$, we have established in \cite{ReyeEnr} the relations between the derived categories of a $2$-dimensional linear section $X$ of $\nS^*_2$ and the double cover $Y$ of the orthogonal linear section of $\nS_4$ branched along the set of rank 2 or 3 points after the inspiring works \cite{Lines} and \cite{IK}. In the latter case of $n=3$, $X$ is known  as an Enriques surface of Reye congruence, while $Y$ is known as an Artin-Mumford double solid.

The aim of the present paper is to put an algebro-geometric ground for our work \cite{HoTa3}. Indeed this is an extended version of the first part of \cite{Arxiv}. In a companion paper \cite{DerSym}, we will study homological properties of $\nS_2^*$ and $\nT_4$ for the cases $n=3,4$ based on the results of this paper.
%
%
In this paper, we are concerned with the birational geometry of $\nS_r$ for general $n$ from the viewpoint of minimal model theory. In particular, for even $r$, we  present a precise description of the double covers $\nT_r$ of $\nS_r$ branched along $\nS_{r-1}$. If $r\leq n$, we show that $\nS_r$ and $\nT_r$ are $\mQ$-factorial $\frac{(2n+3-r)r-2}{2}$-dimensional Fano varieties with Picard number one and Fano index $\frac{r(n+1)}{2}$ with only canonical singularities
in Subsection \ref{subsection:Spr}. 

As an interesting application of these general results, we will consider {\it orthogonal} linear sections of $\nS^*_{n+2-r}$ and $\nT_r$, which entail a pair of Calabi-Yau varieties of the same dimensions.  These Calabi-Yau varieties naturally generalize those studied in  \cite{HoTa3, Arxiv, DerSym} for $n=4$, and indicates that HPD holds for  $\nS^*_{n+2-r}$ and $\nT_r$ (see Subsection \ref{subsection:Pl}).

Below is the summary of the birational geometry of the double covering $T_4$ of $S_4$ for genreal $n$ which we establish in this paper. 
Note that a general point of $\nS_4$ corresponds to a quadric of rank four in $\mP(V)$. It 
has two connected $\mP^1$-families of $(n-2)$-planes which we identify with the respective conics in $\rG(n-1,V)$. The double cover $\nT_4$ will be defined as the space which parametrizes the connected families of $(n-2)$-planes in quadrics, and will be described by making precise connection to the Hibert scheme of conics in $\rG(n-1,V)$. In Section \ref{section:BirY}, we show the following:
\begin{thm}
Set $\hcoY:=\nT_4$ and denote by $\hcoY_0$
the Hilbert scheme of conics in $\rG(n-1,V)$.
Then there is a commutative diagram of birational maps as follow\,$:$
\[
\xymatrix{ & &  \hcoY_0\ar[d]\\
\hcoY_3\ar[dr]\ar@{-->}[rr]^{\text{\tiny{\text{$($anti-$)$flip}}}} & & \widetilde{\hcoY}\ar[dl]\ar[dr]^{\Lrho\,_{\widetilde{\hcoY}}} & \\
 & \overline{\hcoY}' & &\;\;\;\;\hcoY:=\nT_4,}
\]
where
\begin{itemize}
\item $\hcoY_3:=\mathrm{G}(3,\wedge^2 \mathfrak{Q})$
with the universal quotient bundle $\frQ$ of $\rG(n-3,V)$,
\item $\overline{\hcoY}'$ is the normalization
of the subvariety $\overline{\hcoY}$ of $\rG(3,\wedge^{n-1} V)$
parametrizing $3$-planes annihilated by at least $n-3$ linearly independent vectors in $V$ by the wedge product {$\text{\rm
(Propositions~\ref{prop:barY-1-2-3}, \ref{lem:appendixB-UU-solve})}$},
\item
$\hcoY_3\to \overline{\hcoY}'$
is a small contraction 
with non-trivial fibers being copies of $\mP^{n-3}$ 
{$\text{\rm (Proposition~\ref{pro:barYsing})}$},
\item
$\hcoY_3\dashrightarrow \widetilde{\hcoY}$
is the $($anti-$)$\,flip for the small contraction  
$\hcoY_3\to \overline{\hcoY}'$ 
{$\text{\rm (Section \ref{subsection:BlowUp})}$,}
\item
$\widetilde{\hcoY}\to \overline{\hcoY}'$
is a small contraction
with non-trivial fibers being copies of $\mP^5$ 
{$\text{\rm (Proposition~\ref{prop:tildeY})}$},
\item
$\Lrho\,_{\widetilde{\hcoY}}\colon \widetilde{\hcoY}\to \hcoY$ is an extremal 
divisorial contraction 
{$\text{\rm (Proposition~\ref{prop:gendescr}(2))}$},
\item
$\hcoY_0\to \widetilde{\hcoY}$ is the blow-up along a smooth subvariety 
{$\text{\rm (Section \ref{subsection:BlowUp})}$}.
\end{itemize}
\end{thm}
In the course of the proof, we give an explicit construction of the Hilbert scheme $\hcoY_0$ of conics in $\rG(n-1,V)$ in Subsection \ref{subsection:Hilb}.
In Section \ref{section:FY}, the contraction $\Lrho\,_{\widetilde{\hcoY}}\colon \widetilde{\hcoY}\to \hcoY$ is studied in detail. 
Let $F_{\widetilde{\hcoY}}$ be $\Lrho\,_{\widetilde{\hcoY}}$-exceptional divisor
and  $G_{\hcoY}$ be its image in $\hcoY$. 
We determine the biregular structure of $F_{\widetilde{\hcoY}}\to G_{\hcoY}$
by introducing a natural double cover of $F_{\widetilde{\hcoY}}$.  
Flattening of the morphism $F_{\widetilde{\hcoY}}\to G_{\hcoY}$ is constructed in Section \ref{section:FY}.  Despite its technical nature, the flat morphism plays crucial roles for our caluculations of the cohomologies of $\hcoY$ in \cite{DerSym}. 

\begin{ackn}
This paper is supported in part by Grant-in
Aid Scientific Research (S 24224001, B 23340010 S.H.) and Grant-in
Aid for Young Scientists (B 20740005, H.T.).  They also thank Nicolas
Addington and Sergey Galkin for useful communications.
\end{ackn}

\noindent 

\vspace{0.5cm}
\noindent\textbf{\textcolor{black}{Notation: }}
We will denote by $V_i$ an $i$-dimensional vector subspace of $V$.

\vspace{0cm}

\global\long\def\Homega{H_{\mP(\Omega(1))}}
 \global\long\def\Hwomega{H_{\mP(\Omega(1)^{\wedge2})}}
 \global\long\def\Lwomega{L_{\mP(\Omega(1)^{\wedge2})}}
 \global\long\def\HGo{H_{\mathrm{G}(\Omega(1)^{\wedge2})}}
 \global\long\def\HGt{H_{\mathrm{G}(\wedge^{2} T(-1))}}

\global\long\def\Lt{L_{\mP(T(-1))}}
 \global\long\def\Lwt{L_{\mP(T(-1)^{\wedge2})}}
 \global\long\def\Ht{H_{\mP(T(-1))}}
 \global\long\def\Hwt{H_{\mP(T(-1)^{\wedge2})}}


\section{{\bf Basics for symmetric determinantal loci $\nS_r$}}

As introduced in the preceding section, we denote by $\nS_r\subset \mP({\ft S}^2 V^*)$ the locus of quadrics in $\mP(V)$ of rank at most $r$. 

\subsection{Springer type resolution $\widetilde{\nS}_r$ of $\nS_r$}
\label{subsection:Spr}
Let $\mathfrak{Q}$ be 
the universal quotient bundle of rank $r$ on $\rG(n+1-r, V)$ and 
define the following projective bundle over $\rG(n+1-r, V)$: 
\begin{equation} \label{eq:deftSr}
\widetilde{\nS}_r:=\mP(\ft{S}^{2}\mathfrak{Q}^{*})\to \rG(n+1-r,V).
\end{equation}
When $r=n+1$, we consider this as the projective bundle over a point
\[
\widetilde{\nS}_{n+1}=\mP({\ft S}^2 V^*)\to {\rm pt} 
\]
{with $\widetilde{\nS}_{n+1}=\nS_{n+1}$.} 
Considering the (dual of the) universal exact sequence,
we see that there is a canonical injection $\mathfrak{Q}^{*}\hookrightarrow V^{*}\otimes\sO$,
which entails the injection $\ft{S}^{2}\mathfrak{Q}^{*}\hookrightarrow\ft{S}^{2}V^{*}\otimes\sO.$
With this injection, composed with the natural surjection $\mP(\ft{S}^{2}V^{*}\otimes\sO)\to\mP(\ft{S}^{2}V^{*})$,
we have a morphism\begin{equation}
\widetilde{\nS}_r=
\mP(\ft{S}^{2} \mathfrak{Q}^{*})\to\mP(\ft{S}^{2}V^{*}).\label{eq:UUtoS2V}\end{equation}

By construction,
the pull-back of $\sO_{\mP({\ft S}^2 V^*)}(1)$ to $\widetilde{\nS}_r$ is
the tautological divisor $\sO_{\mP({\ft S}^2 \frQ\,^*)}(1)$,
which we denote by $M_{\widetilde{\nS}_r}$.

\begin{prop}
\label{prop:Spr}~

\noindent{\rm (1)} The image of the morphism $($\ref{eq:UUtoS2V}$)$ 
coincides with $\nS_r$.
The induced morphism $\Lp\,_{\widetilde{\nS}_r}\colon \widetilde{\nS}_r\to \nS_r$
is a resolution of $\nS_r$.

\noindent{\rm (2)} $\widetilde{\nS}_r=\{([V_{n+1-r}],[Q])\mid V_{n+1-r}\subset 
\Sing Q
\}\subset\rG(n+1-r,V)\times\mP(\ft{S}^{2}V^{*})$,
where $Q$ is {a quadric} in $\mP(V)$.\end{prop}
\begin{proof}
(1) Since the fiber of $\mathfrak{Q}^{*}$ over a point $[V_{n+1-r}]\in\rG(n+1-r,V)$
is $(V/V_{n+1-r})^{*}$, the fiber of the projective bundle $\widetilde{\nS}_r\to\rG(n+1-r,V)$
over $[V_{n+1-r}]$ is $\mP(\ft{S}^{2}(V/V_{n+1-r})^{*})$,
which parameterizes quadrics in $\mP(V/V_{n+1-r})\simeq \mP^{r-1}$.
The morphism $\mP({\ft S}^2 \frQ^*)\to \mP({\ft S}^2 V^*)$
sends $\mP(\ft{S}^{2}(V/V_{n+1-r})^{*})$ into $\mP(\ft{S}^{2}V^{*})$.
Then the image is identified with quadrics in $\mP(V)$ which are
singular at $[V_{n+1-r}]$, or equivalently, symmetric matrices whose
kernels contain $[V_{n+1-r}]$. 
Therefore the image is $\nS_r$.
The morphism $\Lp\,_{\widetilde{\nS}_r}\colon \widetilde{\nS}_r\to \nS_r$ is
one to one over the locus of matrices of rank $r$ in $\nS_r$, since
a symmetric matrix of rank $r$ with the kernel $V_{n+1-r}$
{determines} uniquely the corresponding
quadric in $\mP(V/V_{n+1-r})$. Hence $\widetilde{\nS}_r$ is birational 
to $\nS_r$ under {$\Lp\,_{\widetilde{\nS}_r}$.}
Finally, $\widetilde{\nS}_r$ is smooth since it is a projective bundle, and hence
$\Lp\,_{\widetilde{\nS}_r}$ is a resolution of $\nS_r$. 

The assertion (2) easily follows from the proof of (1). 
\end{proof}

Using the Springer type resolution $\Lp\,_{\widetilde{\nS}_r}$,
we can derive several properties of $\nS_r$.

\noindent $\bullet$ {\bf Dimension.} 
Since $\widetilde{\nS}_r$ is a $\mP^{\binom{r+1}{2}-1}$-bundle over $\rG(n+1-r,V)$, it holds 
\begin{equation}
\label{eq:dimSr}
\dim \nS_r=\dim \widetilde{\nS}_r=
\frac{(r+1)r}{2}-1+r(n+1-r).
\end{equation}

\noindent $\bullet$ {\bf Canonical divisor.}
Since $\widetilde{\nS}_r=\mP({\ft S}^2 \mathfrak{Q}^*)$ and
$\det {\ft S}^2 \mathfrak{Q}\simeq \sO_{\rG(n+1-r,V)}(r+1)$,
we have
\begin{equation}
\label{eq:adjSr1}
K_{\widetilde{\nS}_r}=-\binom{r+1}{2} M_{\widetilde{\nS}_r}
-(n-r)L_{\widetilde{\nS}_r},
\end{equation}
where 
$M_{\widetilde{\nS}_r}$ is the tautological divisor of
$\mP({\ft S}^2 \mathfrak{Q}^*)$ and
$L_{\widetilde{\nS}_r}$ is the pull-back of 
$\sO_{\rG(n+1-r,V)}(1)$.

\vspace{5pt}

In the sequel in this subsection, we assume that $r\leq n$.

\vspace{5pt}

\noindent $\bullet$ {\bf Exceptional divisor.}
By Proposition \ref{prop:Spr} (2) and $\rho(\widetilde{\nS}_r/\nS_r)=1$,
the exceptional locus $E_r$ of $\Lp\,_{\widetilde{\nS}_r}$ is a prime divisor
and the induced map $E_r\to \nS_{r-1}$ is a $\mP^{n+1-r}$-bundle over
$\nS_{r-1}\setminus \nS_{r-2}$.
We have 
\begin{equation}
\label{eq:Er}
E_r=rM_{\widetilde{\nS}_r}-2L_{\widetilde{\nS}_r}.
\end{equation}
Indeed, 
note that we may write 
$E_r=aM_{\widetilde{\nS}_r}-bL_{\widetilde{\nS}_r}$ with some integers $a$ and $b$ since $M_{\widetilde{\nS}_r}$ and $L_{\widetilde{\nS}_r}$ generate
$\Pic \widetilde{\nS}_r$.
Let $\mP\simeq \mP^{n+1-r}$ be the fiber of 
$E_r\to \nS_{r-1}$ over a point of
$\nS_{r-1}\setminus \nS_{r-2}$.
Then, by (\ref{eq:adjSr1}) and
$M_{\widetilde{\nS}_r}|_{\mP}=0$,
we have $K_{\widetilde{\nS}_r}|_{\mP}=\sO_{\mP}(-(n-r))$.
Therefore, using $K_{\mP}=
K_{E_r}|_{\mP}=(K_{\widetilde{\nS}_r}+E_r)|_{\mP}$,
we obtain
$E_r|_{\mP}=\sO_{\mP}(-2)$.
Thus $b=2$.
We have $a=r$
since the restriction of $E_r$ to a fiber 
$\mP({\ft S}^2 (V/V_{n+1-r})^*)$
of $\widetilde{\nS}_r\to \rG(n+1-r,V)$ 
is the locus of singular quadrics in
$\mP(V/V_{n+1-r})$, and it is a degree $r$ hypersurface in
$\mP({\ft S}^2 (V/V_{n+1-r})^*)$.

\noindent $\bullet$ {\bf Generic Singularity.}
By $E_r|_{\mP}=\sO_{\mP}(-2)$,
we see that
\begin{equation}
\label{equation:SingSr}
\text{${\nS}_r$ has $\frac 12 (1^{n+2-r})$-singularities along $S_{r-1}\setminus S_{r-2}$},
\end{equation}
hence $\Sing \nS_r=\nS_{r-1}$.

\noindent $\bullet$ {\bf Discrepancy and Fano index.}
The two equalities (\ref{eq:adjSr1}) and (\ref{eq:Er}) give the following 
presentation of
$K_{\widetilde{\nS}_r}$:
\begin{equation}
\label{eq:adjSr2}
K_{\widetilde{\nS}_r}{=}_{{\mQ}} -\frac{r(n+1)}{2} M_{\widetilde{\nS}_r}
+\frac{n-r}{2} E_r.
\end{equation}
The pushforward of (\ref{eq:adjSr2}) immediately gives
\begin{equation}
\label{eq:Fanoindex}
K_{{\nS}_r}=_{\mQ} -\frac{r(n+1)}{2} M_{{\nS}_r}.
\end{equation}
Combining (\ref{eq:adjSr2}) and (\ref{eq:Fanoindex}),
we obtain
\[
K_{\widetilde{\nS}_r}{=}_{{\mQ}} \Lp\,_{\widetilde{\nS}_r}^*K_{{\nS}_r}
+\frac{n-r}{2} E_r.
\] 
In particular,
$\nS_r$ has only terminal singularities if $n>r$,
and canonical singularities if $n=r$.
$\nS_r$ is $\mQ$-factorial since $\widetilde{\nS}_r$ is smooth and
$\Lp\,_{\widetilde{\nS_r}}$ is a divisorial contraction.

\noindent $\bullet$
{\bf Gorenstein index.}
$K_{\nS_r}$ is Cartier in case $n-r$ is even. 
In case $n-r$ is odd, $2K_{\nS_r}$ is Cartier while $K_{\nS_r}$ is not.
%

Indeed, when $n-r$ is even,
the integral divisor $K_{\widetilde{\nS}_r}-\frac{n-r}{2} E_r$ is the pull-back of a Cartier
divisor on $\nS_r$
by the Kawamata-Shokurov base point free theorem. 
Then, in this case, the formulas (\ref{eq:adjSr2})
and (\ref{eq:Fanoindex})
mean linear equivalences.
In particular, $K_{\nS_r}$ is Cartier.
In case $n-r$ is odd, we see the assertion by a similar argument
and (\ref{equation:SingSr}).

\subsection{Double cover $\nT_r$ of $\nS_r$ with even $r$}
Throughout in this subsection, we suppose $r$ is even.
When $r$ is even, due to the fact that a quadric of even rank contains 
two connected families of maximal linear subspaces in it, the determinantal 
locus $\nS_r$ has a natural double cover. We describe below the double cover by 
formulating Springer type morphism.

%
%

Note that
any quadric of rank at most $r$ contains
$(n-\frac r2)$-planes.
We will introduce the variety $\nU_r$ which parameterizes pairs
$([\Pi],[Q])$ of quadrics $Q$ of rank at most $r$ 
and $(n-\frac r2)$-planes $\mP(\Pi)$
such that $\mP(\Pi)\subset Q$. To parametrize $(n-\frac r2)$-planes
in $\mP(V)$, consider the Grassmannian $\mathrm{G}(n-\frac r2 +1,V)$.
Let \begin{equation}
0\to{\eQ}_{\frac r2}^{*}\to V^{*}\otimes\sO_{\mathrm{G}(n-\frac r2+1,V)}\to\eS_{n-\frac r2+1}^{*}\to0\label{eq:Q*S}\end{equation}
 be the dual of the universal exact sequence on $\mathrm{G}(n-\frac r2+1,V)$,
where $\eQ_{\frac r2}$ is the universal quotient bundle of rank $\frac r2$
and $\eS_{n-\frac r2+1}$ is the universal subbundle of rank $n-\frac r2+1$. 
For brevity, we often omit the subscripts writing them by $\eS$ and $\eQ$.
For an $(n-\frac r2)$-plane $\mP(\Pi)\subset\mP(V)$, there exists a natural
surjection $\ft{S}^{2}V^{*}\to\ft{S}^{2}H^{0}(\mP(\Pi),\sO_{\mP(\Pi)}(1))$
such that the projectivization of the kernel consists of the quadrics
containing $\mP(\Pi)$. By relativizing this surjection over $\mathrm{G}(n-\frac r2+1,V)$,
we obtain the following surjection: $\ft{S}^{2}V^{*}\otimes\sO_{\mathrm{G}(n-\frac r2-1,V)}\to\ft{S}^{2}\eS^{*}.$
Let $\sE^{*}$ be the kernel of this surjection, and consider the
following exact sequence: \begin{equation}
0\to\sE^{*}\to\ft{S}^{2}V^{*}\otimes\sO_{\mathrm{G}(n-\frac r2+1,V)}\to\ft{S}^{2}\eS^{*}\to0.\label{eq:sE0}\end{equation}
Now we set $\nU_r:=\mP(\sE^{*})$ and denote by $\Lrho\,_{\nU_r}$ the projection
$\nU_r\to\mathrm{G}(n-\frac r2+1,V)$. By (\ref{eq:sE0}), $\nU_r$
is contained in $\mathrm{G}(n-\frac r2+1,V)\times\mP(\ft{S}^{2}V^{*})$.
Since the fiber of $\sE^{*}$ over $[\Pi]$ parameterizes quadrics
in $\mP(V)$ containing $\mP(\Pi)$, we have \[
\nU_r=\{([\Pi],[Q])\mid\mP(\Pi)\subset Q\}\subset\mathrm{G}(n-\frac r2+1,V)\times\mP(\ft{S}^{2}V^{*}). \]
 Note that $Q$ in $([\Pi],[Q])\in\nU_r$ is a quadric of rank at most $r$
since quadrics contain $(n-\frac r2)$-planes only when their ranks are 
at most $r$. 
%
%
Hence the
symmetric determinantal locus $\nS_r$ is the image of the natural projection $\nU_r\to\mP(\ft{S}^{2}V^{*})$.
Now we let \[
\xymatrix{ & \nU_r\;\ar[r]^{\;\;\Lpi\,_{\nU_r}\;\;} & \;\nT_r\;\ar[r]^{\;\;\Lrho\,\,_{\nT_r}\;\;} & \;\nS_r}
\]
be the Stein factorization of $\nU_r\to\nS_r$. By (\ref{eq:sE0}),
the tautological divisor of $\mP(\sE^{*})\to\mathrm{G}(n-\frac r2+1,V)$
is nothing but the pull-back of a hyperplane section of $\nS_r$. We
set \[
M_{\nU_r}:=\Lpi\,_{\nU_r}^{\;*}\circ\Lrho\,_{\nT_r}^{\;*}\sO_{\nS_r}(1).\]
 We denote by $\nU_{r[Q]}$ the fiber of $\nU_r\to\nS_r$ over a point
$[Q]\in\nS_r$.

\vspace{0.3cm}
 
\begin{prop}
\label{Z_Q} 
For a quadric $Q$ of rank $r$, the fiber ${\nU}_{r[Q]}$
is the orthogonal Grassmannian $\OG(\frac r2,r)$ \textcolor{black}{which
consists of two connected components. }
\end{prop}
\begin{proof} Quadric $Q$ of even rank $r$ induces a non-degenerate
symmetric bilinear form $q$ on the quotient $V/V_{n+1-r}$, where $V_{n+1-r}$
is the $(n+1-r)$-dimensional vector space such that $[V_{n+1-r}]$ is the vertex
of $Q$. Then $(n-\frac r2)$-planes on $Q$ naturally correspond to
the maximal isotropic subspaces in $V/V_{n-r+1}$ with respect to $q$,
which are parameterized by the orthogonal Grassmannian $\OG(\frac r2,r)$.
\end{proof}

\vspace{0.3cm}

\begin{prop}
\label{cla:double} The finite morphism $\nT_r\to\nS_r$ is of degree two
and is branched along $\nS_{r-1}$. 
\end{prop}
\begin{proof} By Proposition \ref{Z_Q}, the degree of $\nT_r\to\nS_r$
is two since $\nU_{r[Q]}$ has two connected components for a quadric
$Q$ of rank $r$. If a quadric $Q$ has rank at most $r-1$, the family of $(n-\frac r2)$-planes in $Q$ is connected.
Hence we have the assertion. \end{proof}

By this proposition, we see that $\nT_r$ parameterizes connected
families of $(n-\frac r2)$-planes in quadrics of rank at most $r$ 
in $\mP(V)$
(cf. Fig.1).
\begin{defn}
We call $\nT_r$ the {\it double symmetric determinantal locus} of rank at most $r$.
We call a point of $\Lrho\,_{\nT_r}^{-1}(\nS_i\setminus \nS_{i-1})$
{\it a rank $i$ point} for $1\leq i\leq r$.
\end{defn}

$\nT_r$ inherits good properties from $\nS_r$ as follows:
 
\begin{prop}
\label{cla:ZY} 
\begin{enumerate}[$(1)$]
\item
The Picard number of $\nU_r$ is two and $\Lpi\,_{\nU_r}\colon\nU_r\to\nT_r$
is a Mori fiber space. In particular, $\nT_r$ is $\mQ$-factorial
and has Picard number one.
\item $\nT_r$ has only Gorenstein canonical singularities
and $\Sing \nT_r$ is contained in 
the inverse image of $\nS_{r-2}$.
In particular, $\dim \Sing \nT_r$ is smaller than
$\dim \Sing \nS_r$ in case $r\leq n$. 
\item 
$\nT_r$ is a Fano variety with
\begin{equation}
\label{eq:FanoindexTr}
K_{{\nT}_r}= -\frac{r(n+1)}{2} M_{{\nT}_r},
\end{equation}
where $M_{{\nT}_r}$ is the pull-back of $\sO_{\nS_r}(1)$.
\end{enumerate}
\end{prop}
\begin{proof} 
(1) The Picard number of $\nU_r$ is two
since $\nU_r$ is a projective bundle over 
$\mathrm{G}(n-\frac r2 +1,V)$. Therefore 
the Picard number of $\nT_r$ is one since the relative Picard number
of $\Lpi_{\nU_r}\colon\nU_r\to\nT_r$ is one. 
$\Lpi_{\nU_r}$ is a Mori fiber
space since a general fiber of $\Lpi_{\nU_r}$
is a Fano variety by Proposition \ref{Z_Q}. 
$\nT_r$ is $\mQ$-factorial by \cite[Lemma 5-1-5]{KMM}.

\noindent(2)
To show the claim (2), we will construct the following commutative diagram:
\begin{equation}
\label{eq:STcomm}
\begin{matrix}
\xymatrix{\widetilde{\nU}_r\ar[r]^{\Lpi_{\widetilde{\nU}_r}}\ar[d]_{\Lp\,_{\widetilde{\nU}_r}} & \widetilde{\nT}_r\ar[r]^{\Lrho\,_{\widetilde{\nT}_r}}\ar[d]_{\Lp\,_{\widetilde{\nT}_r}} & \widetilde{\nS}_r\ar[d]_{\Lp\,_{\widetilde{\nS}_r}}\\
\nU_r\ar[r]_{\Lpi_{\nU_r}} & \nT_r\ar[r]_{\Lrho\,_{\nT_r}} & \nS_r.}
\end{matrix}
\end{equation}

\noindent $\bullet$ 
$\widetilde{\nU}_r$ is defined in $\mathrm{G}(\frac r2,\frQ)\times_{\rG(n+1-r,V)}\mP(\ft{S}^{2}\frQ^{*})$, in a similar way to $\nU_r$, by 
\[
\widetilde{\nU}_r:=\{([\Pi],[Q];[V_{n+1-r}])\mid\mP(\Pi)\subset Q\subset \mP(V/V_{n+1-r})\}. 
\]
\noindent $\bullet$ Then the projection to the second factor yields a morphism 
$\widetilde{\nU}_r\to \widetilde{\nS}_r$ and
the morphism 
$\mathrm{G}(\frac r2,\frQ)\times_{\rG(n+1-r,V)}\mP(\ft{S}^{2}\frQ^{*})
\to
\mathrm{G}(n+1-\frac r2,V)\times \mP(\ft{S}^{2}V^{*})$
induces a morphism 
$\Lp\,_{\widetilde{\nU}_r}\colon \widetilde{\nU}_r\to \nU_r$.
It is easy to see that $\Lp\,_{\widetilde{\nU}_r}$ is a birational morphism.

\noindent $\bullet$ Let
\[
\xymatrix{ & \widetilde{\nU}_r\;\ar[r]^{\;\;\Lpi_{\widetilde{\nU}_r}\;\;} & \;\widetilde{\nT}_r\;\ar[r]^{\;\;\Lrho_{\widetilde{\nT}_r}\;\;} & \;\widetilde{\nS}_r}
\]
be the Stein factorization of $\widetilde{\nU}_r\to\widetilde{\nS}_r$.
By the definition of Stein factorization,
we have
${\Lpi_{\widetilde{\nU}_r}}_*\sO_{\widetilde{\nU}_r}=\sO_{\widetilde{\nT}_r}$
and
${\Lpi_{{\nU}_r}}_*\sO_{{\nU}_r}=\sO_{{\nT}_r}$.
Therefore, by
\begin{equation}
\label{eq:long}
{\Lp_{\widetilde{\nS}_r}}_*{\Lrho\,_{\widetilde{\nT}_r}}_*\sO_{\widetilde{\nT}_r}=
{\Lp_{\widetilde{\nS}_r}}_*{\Lrho\,_{\widetilde{\nT}_r}}_*{\Lpi_{\widetilde{\nU}_r}}_*\sO_{\widetilde{\nU}_r}=
{\Lrho\,_{{\nT}_r}}_*{\Lpi_{{\nU}_r}}_*{\Lp\,_{\widetilde{\nU}_r}}_*\sO_{\widetilde{\nU}_r}={\Lrho\,_{{\nT}_r}}_*\sO_{{\nT}_r},
\end{equation}
we see that the Stein factorization of $\Lp\,_{\widetilde{S}_r}\circ
\Lrho\,_{\widetilde{T}_r}$ is $\widetilde{\nT}_r\to \nT_r\to \nS_r$.
We denote by $\Lp\,_{\widetilde{\nT}_r}\colon \widetilde{\nT}_r\to \nT_r$ the induced morphism.

Now we have completed the diagram (\ref{eq:STcomm}).
Similarly to the proof of Proposition \ref{cla:double},
we see that the branch locus of $\Lrho\,_{\widetilde{T}_r}\colon \widetilde{\nT}_r\to \widetilde{\nS}_r$ is $\Lp\,_{\widetilde{\nS}_r}$-exceptional divisor $E_r$.
Since $\widetilde{\nU}_r\to \widetilde{\nT}_r$ is a Mori fiber space,
$\widetilde{\nT}_r$ has only rational singularities by \cite{Fuj} and 
in particular is Cohen-Macaulay.
Therefore $\Lrho\,_{\widetilde{\nT}_r}$ is flat.
First we treat the case where $r=n+1$.
Then $\nT_{n+1}$ is Gorenstein since it is the double cover of $\nS_{n+1}=\mP({\ft S}^2 V^*)$ branched along the divisor $\nS_{n}$.
Thus $\nT_{n+1}$ has only canonical singularities by \cite{Fuj}.
$\Sing \nT_{n+1}$ is contained in the inverse image of $\Sing \nS_n=\nS_{n-1}$. Now we have verified the assertion (2) in case $r=n+1$.
{Let us} assume that $r\leq n$. 
Then, by (\ref{eq:Er}),
it holds that 
\begin{equation}
\label{eq:redEr}
\Lrho\,_{\widetilde{\nT_r}}^*(\frac r2 M_{\widetilde{\nS}_r}-L_{\widetilde{\nS}_r})\sim (\Lrho\,_{\widetilde{\nT_r}}^*E_r)_{\mathrm{red}}
\end{equation}
and
${\Lrho\,_{\widetilde{\nT}_r}}_*\sO_{{\widetilde{\nT}_r}}
=\sO_{{\widetilde{\nS}_r}}\oplus \sO_{{\widetilde{\nS}_r}}(-\frac r2 M_{\widetilde{\nS}_r}+L_{\widetilde{\nS}_r}).$
By (\ref{eq:long}),
we see that 
${\Lrho\,_{{\nT}_r}}_*\sO_{{\nT}_r}
=\sO_{{\nS}_r}\oplus \sO_{{\nS}_r}(-\frac r2 M_{{\nS}_r}+L_{{\nS}_r})$
with $L_{{\nS}_r}:=\Lpi_{\widetilde{\nS}_r *} L_{\widetilde{\nS}_r}$
and
\begin{equation}
\label{eq:specT}
{\nT}_r=\Spec_{{\nS}_r}
\big(\sO_{{\nS}_r}\oplus \sO_{{\nS}_r}(-\frac r2 M_{{\nS}_r}+L_{{\nS}_r})\big).
\end{equation}
Pushing (\ref{eq:redEr}) forward by $p_{\widetilde{\nT}_r}$,
we obtain
\begin{equation}
\label{eq:triv}
\Lrho\,_{{\nT_r}}^*(\frac r2 M_{{\nS}_r}-L_{{\nS}_r})\sim 0.
\end{equation}
In particular,
$\Lrho\,_{{\nT}_r}^* L_{{\nS}_r}$ is Cartier 
since so is $M_{{\nS}_r}$.
Therefore $K_{\nT_r}$ is Cartier by (\ref{eq:adjSr1})
and the formula $K_{\nT_r}=\Lrho\,_{\nT_r}^*K_{\nS_r}$.
Namely, $\nT_r$ is Gorenstein. 
To show that $\nT_r$ has only canonical singularities,
let $f\colon \widetilde{\nR}_r\to \widetilde{\nT}_r$ be a resolution.
Then, by the ramification formula, we have $K_{\widetilde{\nR}_r}\geq
f^*\Lrho\,_{\widetilde{\nT}_r}^* K_{\widetilde{\nS}_r}$.
Since $\nS_r$ has only canonical singularities, we have
$K_{\widetilde{\nS}_r}\geq \Lp\,_{\widetilde{\nS}_r}^* K_{{\nS}_r}$.
Therefore
\[
K_{\widetilde{\nR}_r}\geq
f^*\Lrho\,_{\widetilde{\nT}_r}^* K_{\widetilde{\nS}_r}\geq
f^*\Lrho\,_{\widetilde{\nT}_r}^* \Lp\,_{\widetilde{\nS}_r}^*K_{{\nS}_r}=
f^* \Lp\,_{\widetilde{\nT}_r}^*\Lrho\,_{{\nT}_r}^*K_{{\nS}_r}
=f^* \Lp\,_{\widetilde{\nT}_r}^*K_{{\nT}_r}.
\]
This means that $\nT_r$ has only canonical singularities.

By (\ref{equation:SingSr}) and (\ref{eq:specT}),
we see that  
$\nT_r$ is smooth at the inverse image of
a rank $r-1$ point $s \in \nS_r$
since $L_{\nS_r}$ generates the divisor class group at $s$
and then $(\ref{eq:specT})$ coincides with punctured universal cover
near $s$.

(3) 
If $r=n+1$, then
the canonical divisor of $\nT_r$ is given by \[
-\binom{n+2}{2} M_{\nT_r}+\frac {n+1}2 M_{\nT_r}=
-\frac{(n+1)^2}{2} M_{\nT_r}
\]
since the degree of the branch locus $\nS_n$ is $n+1$.
If $r\leq n$, then
the assertion follows from $K_{\nT_r}=\Lrho\,_{\nT_r}^*K_{\nS_r}$,
(\ref{eq:adjSr1}) and (\ref{eq:triv}).
\end{proof}

\begin{rem}
It is useful to consider that 
$\widetilde{\nT}_r\to\widetilde{\nS}_r$ as in the diagram (\ref{eq:STcomm})
is the family over $\rG(n+1-r,V)$ of the double cover
$\nT_r\to \nS_r$ for $r$-dimensional vector spaces $V/V_{n+1-r}$ with
$[V_{n+1-r}]\in \rG(n+1-r,V)$.
\end{rem}

\subsection{Dual situations and orthogonal linear sections}

To consider
projective duality for the symmetric determinantal loci
in $\mP({\ft S}^2 V^*)$, 
the symmetric determinantal loci in 
$\mP({\ft S}^2 V)$ naturally appear.
{Recall that we denote by $\nS^*_r$ the symmetric determinantal 
locus of rank at most $r$ in $\mP({\ft S}^2 V)$. Similarly to $\nS_r$, 
$\nS^*_1$ is the second Veronese variety $v_2(\mP(V))$ and
$\nS^*_r$ is the $r$-secant variety of $\nS^*_1$. Corresponding to our 
definitions $\nU_r, \nT_r$ and $\widetilde{\nS_r}$ for $\nS_r$ in 
$\mP(\ft{S}^2V^*)$, we have similar definitions  $\nU_r^*, \nT_r^*$ and 
$\widetilde{\nS_r^*}$ for $\nS_r^*$ in 
$\mP(\ft{S}^2V)$.}
%
%

For a linear subspace $L_{k+1}\subset {\ft S}^2 V^*$ of dimension $k+1$,
we say that $\nS_r\cap \mP(L_{k+1})$ is a {\it linear section} of
$\nS_r$ if $\nS_r\cap \mP(L_{k+1})$ is of codimension 
$\dim {\ft S}^2 V^*-(k+1)$ in $\nS_r$. {Linear sections of $\nS^*_r$ is 
defined for linear subspaces in  ${\ft S}^2 V$ in a 
similar way. }

Let $L_{k+1}^{\perp} \subset {\ft S}^2 V$ be the linear subspace 
orthogonal to $L_{k+1}$ with respect to the dual pairing.
For a triple $(\nS_r,\nS_s^*,L_{k+1})$, we say  
that linear sections
$\nS_r\cap \mP(L_{k+1})$ and $\nS^*_s\cap \mP(L_{k+1}^{\perp})$
are mutually {\it orthogonal}. 
%
%
By slight abuse of terminology, we also call the pull-back
of a linear section of $\nS_r$ by the double cover $\nT_r\to \nS_r$
{\it a linear section of $\nT_r$}.

\section{{\bf Pairs of Calabi-Yau sections and plausible duality}}
In this paper, we adopt the following definition of Calabi-Yau variety
and also Calabi-Yau manifold.

\begin{defn}
A normal projective variety $X$ is called \textit{a Calabi-Yau variety}
if $X$ has only Gorenstein canonical singularities, and its canonical
divisor is trivial and $h^{i}(\sO_{X})=0$ for $0<i<\dim X$.
If $X$ is smooth, then $X$ is called \textit{a Calabi-Yau manifold}.
A smooth Calabi-Yau threefold is abbreviated as a Calabi-Yau threefold.
\end{defn}

\subsection{Calabi-Yau linear section of $\nS_r$}
\begin{prop}
\label{prop:CYSr}
Assume that 
$n-r$ is even and $r<n+1$.
Then a general linear section $\nS_r^{\textsc{CY}}$
of codimension $\frac{r(n+1)}{2}$
is a Calabi-Yau variety 
of dimension $\frac{r(n+2-r)}{2}-1$
with only terminal $($resp.~canonical\,$)$ singularities if $r<n$ $($resp.~$r=n)$. 
Moreover, a general $\nS_r^{\textsc{CY}}$ is smooth
if and only if
$r\leq 2$.
\end{prop}

\begin{proof}
$\nS_r^{\textsc{CY}}$
has trivial canonical divisor
by (\ref{eq:Fanoindex}) since $K_{\nS_r}$ is Cartier
in case $n-r$ is even.
Since $\nS_r$ has only terminal (resp.~canonical) singularities
in case $r<n$ (resp.~$r=n$) and
is a Fano variety as we saw in
the subsection \ref{subsection:Spr},
it holds that $h^i(\sO_{\nS_r})=0$ for any $i>0$ 
and $h^i(\sO_{\nS_r}(-jM_{\nS_r}))=0$ for any $i<\dim \nS_r$ and
$j>0$
by the Kodaira-Kawamata-Viehweg vanishing theorem.
Therefore we have 
$h^i(\sO_{\nS_r^{\textsc{CY}}})=0$ for any $0<i<\dim \nS_r^{\textsc{CY}}$
by the Koszul complex.
By a version
of the Bertini theorem (cf.~\cite[Prop.~0.8]{And}),
a general 
$\nS_r^{\textsc{CY}}$ has only 
terminal (resp.~canonical) singularities
in case $r<n$ (resp.~$r=n$).
Therefore a general 
$\nS_r^{\textsc{CY}}$ is a Calabi-Yau variety.

Since $r<n+1$,
$\Sing \nS_r=\nS_{r-1}$.
Thus the second assertion is equivalent to
that
$\dim \nS_{r-1}=\frac{r(r-1)}{2}-1+(r-1)(n+2-r)<\frac{r(n+1)}{2}$ holds 
if and only if $r\leq 2$.
A proof of this claim is elementary.
\end{proof}

\begin{rem}
\label{rem:CYSr}
In case $n-r$ is odd,
we can show the following by the same argument as
in the proof of Proposition \ref{prop:CYSr}:

Linear sections of $\nS_r$ of codimension $\frac {r(n+1)}{2}$ 
does not have trivial canonical divisors but
bi-canonical divisors are trivial.
Except this, the same properties as $\nS_r^{\textsc{CY}}$
hold for them.
\end{rem}

By {the above} proposition,
we observe that
\begin{equation}
\label{eq:ob}
\dim
\nS_r^{\textsc{CY}}=
\dim \nS_{n+2-r}^{\textsc{CY}}=
\dim \nS_{n+2-r}^{* \textsc{CY}}.
\end{equation}
This indicates certain duality between 
$\nS_r$ and $\nS_{n+2-r}^*$.
We will discuss this duality in Subsection \ref{subsection:Pl}.

If $r=1$,
then $\nS_1$ is isomorphic to
the second Veronese variety $v_2(\mP(V))$.
Therefore its linear sections are
complete intersections of quadrics in $\mP(V)$.
 
In the next subsection, 
we adopt the dual setting and consider $\nS^*_2$ and
its linear sections $\nS_2^{*\textsc{CY}}$
in detail.

\subsection{Rank two case and Calabi-Yau manifold $X$ of a Reye congruence}

\label{subsection:ranktwo}

{Consider the determinant locus $\nS_2^*$ in $\mP(\ft{S}^2V)$ and also 
$\nU_2^*,\nT_2^*,\widetilde{\nS_2^*}$ defined in the same way as $\nU_2,\nT_2,
\widetilde{\nS_2}$ for $\nS_2$ in $\mP(\ft{S}^2V^*)$. Note that  $\nU^*_2\simeq \nT^*_2$ holds in this case.}
%
%

Let us {write} the exact sequence (\ref{eq:sE0}) for $\nS^*_2$
{by} noting that $\rG(n,V^*)=\mP(V)$ and $\eS=\Omega^1_{\mP(V)}$:
\begin{equation}
\label{eq:Er=2}
0\to\sE^{*}\to\ft{S}^{2}V\otimes\sO_{\mP(V)}\to\ft{S}^{2}T_{\mP(V)}(-1)\to0.
\end{equation}

\begin{prop}
\label{prop:sEtriv}
$\sE\simeq V^*\otimes \sO_{\mP(V)}(1)$.
\end{prop}

\begin{proof}
Taking fibers of (\ref{eq:Er=2}) at a point $[V_1]\in \mP(V)$,
we obtain the exact sequence
$0\to V\otimes V_1\to {\ft S}^2 V\to {\ft S}^2 (V/V_1)\to 0$.
Therefore the fiber of $\sE^*$ at $[V_1]$ is $V\otimes V_1$,
which show the claim.
\end{proof}
Therefore it holds that
\[
\nT_2^*\simeq \nU^*_2:=\mP(\sE^*)\simeq \mP(V)\times \mP(V).
\]
Moreover, by the proof of Proposition \ref{prop:sEtriv},
we see that
the map $\nT^*_2\to \mP({\ft S}^2 V)$
is given by
$\mP(V)\times \mP(V)\ni ([{\bf v}],[{\bf w}]) \mapsto
[{\bf v}\otimes {\bf w}+{\bf w}\otimes {\bf v}]\in \mP({\ft S}^2 V)$.
Therefore $\nS^*_2$, which is the image of this map, is nothing but
the symmetric product
${\ft S}^2 \mP(V)$.
In \cite{Arxiv}, we show that,
by identifying ${\ft S}^2 \mP(V)$ with the Chow variety of 
degree two $0$-cycles in $\mP(V)$ (cf.~\cite{GKZ}),
$\widetilde{\nS^*_2}$ is isomorphic to the Hilbert scheme of 
length two subschemes in $\mP(V)$, and 
the Springer resolution $\widetilde{\nS^*_2}\to \nS^*_2$ coincides with
the Hilbert-Chow morphism. 

{For brevity of notation, we fix the following definitions 
in what follows:}
\[
\chow:=\nS^*_2 
\;\text{ and }\;
X:=\text{a codimension $n+1$ linear section of } \nS_2^*.
\]
In \cite{Ol} (see also \cite{Arxiv}), 
a general $X$ is called a
\textit{Reye congruence} since it is isomorphic to 
a $(n-1)$-dimensional subvariety of $\rG(2,V)$.
By Proposition \ref{prop:CYSr} and Remark \ref{rem:CYSr}, 
Reye congruence $X$ is a Calabi-Yau variety when 
$n$ is even; when $n$ is odd, $X$ has similar properties except that 
$2K_X\sim 0$. In particular, when  $n=3$, $X$ is an 
Enriques surface (see \cite{Co}).

The proof of the following proposition
is standard, so we omit it here (cf.~\cite{Arxiv}).

\begin{prop}
For a general $X$, it holds that
\[
\pi_{1}(X)\simeq\mZ_{2},\;\;\Pic X\simeq\mZ\oplus\mZ_{2},\]
where the free part of $\Pic X$ is generated by the class $D$ of
a hyperplane section of $\chow$ restricted to $X$. 
\end{prop}

When $n=4$, $X$ is a Calabi-Yau threefold with the following invariants
\cite[Proposition 2.1]{HoTa1}: \[
\deg(X)=35,\;\; c_{2}.D=50,\;\; h^{2,1}(X)=26,\; h^{1,1}(X)=1,\]
 where $c_{2}$ is the second Chern class of $X$.

\subsection{Calabi-Yau linear section of $\nT_r$}
In this subsection, we assume that $r$ is even.

\begin{prop}
\label{prop:CYTr}
A general linear section $\nT_r^{\textsc{CY}}$
of codimension $\frac{r(n+1)}{2}$
is a Calabi-Yau variety 
{of dimension $\frac{r(n+2-r)}{2}-1$
with only canonical singularities.} 
Moreover, a general $\nT_r^{\textsc{CY}}$ is smooth
if 
$r\leq 4$.
\end{prop}

\begin{proof}
By (\ref{eq:FanoindexTr}),
$\nT_r^{\textsc{CY}}$
has trivial canonical divisor.
Since $\nT_r$ is a Fano variety with only canonical singularities 
by Proposition \ref{cla:ZY},
We can show that   
$h^i(\sO_{\nT_r^{\textsc{CY}}})=0$ for any $0<i<\dim \nT_r^{\textsc{CY}}$,
and a general 
$\nT_r^{\textsc{CY}}$ has only canonical singularities
{in} the same way as in the proof of Proposition \ref{prop:CYSr}.
Therefore a general 
$\nT_r^{\textsc{CY}}$ is a Calabi-Yau variety.

Since 
$\Sing \nT_r$ is contained in the inverse image of $\nS_{r-2}$
by Proposition \ref{cla:ZY} (2),
the second assertion follows once we show
that
$\dim \nS_{r-2}=\frac{(r-1)(r-2)}{2}-1+(r-2)(n+3-r)<\frac{r(n+1)}{2}$ 
{holds} if and only if
$r\leq 4$.
A proof of the latter is elementary.
\end{proof}

We have already studied $\nT_2^{\textsc{CY}}$ in the subsection
\ref{subsection:ranktwo}.
We deal with $\nT_4^{\textsc{CY}}$ in detail in the subsection \ref{subsection:CY3Y}.

\subsection{Rank four case and Calabi-Yau manifold $Y$}

\label{subsection:CY3Y}
For brevity of notation, we introduce the following definitions:   
\[
\Hes:=\nS_4,\quad \UU:=\widetilde{\nS}_4,\quad \hcoY:=\nT_4, 
\quad \Zpq:=\nU_4,
\]
{while retaining} the notation $\nS_1,\nS_2,\nS_3\subset \Hes$. 
{We denote by $\Zpq_{[Q]}$ the fiber of the morphism $\Zpq\to \Hes$ over 
a point $[Q]$. Recall that $\Lpi_{\nU_4}=\Lpi_\Zpq:\Zpq \to \hcoY$ is defined 
by the Stein factorization $\Zpq\to\hcoY\to\Hes$ of $\Zpq\to\Hes$.}

\def\FigQuad{\resizebox{11cm}{!}{\includegraphics{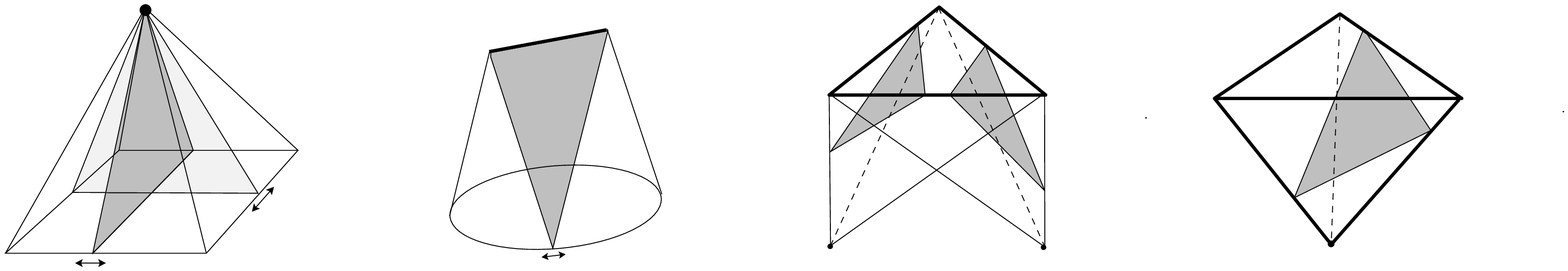}}} 
\def\FigQuadDisplay{ 
\begin{xy} 
(0,0)*{\FigQuad}, 
(-45,12)*{\mP(V_{n-3})}, 
(-16,12)*{\mP(V_{n-2})}, 
( 11,12)*{\mP(V_{n-1})}, 
( 40,12)*{\mP(V_n)}, 
(-45,-12)*{\mP^1 \sqcup \mP^1}, 
(-16,-12)*{\mP^1}, 
( 11,-12)*{\mP^{n-1}\sqcup_{1pt}\mP^{n-1}}, 
( 40,-12)*{2\,\mP^{n-1}} \end{xy} }

\[
\FigQuadDisplay\]
 \vspace{0.2cm}
 \begin{fcaption} 
\item \textbf{Fig.1. Quadrics $Q$ of rank at most four in $\mP(V)$ and families of $(n-2)$-planes
therein.} The singular loci of $Q$ are written by $\mP(V_{k})$ with
$k=n+1-{\rm {rk}\, Q}$. Also the parameter spaces of the planes in
each $Q$ are shown ($\mP^{n-1}\sqcup_{1pt}\mP^{n-1}$ represents the
union of $\mP^{n-1}$'s intersecting at one point). See also Fig.2 in
the subsection \ref{sub:tildeY-Y}. \end{fcaption} 
\vspace{0.0cm}

\begin{prop}
 \label{cla:double2} If $\rank Q=4$, then $\Zpq_{[Q]}$ is a disjoint union of two smooth rational curves, {each of which} 
is identified
with a conic in $\mathrm{G}(n-1,V)$. If $\rank Q=3$, then $\Zpq_{[Q]}$
is a smooth rational curve, which is also identified with a conic
in $\mathrm{G}(n-1,V)$. If $\rank Q=2$, then $\Zpq_{[Q]}$ is the
union of two $\mP^{n-1}$'s intersecting at one point. If $\rank Q=1$,
then $\Zpq_{[Q]}$ is a $($non-reduced\,$)$ $\mP^{n-1}$. 
In particular, $\Lpi_{\Zpq}\colon \Zpq\to \hcoY$ is generically a conic bundle.
\end{prop}
\begin{proof} If $\rank Q=4$, the fiber $\Zpq_{[Q]}$ consists of
two disconnected components, and is isomorphic to the orthogonal Grassmannian
$\OG(2,4)$ by Proposition \ref{Z_Q}. To be more explicit, let $\mP(V_{n-3})\subset\mP(V)$
be the vertex of $Q$. Then the quadric $Q$ is the cone over $\mP^{1}\times\mP^{1}$
with the vertex $\mP(V_{n-3})$. There are two distinct $\mP^{1}$-families
of lines in $\mP^{1}\times\mP^{1}$. Each of the families can be understood
as the corresponding conic in $\mathrm{G}(2,V/V_{n-3})$, which gives
one of the connected components of $\OG(2,4)$. Under the natural
map $\mathrm{G}(2,V/V_{n-3})\rightarrow\mathrm{G}(n-1,V)$, we have two
$\mP^{1}$- families of $2$-planes in $Q$ parameterized by the conics
in $\mathrm{G}(n-1,V)$.

If $\rank Q=3$, the vertex of the quadric $Q$ is a $\mP(V_{n-2})\subset\mP(V)$.
The quadric $Q$ is the cone over a conic with the vertex $\mP(V_{n-2})$.
The conic is contained in $\mP(V/V_{n-2})=\mathrm{G}(1,V/V_{n-2})$, and
can be identified with a conic in $\mathrm{G}(n-1,V)$ under the natural
map $\mathrm{G}(1,V/V_{n-2})\rightarrow\mathrm{G}(n-1,V)$.

If $\rank Q=2$, then the quadric $Q$ has a vertex $\mP(V_{n-1})\subset\mP(V)$
and is the union of two $(n-1)$-planes intersecting along the $(n-2)$-plane
$\mP(V_{n-1})$. Hence $\Zpq_{[Q]}\subset\mathrm{G}(n-1,V)$ is given by
the union of the corresponding $\mP^{n-1}$'s, i.e., $\mathrm{G}(n-1,n)$'s
in $\mathrm{G}(n-1,V)$, which intersect at one point $\mP(V_{n-1})$.

If $\rank Q=1$, then $Q$ is a double $(n-1)$-plane. Thus $\Zpq_{[Q]}$
is a (non-reduced) $\mP^{n-1}\cong\mathrm{G}(n-1,n)$. \end{proof}

We write by $G_{\hcoY}^{1}$ (resp.~$G_{\hcoY}^{2}$, $G_{\hcoY}$) the inverse
image under $\Lrho\,_{\hcoY}$ of $\nS_1$ 
(resp.~$\nS_2\setminus \nS_1$, $\nS_2$). We note that $G_{\hcoY}\simeq \nS_1\simeq \ft{S}^{2}\mP(V^{*})$ and $G_{\hcoY}^{1}\simeq \nS_2 \simeq v_{2}(\mP(V^{*}))$
since
$\nS_2$ is contained in the branch locus of $\Lrho\,_{\hcoY}$. Using these,
we summarize our construction above in the following diagram:
\def\xyHYZG{ \begin{matrix} 
\begin{xy} 
(35,0)*+{\Zpq}="Z",  
(65,0)*+{\mathrm{G}(n-1,V)}="G", 
(35,-13)*+{\hcoY}="Y", 
(35,-26)*+{\ \Hes,}="H", 
(18,-13)*+{G_\hcoY^1\;\;\subset \;\;\;G_\hcoY \;\subset\;\;},  
(15,-20)*+{\vcorr{-90}{$\simeq$} \;\qquad\;\; \vcorr{-90}{$\simeq$}}, 
(15,-26)*+{v_2(\mathbb{P}(V^*))\subset \ft{S}^2\mathbb{P}(V^*) \;\subset\ \ \;},
\ar^{\Lpi_\Zpq} "Z";"Y" 
\ar^{\Lrho_\hcoY} "Y";"H" 
\ar_{\Lrho_\Zpq\;\;}^{\;_{\text{proj.~bundle}}\;\;\;\;\;\quad} "Z";"G" 
\end{xy} \end{matrix} }
\begin{equation}
\xyHYZG\label{eq:Z}\end{equation}
where $\Lpi_{\Zpq}$ is a $\mP^{1}$-fibration over $\hcoY\setminus G_{\hcoY}$
by Proposition \ref{cla:double2}. In Section \ref{section:BirY},
we will construct a nice desingularization $\widetilde{\hcoY}$ of
$\hcoY$. Also, in Sections \ref{section:BirY} and \ref{section:FY}, 
we will study the geometry of $\widetilde{\hcoY}\to\hcoY$
along the loci $G_{\hcoY}$ and $G_{\hcoY}^{1}$ in full detail. 

Now {consider the linear section of $\hcoY=T_4$ and we set}
\[
Y:=\nT_4^{\textsc{CY}}.
\]
By Proposition \ref{prop:CYTr},
a general $Y$ is a Calabi-Yau manifold of dimension $2n-5$.

{
By using the fibration $\L\pi_{\Zpq}\colon \Zpq\to \hcoY$,  
it is possible to compute several invariants of $Y$.   
Computations have been done 
for the case $n=4$ in \cite[Prop.3.11 and Prop.3.12]{HoTa1}, \cite{Arxiv}, 
which we summarize below:}

\begin{prop}
\label{prop:Y} A general $Y$ is a simply connected smooth Calabi-Yau
$3$-fold such that $\Pic Y=\mZ[M]$, ${M}^{3}=10$, $c_{2}(Y).{M}=40$
and $e(Y)=-50$. In particular, $h^{1,1}(Y)=1$ and $h^{1,2}(Y)=26$. 
\end{prop}

{ It should be noted here that the Spec construction 
(\ref{eq:specT}) of $\nT_4=\hcoY$ generalizes the covering  
constructed in \cite[eq.(3.4)]{HoTa1} for $n=4$.}

\hspace{5pt}

In the following two subsections, we discuss
two plausible dualities between
$\nS^*_a$ and $\nT_b$ for certain pairs of $a$ and $b$.

\subsection{Linear duality and beyond}
The exact sequence (\ref{eq:sE0}) means that
the fibers of ${\ft S}^2 \eS$ and $\sE^*$ over a point of 
$\mathrm{G}(n+1-\frac r2,V)$
are the orthogonal spaces to each other when we consider them
as subspaces in ${\ft S}^2 V$ and ${\ft S}^2 V^*$, respectively.
The pair ${\ft S}^2 \eS$ and $\sE^*$ is an example of 
{\it orthogonal bundles}.

In \cite[\S 8]{HPD1}, Kuznetsov has established 
the homological projective duality between
a projective bundle $\mP(\sV)$ over a smooth base $S$
and its orthogonal bundle $\mP(\sV^{\perp})$ for a 
globally generated vector bundle $\sV$ on $S$.
He has called this duality {\it linear duality} in \cite{ICM}.
{Due to this general result, we know that} $\mP({\ft S}^2 \eS)$ 
and $\mP(\sE^*)$ are homological projective dual. Note that $\mP({\ft S}^2 \eS)=\widetilde{\nS^*}_{n+1-\frac r2}$ and $\mP(\sE^*)=\nU_r$.
Mutually orthogonal linear sections $X$ and $Z$
of $\mP({\ft S}^2 \eS)$ and $\mP(\sE^*)$
of codimensions $\rank {\ft S}^2 \eS$ and $\rank \sE^*$ respectively
{have the equal dimensions,} 
$\dim \rG(n+1-\frac r2,V)-1=\frac r2 (n+1-\frac r2)-1$, 
and are derived equivalent by \cite[\S 8]{HPD1}. 
Let $Y$ be the double cover of the image of $Z$ on $\mP({\ft S}^2 V^*)$.
{The derived equivalence between $X$ and $Z$} 
indicates that {there is some} relationship between 
non-commutative resolutions of $\sD^b(X)$ and $\sD^b(Y)$.
{Indeed, in \cite{ReyeEnr},
we have shown that this is the case when $n=3$ and $r=4$.} 
Note that in this case,
a general $X$ is a so-called Enriques-Fano threefold and 
a general $Y$ is a del Pezzo surface of degree two [ibid.].
{
In this case (of $n=3$ and $r=4$),  
we can also investigate the derived categories of  mutually orthogonal 
linear sections of $\nS^*_2$ and $\nT_4$ for a triple $(\nS_4,\nS_2^*,L_4)$, 
which define, respectively, an Enriques surface of Reye congruence and 
Artin Mumford double solid.  In \cite{DerSym}, we have found natural 
Lefschetz collections, which indicates that certain non-commutative 
resolutions of $\nS^*_2$ and $\nT_4$  are homological projective dual 
to each other. One may suspect that, with finding 
suitable Lefschetz collections, non-commutative resolutions of 
$\nS^*_{n+1-\frac r2}$ and $\nT_r$ are homologically projective dual 
to each other in general.}

\subsection{Plausible duality}
\label{subsection:Pl}
Assume that $r$ is even. Then $n-(n+2-r)$ is also even.
Therefore we obtain mutually orthogonal
Calabi-Yau linear sections $\nS^{* \textsc{CY}}_{n+2-r}$ and $\nT^{\textsc{CY}}_r$ by Propositions \ref{prop:CYSr} and \ref{prop:CYTr}.

We suspect an equivalence of 
the derived categories of 
certain non-commutative resolutions of 
orthogonal linear sections
$\nS^{* \textsc{CY}}_{n+2-r}$ and $\nT^{\textsc{CY}}_r$ 
rather than
$\nS^{* \textsc{CY}}_{n+2-r}$ and $\nS^{\textsc{CY}}_r$.
More generally, we speculate
that 
certain non-commutative resolutions of 
$\nS^*_{n+2-r}$ and $\nT_r$ with suitable Lefschetz collections {for each}
are homologically projective dual.
In fact, this is established in case $r=n+1$ \cite{Quad}
({called} Veronese-Clifford duality).
Note that in case $n=r=4$,
both {$\nS^{* \textsc{CY}}_{2}=X$ and $\nT^{\textsc{CY}}_4=Y$}
are smooth, and hence they are of considerable interest.
In \cite{DerSym}, we {have} constructed
(dual) Lefschetz collections in the derived categories 
of $\widetilde{\nS}^*_{2}$ and $\nT_4$, and 
{have proved} the derived equivalence between
$\nS^{* \textsc{CY}}_{2}$ and $\nT^{\textsc{CY}}_4$ in \cite{HoTa3}
using the properties of these collections.

{Having these applications in mind, in the rest of this paper, 
we study the birational geometry of $\hcoY=\nT_4$ for general $n$. 
Since we will be concentrated on the case $r=4$, we will extensively 
use the notation introduced in the beginning of 
the subsection \ref{subsection:CY3Y}. }

\section{{\bf Birational geometry of $\hcoY$}}
\label{section:BirY}

Proposition \ref{cla:double2} indicates {a correspondence}
between points in $\hcoY$ and conics in $\rG(n-1,V)$.
In this section, we explicitly construct a birational map between $\hcoY$ 
and the Hilbert scheme $\hcoY_0$
of conics in $\rG(n-1,V)$.

\subsection{Conics and planes in $\mathrm{G}(n-1,V)$\label{sub:Conics-and-planes}}

Let $q$ be a conic in $\rG(n-1,V)$ and $\mP_q$ the plane spanned by $q$.
Noting that $\rG(n-1, V)$ is the intersection of the
Pl\"ucker quadrics in $\mP(\wedge^{n-1} V)$,
we see that either 
$\mP_q\subset \rG(n-1,V)$ or 
$\rG(n-1,V)\cap \mP_q=q$ holds for $\mP_q$.

When $\mP_q\subset \rG(n-1,V)$,
we note that there are exactly two types of planes contained in
$\rG(n-1,V)\subset \mP(\wedge^{n-1} V)$:\begin{equation}
\begin{alignedat}{2}{\rm P}_{V_{n-2}}:= & \{[\Pi]\in\mathrm{G}(n-1,V)\mid V_{n-2}\subset\Pi\}\cong\mP^{2} & (\rho\text{-plane}),\\
{\rm P}_{V_{n-3}V_{n}}:= & \{[\Pi]\in\mathrm{G}(n-1,V)\mid V_{n-3}\subset\Pi\subset V_{n}\}\cong\mP^{2} & (\sigma\text{-plane})\end{alignedat}
\label{eq:xxxxx}\end{equation}
with some $V_{n-2}\subset V$ and $V_{n-3}\subset V_{n}\subset V$, respectively.
As displayed above, we call these planes \textit{$\rho$-plane} and
\textit{$\sigma$-plane}, respectively. 
It is easy to deduce the following proposition:
%
%
\begin{prop}
\label{prop:barPrho-barPsigma} In $\rG(3,\wedge^{n-1}V)$,
the set of $\rho$-planes $\overline{\Prt}_{\rho}$
and 
the set of $\sigma$-planes $\overline{\Prt}_{\sigma}$
are given by
\[
\begin{aligned}
\overline{\Prt}_{\rho} & =  \left\{ \big[\left(V/V_{n-2}\right)\wedge\left(\wedge^{n-2}V_{n-2}\right)\big]\mid[V_{n-2}]\in\rG(n-2,V)\right\}\\
\overline{\Prt}_{\sigma} & =  \left\{ \;\big[\wedge^{2}\left(V_{n}/V_{n-3}\right)\wedge (\wedge^{n-3}V_{n-3})\big]\;\mid[V_{n-3}\subset V_{n}]\in \rF(n-3,n,V)\right\},
\end{aligned}
\]
where $\overline{\Prt}_{\rho}\simeq  \rG(n-2,V)$ and $\overline{\Prt}_{\sigma}
\simeq \rF(n-3,n,V)$. 
 \end{prop}

Let us make the following definition:
\begin{defn}
We call a conic $q$ in $\rG(n-1,V)$ a {\it $\tau$-conic} if 
$\mP_{q}\cap \rG(n-1,V)=q$.
A conic $q$ is called a {\it $\rho$-conic} and 
{\it $\sigma$-conic} 
if the plane $\mP_{q}$ is contained in $\rG(n-1,V)$, {and in that case}
$\mP_q$ is {called} a $\rho$-plane and $\sigma$-plane, 
{respectively.} 
\end{defn}

{Let us denote by $[Q_y]$ the image of $y \in \hcoY$ under $\hcoY\to\Hes$. 
By slight abuse of terminology, we say $y$ is a rank $k$ point if 
$\rank Q_y=k$. By Proposition \ref{cla:double2}, the fiber of $\Zpq\to \hcoY$ 
over a rank $3$ or $4$ point $y$ is a conic, which we denote it by $q_y$.}


\begin{prop}
\label{lem:qy}$(1)$ If $\rank Q_{y}=4$, then $q_{y}$ is a $\tau$-conic. $(2)$
If $\rank Q_{y}=3$, then 
the plane $\mP_{q_{y}}$ is a $\rho$-plane, hence $q_{y}$ is a $\rho$-conic. 
\end{prop}
\begin{proof}
(1) If $q_y$ is a $\rho$-conic,
then $(n-2)$-planes in $Q_y$ parameterized by $q_y$ must contain a $\mP(V_{n-2})$ in common
but this can not be the case.
If $q_y$ is a $\sigma$-conic,
then $(n-2)$-planes in $Q_y$ parameterized by $q_y$
must be contained in one $\mP(V_{n})$
but this also can not be  the case.
Hence $q_{y}$ is a $\tau$-conic. The claim (2) is clear since the planes
parametrized by $q_{y}$ contain the vertex $\mP(V_{n-2})$ of $Q_{y}$
in common. 
\end{proof}

\begin{example}
\label{ex:conics} 

(Smooth Conics) Taking a basis ${\bf e}_{1},\dots,{\bf e}_{n+1}$
of $V$, consider the subspaces  $V_{n-3}=\langle{\bf e}_{4},\dots, {\bf e}_{n}\rangle$,
$V_{n}=\langle{\bf e}_{1},\dots,{\bf e}_{n}\rangle$
and $V_{n-2}=\langle{\bf e}_{4},\dots, {\bf e}_{n+1}\rangle$. An 
example of $\tau$-conic may be given\[
q_{\tau}=\left\{ [s{\bf e}_{1}+t{\bf e}_{2},s{\bf e}_{3}+t{\bf e}_{4},{\bf e}_{5},\dots, {\bf e}_{n+1}]\mid[s,t]\in\mP^{1}\right\} .\]
Similarly, as a $\mP^{1}$-family of planes in the $\rho$-plane ${\rm P}_{V_{n-2}}$
and $\tau$-plane ${\rm P}_{V_{n-3}V_{n}}$, respectively, we have the
following examples: \[
q_{\rho}=\left\{ [s^{2}{\bf e}_{1}+st{\bf e}_{2}+t^{2}{\bf e}_{3},{\bf e}_{4},\dots,{\bf e}_{n+1}]\right\} ,\; q_{\sigma}=\left\{ [s{\bf e}_{1}+t{\bf e}_{2},s{\bf e}_{2}+t{\bf e}_{3},{\bf e}_{4},\dots, {\bf e}_n]\right\} ,\]
where $[s,t]\in\mP^{1}$ parameterizes each conic $q$. \hfill $\square$ 
\end{example}

\begin{example}
\label{ex:ranktwo}
(Rank two conics)
Since a line in $\rG(n-1,V)$ takes the form $l_{V_{n-2}V_{n}}=\left\{ [\Pi]\mid V_{n-2}\subset\Pi\subset V_{n}\right\} $
with some $V_{n-2}\subset V_{n}\subset V$, reducible conics $q$ have
the following form: \begin{equation}
q=l_{V_{n-2}V_{n}}\cup l_{V_{n-2}'V_{n}'}
\label{eq:conic-q-rk2}\end{equation}
with

$\bullet$ $\dim(V_{n-2}\cap V_{n-2}')\geq
n-3$, 

$\bullet$ $V_{n-2},V_{n-2}'\subset V_{n}\cap V_{n}'$, and

$\bullet$ $V_{n-2}\not =V'_{n-2}$ or $V_n\not =V'_n$.

\noindent
These conics will be described in detail in the section \ref{section:FY}. 
\end{example}

{Descriptions of rank one conics may be found in Appendix \ref{app:aU}.}

\subsection{Hilbert scheme $\hcoY_0$ of conics on $\rG(n-1,V)$}
\label{subsection:Hilb}
Consider a point $[U]\in \rG(3,\wedge^{n-1}V)$.  
To describe conics in $\rG(n-1,V)\subset \mP(3,\wedge^{n-1}V)$,
it suffices to find a {condition} for a plane $\mP(U)$
to be contained in $\rG(n-1,V)$ or cut out a conic from 
$\rG(n-1,V)$.
For this, we introduce 
the composite $\varphi$ of the following maps:
\begin{equation}
\varphi\colon \ft{S}^{2}(\wedge^{n-1}V)\simeq\ft{S}^{2}(\wedge^{2}V^{*})
\overset{\psi}{\to}\wedge^{4}V^{*},\label{eq:phiU}
\end{equation}
where the first map is {induced by} 
the duality $\wedge^{n-1}V\simeq\wedge^{2}V^{*}$
coming from the wedge product pairing 
$\wedge^{n-1}V\times\wedge^{2}V\to\wedge^{n+1}V$,
and $\psi$ is {induced by} the wedge product.
Note that the zero locus of $\psi$ is nothing but $\rG(2,V^*)$ 
{since we obtain the Pl\"ucker quadrics defining $\rG(2,V^*)$
by writing $\psi$ with coordinates.}
Moreover, the duality $\wedge^{n-1}V\simeq\wedge^{2}V^{*}$
induces an isomorphism $\rG(n-1,V)\simeq \rG(2,V^*)$.
Therefore $\rG(n-1,V)$ is the zero locus of $\varphi$.

Now we consider the restriction of $\varphi$ to
a $3$-plane $U\subset \wedge^{n-1} V$:
\[
\varphi_{U}:=\varphi\vert_{\ft{S}^{2}U}\colon  \ft{S}^{2}U\to \wedge^4 V^*.
\]
Let $U'$ be the $3$-plane of $\wedge^2 V^*$ corresponding to $U$
and denote by $\psi_{U'}$ the restriction of $\psi$ to $U'$.
Since $\rG(2,V^*)$ is
the zero locus of $\psi$,
$\mP(U')\subset \rG(2,V^*)$ iff $\psi_{U'}=0$.
Similarly,
$\mP(U')\cap \rG(2,V^*)$ is a conic iff 
the restrictions of the Pl\"ucker quadrics on $\mP({\ft S}^2 {U'}^*)$ form 
a point, i.e., one-dimensional subspace of ${\ft S}^2 {U'}^*$,
which is equivalent to the condition $\rank \psi_{U'}=1$.
Translating this, we immediately obtain the following 
descriptions on the {intersection} $\mP(U)\cap \rG(n-1,V)$:

\begin{prop}
\label{pro:conic-varphiU}
For a $3$-plane $U\subset \wedge^{n-1}V$,
$\mP(U)\cap \rG(n-1,V)$ contains a conic
iff $\rank \phi_U\leq 1$.
Moreover, the following {properties} hold$\,:$
\begin{myitem2} \item[$(1)$]  
$\left\{ [U]\in\rG(3,\wedge^{n-1} V)\mid
\varphi_{U}=0\right\} =\overline{\Prt}_{\rho}\sqcup\overline{\Prt}_{\sigma}$.

\item[$(2)$] If $\rank \varphi_{U}=1$, then 
$\mP(U)\cap \rG(n-1,V)$ is a conic which is the zero
locus of $\varphi_U$.
\end{myitem2}\end{prop}

Motivated from the above descriptions of conics, we {define} 
the following scheme with reduced structure:
\begin{equation}
\hcoY_{0}:=\left\{ ([U],[c_{U}])\mid[U]\in\rG(n-1,V),
[c_{U}]\in\mP(\ft{S}^{2}U^{*})\text{ s.t. } (c_{U})_0\subset (\varphi_{U})_0
\right\},\label{eq:resolutionY0}\end{equation}
where 
$(c_{U})_0$ and $(\varphi_{U})_0$ {represents} the zero locus in $\mP(U)$ of 
$c_U$ and $\varphi_{U}$, respectively.


\begin{thm}
\label{prop:Y0Hilb}
$\hcoY_0$ is smooth and isomorphic to the Hilbert
scheme of conics on $\rG(n-1,V)$. 
\end{thm}

\begin{proof}
By definition, $\hcoY_0$ obviously parameterizes conics
in $\rG(n-1,V)$ in one to one way.
Moreover, there is a family 
in $\mP(\wedge^{n-1} V)\times \hcoY_0$ 
of corresponding conics $(c_U)_0$ at each point $([U], [c_U])\in \hcoY_0$.
Therefore, by the universal property of the Hilbert scheme,
there is a unique map from $\hcoY_0$ 
to the Hilbert scheme {$\mathrm{Hilb}^{\mathrm{co}}\rG(n-1,V)$} of 
conics in $\rG(n-1,V)$.
{Since the smoothness of the Hilbert scheme 
is known in \cite{IM} and \cite{CHK}, 
we have $\hcoY_0\simeq \mathrm{Hilb}^{\mathrm{co}}\rG(n-1,V)$.}
\end{proof}

{
Let us consider the natural projection $\hcoY_0\to \rG(n-1,V)$ and denote
by $\overline{\hcoY}$ its image with the reduced structure. 
Let $\nu\colon \overline{\hcoY}'\to \overline{\hcoY}$ be 
the normalization (one should be able to show that $\overline{\hcoY}$ 
is normal in general extending the explicit 
description given in \cite{Arxiv} for $n=4$).
The following descriptions of $\overline{\hcoY}$ 
and related properties are easy to derive:
}

\begin{prop} \label{prop:barY-1-2-3}
{$(1)$ We have
\[
\overline{\hcoY}=
\left\{ [U]\in\rG(3,\wedge^{n-1}V)\mid\rank \varphi_U\leq 1\right\}.
\]
$(2)$ $\hcoY_0\to \overline{\hcoY}'$ is isomorphism 
outside $\nu^{-1}\overline{\Prt}_{\rho}$ and $\nu^{-1}\overline{\Prt}_{\sigma}$.}

\noindent
$(3)$ Let $G_{\rho}$ and $F_{\sigma}$ be 
the exceptional set over $\nu^{-1}\overline{\Prt}_{\rho}$ and $\nu^{-1}\overline{\Prt}_{\sigma}$, respectively. 
Then $G_{\rho}\to \nu^{-1}\overline{\Prt}_{\rho}$ and $F_{\sigma}\to \nu^{-1}\overline{\Prt}_{\sigma}$ are $\mP^5$-bundles whose fiber parameterizes $\rho$- or $\sigma$-conics in a fixed $\rho$- or $\sigma$-plane respectively. 

\end{prop}
\subsection{Small resolution $\hcoY_{3}\to \overline{\hcoY}'$}
{We find  a small resolution
$\hcoY_3\to \overline{\hcoY}'$ by 
translating the condition  $\rank \varphi_U\leq 1$  into an equivalent 
form.}
%
For each $v\in V$, {let us} define a linear map 
$E_{v}:\wedge^{n-1}V\to\wedge^{n}V$
by $u\mapsto v\wedge u$. Consider the restriction $E_{v}\vert_{U}$
to $U\subset\wedge^{n-1}V$ and introduce 
\[
a_{U}=\left\{ v\in V\mid E_{v}\vert_{U}=0\right\},
\]
which is nothing but the annihilator of $U$.
Note that $\dim U=3$ implies $\dim a_U\leq n-2$.
We prove the following proposition  
in Appendix \ref{sec:Appendix-B}.
\begin{prop}
\label{lem:appendixB-UU-solve} 
For $[U]\in\rG(3,\wedge^{n-1}V)$,
$\dim a_{U}\geq n-3 \Longleftrightarrow \rank\varphi_{U}\leq 1$.
\end{prop}

By this proposition, {it is immediate to see} that
\[
\overline{\hcoY}=\left\{ [U]\in\rG(3,\wedge^{n-1}V)
\mid\dim a_{U}\geq n-3\right\}. 
\]
{Below we define} a Springer type resolution 
$\hcoY_3\to \overline{\hcoY}'$, 
which turns out to be a small resolution.
\begin{defn}
\label{def:Y3} {For $n\geq 3$,} we define
\[
\hcoY_{3}=\left\{ ([U],[V_{n-3}])\mid V_{n-3}\subset a_{U}\right\} 
\subset\mathrm{G}(3,\wedge^{n-1}V)\times\mathrm{G}(n-3,V),
\]
{where $\mathrm{G}(n-3,V)$ should be understood as one point when $n=3$.}
Obviously,
the image of the projection of $\hcoY_{3}$ to the
first factor 
coincides with $\overline{\hcoY}$.
\end{defn}
Since $E_{v}\vert_{U}=0\,(\forall v\in V_{n-3})$ implies that $U$ is
the $\mathbb{C}$-span of non-vanishing vectors of 
the form $\bar{u}_{i}\wedge v_1\wedge \dots \wedge v_{n-3}\,(i=1,2,3)$
with $\bar{u}_{i}\in\wedge^{2}(V/V_{n-3})$ and $v_1,\dots, v_{n-3}$ 
{being a} basis of $V_{n-3}$, the fiber of the natural
projection $\hcoY_{3}\to\mathrm{G}(n-3,V)$ over $[V_{n-3}]\in\mathrm{G}(n-3,V)$
can be identified with $\mathrm{G}(3,\wedge^{2}(V/V_{n-3}))$. Hence
we see that \[
\hcoY_{3}=\mathrm{G}(3,\wedge^2 \mathfrak{Q}),\]
and in particular $\hcoY_{3}$ is smooth. 

\begin{prop}
\label{pro:barYsing} 
The morphism $\Lrho_{\hcoY_{3}}:\hcoY_{3}\to\overline{\hcoY}'$
{is isomorphic over  $\overline{\hcoY}'\setminus\nu^{-1}\overline{\Prt}_{\rho}$ and is a small resolution with $\Lrho_{\hcoY_{3}}^{\,-1}(x)\simeq \mP^{n-3}$ for 
each $x \in \nu^{-1}\overline{\Prt}_{\rho}$.}
%
In particular, $\Lrho_{\hcoY_{3}}$ is an isomorphism if $n=3$, and
$\nu^{-1}\overline{\Prt}_{\rho}=\Sing \overline{\hcoY}'$ if $n\geq 4$.
\end{prop}

\begin{proof} {It is easy to see that the} fiber 
of $\hcoY_{3}\to\overline{\hcoY}'$ over
each point of $\nu^{-1}\overline{\Prt}_{\rho}$ 
is $\rG(n-3,n-2)\simeq\mP^{n-3}$, and $\hcoY_{3}\to\overline{\hcoY}'$
is bijective over $\overline{\hcoY}'\setminus\nu^{-1}\overline{\Prt}_{\rho}$. 
\end{proof}

\begin{rem}
\label{rem:Y3=Ybar}
In case $n=3$, we have
$\hcoY_3=\overline{\hcoY}'=\overline{\hcoY}=\rG(3,\wedge^2 V)$.
\end{rem}
\subsection{Small resolution $\widetilde{\hcoY}\to\overline{\hcoY}'$ 
via the Hilbert scheme $\hcoY_0$ \label{subsection:BlowUp}}

We construct another small resolution 
$\Lp\,_{\widetilde{\hcoY}}\colon 
\widetilde{\hcoY}\to \overline{\hcoY}'$ for $n\geq 4$,
which is the (anti-)flip of $\hcoY_3\to \overline{\hcoY}'$.
{We give} $\widetilde{\hcoY}$ from $\hcoY_{0}$ by contracting
the exceptional set (divisor) over {$\nu^{-1}\overline{\Prt}_{\sigma}$.} 

Let 
$R_{\rho}$ (resp.~$R_{\sigma}$) be the extremal ray spanned by
lines in fibers of $G_{\rho}\to \nu^{-1}\overline{\Prt}_{\rho}$ 
(resp.~$F_{\sigma}\to \nu^{-1}\overline{\Prt}_{\sigma}$). 
We show that $R_{\rho}\not =R_{\sigma}$.
Indeed, note that $F_{\sigma}$ is a prime divisor and 
$G_{\rho}\cap F_{\sigma}=\emptyset$. Therefore, $F_{\sigma}\cdot R_{\rho}=0$ 
and $F_{\sigma}\cdot R_{\sigma}<0$ and hence $R_{\rho}\not =R_{\sigma}$.  
Since $\overline{\hcoY}'$ is smooth along $\overline{\Prt}_{\sigma}$
by Proposition \ref{pro:barYsing},
the discrepancy of $F_{\sigma}$ is positive and then $R_{\sigma}$ is 
$K_{\hcoY_0}$-negative. Therefore there exists a unique extremal 
contraction $\hcoY_0\to\widetilde{\hcoY}$ over $\overline{\hcoY}'$ 
associated to $R_{\sigma}$, which is nothing but the contraction 
of $F_{\sigma}$. We denote by $G_{\sigma}$ the image of $F_{\sigma}$. 

{
The following proposition follows from the above construction 
of $\widetilde{\hcoY}$:} 
\begin{prop}
$\widetilde{\hcoY}$ parameteirizes $\tau$- and $\rho$-conics, and 
$\sigma$-planes.
\end{prop}

{We retain the notation $G_{\rho}$ to represent the locus in 
$\widetilde{\hcoY}$ parameterizing $\rho$-conics and denote by 
$\sQ_{\rho}$ the universal quotient bundle on $\rG(n-2,V)$.}

\begin{prop} \label{prop:Grho}
$G_{\rho}$ is isomorphic to $\mP({\ft S}^2 \sQ_{\rho}^*)$.
{It} is also isomorphic to $\widetilde{\nS}_3$.
\end{prop}

\begin{proof}
The first claim {is clear} since 
$\mP(\sQ_{\rho})\to \overline{\Prt}_{\rho}\simeq \rG(n-2,V)$ is the 
family of {$\rho$-planes.}
The second one follows {from the definition of the resolution  
$\Lp\,_{\widetilde{\nS}_3}:\widetilde{\nS}_3\to\nS_3$ 
(see Proposition~\ref{prop:Spr}).}
\end{proof}


\begin{prop}
\label{prop:tildeY}
$\Lp_{\widetilde{\hcoY}}\colon \widetilde{\hcoY}\to \overline{\hcoY}'$ is 
a small resolution for $n\geq 4$,
and is the blow-up along $\nu^{-1}\overline{\Prt}_{\rho}$ for $n=3$.
Non-trivial fibers of $\Lp_{\widetilde{\hcoY}}$ are copies of $\mP^5$.
\end{prop}

\begin{proof}
$\widetilde{\hcoY}$ is smooth
since $\hcoY_0$ is smooth by Theorem \ref{prop:Y0Hilb}
and $\overline{\hcoY}'$ is smooth along $\nu^{-1}\overline{\Prt}_{\sigma}$
by Proposition \ref{pro:barYsing}.

Note that $G_{\rho}$ is the $\Lp_{\widetilde{\hcoY}}$-exceptional locus
since the restriction of $\Lp_{\widetilde{\hcoY}}|_{G_{\rho}}$
is a $\mP^5$-bundle over $\nu^{-1}\overline{\Prt}_{\rho}\simeq 
\overline{\Prt}_{\rho}$.   
If $n\geq 4$, then $G_{\rho}$ is not a divisor by dimension count.
In case $n=3$, $G_{\rho}$ is a prime divisor. 
Since $\overline{\hcoY}'$ is smooth by 
Proposition \ref{pro:barYsing}, 
and $G_{\rho}\to \nu^{-1}\overline{\Prt}_{\sigma}$ is a $\mP^5$-bundle,
we see that $K_{\widetilde{\hcoY}}=
\Lp_{\widetilde{\hcoY}}^*K_{\overline{\hcoY}'}+5G_{\rho}$.
Let  
$\Lp{\,'}_{\widetilde{\hcoY}}\colon \widetilde{\hcoY}'\to \overline{\hcoY}'$
be the blow-up along $\nu^{-1}\overline{\Prt}_{\rho}$
and $G'_{\rho}$ the $\Lp{\,}'_{\widetilde{\hcoY}}$-exceptional divisor.
Then we have 
$K_{\widetilde{\hcoY}'}=
{\Lp{\,'}}_{\widetilde{\hcoY}}^*K_{\overline{\hcoY}'}+5G'_{\rho}$.
It is well-known that there is only one valuation of $k(\overline{\hcoY}')$
 associated to exceptional divisors with center 
{$\nu^{-1}\overline{\Prt}_{\rho}$} and discrepancy $5$.
Therefore, $\widetilde{\hcoY}$ and $\widetilde{\hcoY}'$ are isomorphic in 
codimension one. Moreover, since 
$-K_{\widetilde{\hcoY}}$ and $-K_{\widetilde{\hcoY}'}$ are relatively ample
over $\overline{\hcoY}'$,  
$\widetilde{\hcoY}$ and $\widetilde{\hcoY}'$ 
{must be} isomorphic by \cite[Lemma 5.5]{Tk}.
\end{proof}


\subsection{\label{sub:tildeYtoH}Rational map $\hcoY_3\dashrightarrow \Hes$ via double spin decomposition}
\label{subsection:spin}

Consider a point $({[U]},[V_{n-3}])
\in\hcoY_{3}=\rG(3,\wedge^2 \mathfrak{Q})$
with ${[U]}\in\rG(3,\wedge^{2}(V/V_{n-3}))$. To describe $\wedge^{3}{U}$,
we use the following irreducible decomposition as $sl(V/V_{n-3})$-modules
(see \cite[$\S 19.1$]{FH} for example): 
\begin{equation}
\begin{aligned}
&\wedge^{3}(\wedge^{2}(V/V_{n-3})) = \\
&\ft{S}^{2}(V/V_{n-3})\otimes\det(V/V_{n-3})\oplus
\ft{S}^{2}(V/V_{n-3})^{*}\otimes\det(V/V_{n-3})^{\otimes2}.
\end{aligned}
\label{eq:spin}
\end{equation}
 We will call this {}``double spin'' decomposition since the 
{symmetric powers} in the r.h.s. are identified with 
$V_{2\lambda_{s}}$ and $V_{2\lambda_{\bar{s}}}$
as $so(\wedge^{2}V/V_{n-3})(\simeq sl(V/V_{n-3}))$-modules, where $\lambda_{s}$
and $\lambda_{\bar{s}}$ represent the spinor and conjugate spinor
weights, respectively (see {[}\textit{loc. cit.}{]}). 
{Considering this decomposition fiberwise in the projective 
bundle $\mP(\wedge^3(\wedge^2\mathfrak{Q}))$ over $\rG(n-3,V)$,
we have} the following sequence of
(rational) maps: 
\begin{equation}
\begin{aligned}
\hcoY_{3}\hookrightarrow 
& \mP(\ft{S}^{2}\mathfrak{Q}\otimes\sO_{\mathrm{G}(n-3, V)}(-1)
   \oplus\ft{S}^{2}\mathfrak{Q}^{*}) \\
& 
  \qquad\qquad 
 \dashrightarrow 
\UU=\mP(\ft{S}^{2}\mathfrak{Q}^{*})\hookrightarrow\mP(\ft{S}^{2}V^{*}\otimes\sO_{\mathrm{G}(n-3, V)}),\end{aligned}
\label{eq:phi-Y-H-as-bundle-hom}
\end{equation}
where the rational map in the middle is the projection to the second
factor and the {last} inclusion comes from the surjection $V\otimes\sO_{\mathrm{G}(n-3, V)}\to \mathfrak{Q}\to0$.
We further consider the natural projection $\mP(\ft{S}^{2}V^{*}\otimes\sO_{\mathrm{G}(n-3, V)})\to\mP(\ft{S}^{2}V^{*})$.
Then the image of the composite is contained in the locus $\Hes$ of 
the quadrics of rank $\leq 4$, and {hence} we have a rational map
\[
\phi\colon \hcoY_{3}\dashrightarrow\Hes\,{(:=\nS_4).}
\]
To obtain a morphism,
we consider the inverse {images} 
$\Prt_{\rho},\Prt_{\sigma}$ of {$\nu^{-1}\overline{\Prt}_{\rho}$
and $\nu^{-1}\overline{\Prt}_{\sigma}$, respectively, under 
the resolution $\hcoY_{3}\to\overline{\hcoY}'$.}
Then it is clear from the definitions that \begin{equation}
\Prt_{\rho}\simeq \rF(n-3,n-2;V)\simeq\mP(\mathfrak{Q}),\;\;\Prt_{\sigma}\simeq \rF(n-3,n;V)\simeq\mP(\mathfrak{Q}^{*}).\label{eq:Prho-Ptau-by-Tangent}\end{equation}
\begin{prop}
\label{prop:Prt}
Under the embedding {$\hcoY_{3}\subset\mP(\ft{S}^{2}\mathfrak{Q}
\otimes\sO_{\rG(n-3,V)}(-1)\oplus\ft{S}^{2}\mathfrak{Q}^{*})$,}
$\Prt_{\rho}$ and $\Prt_{\sigma}$ are identified with \[
\Prt_{\rho}=v_{2}(\mP(\mathfrak{Q})),\;\;\Prt_{\sigma}=v_{2}(\mP(\mathfrak{Q}^{*})).\]
Moreover,
$\Prt_{\rho}=\hcoY_3\cap \mP(\ft{S}^{2}\mathfrak{Q}\otimes\sO_{\mathrm{G}(n-3, V)}(-1))$ scheme-theoretically.
\end{prop}

\begin{proof} {The claims follows from the decomposition
(\ref{eq:spin}) and its explicit description given in 
Proposition \ref{prop:B1}, (\ref{eq:Ivw}).}
\end{proof}

\begin{defn}\label{def:Y2}
We define $\Lrho_{\hcoY_2}\colon\hcoY_2\to \hcoY_3$ to be the 
blow-up along $\Prt_{\rho}$, {and denote}
by $F_{\rho}$ its exceptional divisor.
\end{defn}

{
Clearly there is a morphism $\hcoY_2\to\rG(n-3,V)$ as well as 
$\hcoY_3\to\rG(n-3,V)$.}

\subsection{The case $n=3$ ($\dim V=4$)}
\label{subsection:dimV=4}
{
When $n=3$, projective bundles over $\rG(n-3,V)$ reduce to the corresponding projective spaces, and considerable simplifications may be observed, for example, 
in 
\[
\hcoY_3=\overline{\hcoY}'=\overline{\hcoY}=\rG(3,\wedge^2V) \text{ and } 
\Prt_\rho=v_2(\mP(V))\subset \mP(\ft{S}^2V). 
\]
Also in this case, we have $\hcoY_2
\simeq\widetilde{\hcoY}$ by Propositions \ref{pro:barYsing} and 
\ref{prop:tildeY}. Then the birational morphism 
$\phi\colon \hcoY_{3}\dashrightarrow\Hes(=\mP(\ft{S}^2V^*))$ lifts to
a morphism $\widetilde{\phi}\colon \hcoY_{2}\to\Hes$ by the last assertion in Proposition \ref{prop:Prt}.}

{In this subsection, we study the case $n=3 \,(\dim V=4)$ (where 
$\Hes=\mP(\ft{S}^2V^*)$).  
The results below  will be used to study the case of 
$n\geq 4 \,(\dim V\geq 5)$ (where $\Hes=\nS_4\subset \mP(\ft{S}^2V^*)$) in
the next subsection.  Also these will be used extensively in~\cite{ReyeEnr}.}

\begin{prop}
\label{prop:n=3fib}
\begin{enumerate}[$(1)$]
\item
The Stein factorization of $\widetilde{\phi}\colon \widetilde{\hcoY}\to
\Hes$ factors through the double cover $\Lrho\,_{\hcoY}\colon \hcoY\to \Hes$.
\item Let $\Lrho_{\widetilde{\hcoY}}\colon \widetilde{\hcoY}\to \hcoY$ be the
induced morphism. Then $\Lrho_{\widetilde{\hcoY}}$ is birational
and a $K_{\widetilde{\hcoY}}$-negative extremal divisorial contraction.
\item
Let $F_{\widetilde{\hcoY}}$ be the $\Lrho_{\widetilde{\hcoY}}$-exceptional divisor. Then the image of $F_{\widetilde{\hcoY}}$ by $\Lrho_{\widetilde{\hcoY}}$
coincides with $G_{\hcoY}$, and $F_{\widetilde{\hcoY}}\to G_{\hcoY}$
is a $\mP^1\times \mP^1$-fibration outside $G_{\hcoY}^1$.
\item
It holds that 
\[
K_{\widetilde{\hcoY}}=\Lrho_{\widetilde{\hcoY}}^{\;*}K_{\hcoY}+F_{\widetilde{\hcoY}}.
\]
In particular, $\hcoY$ has only terminal singularities with 
$\Sing\hcoY=G_{\hcoY}$.
\item
Let $w=(w_{kl})$ be the $4\times4$ matrix representing $[Q]\in\mP(\ft{S}^{2}V^{*})$.
Then the fiber of $\tilde{\phi}$ is described according to the rank of $w$ as follow\,$:$
\begin{enumerate}[$(\rm a)$]
\item When $\rank w=4$, $\tilde{\phi}^{-1}([Q])$ consists of two points.

\item  When $\rank w=3$, $\tilde{\phi}^{-1}([Q])$
consists of one point. 

\item  When $\rank w=2$, 
$\tilde{\phi}^{-1}([Q])\simeq\mP^{1}\times\mP^{1}.$

\item  When $\rank w=1$, 
$\tilde{\phi}^{-1}([Q])\simeq\mP(1^{3},2).$
The vertex of $\tilde{\phi}^{-1}([Q])$
corresponds to the $\sigma$-plane ${\rm P}_{V_3}$,
where $Q=2\mP(V_3)$, and $\tilde{\phi}^{-1}([Q])\cap F_{\rho}\simeq \mP^2$
which is a hyperplane section of $\mP(1^{3},2)\subset \mP^6$.
\end{enumerate}

\end{enumerate}
\end{prop}

\begin{proof}
Let $\widetilde{\hcoY}\to \hcoY'\to \Hes$ be the Stein factorization of 
$\widetilde{\phi}$. We denote by $\Lrho_{\widetilde{\hcoY}}$
and $F_{\widetilde{\hcoY}}$, 
the induced morphism $\widetilde{\hcoY}\to \hcoY'$
and the $\Lrho_{\widetilde{\hcoY}}$-exceptional locus respectively
(this notation will be compatible
with {(2) and (3)} after showing that
the induced morphism $\hcoY'\to \Hes$ coincides with
the double cover $\Lrho\,_{\hcoY}\colon\hcoY\to \Hes$).

Let us start with showing
that $\widetilde{\phi}(F_{\rho})=\nS_3$.
Let $Q$ be a rank three quadric $Q$ in $\mP(V)$. 
Then, from (I.3) in Appendix \ref{sec:Appendix-A}, 
$[Q]$ cannot be {in} the image of $\phi$. 
Therefore the locus $\nS_3$ is contained in $\widetilde{\phi}(F_{\rho})$. Since $F_{\rho}$ and $\nS_3$ are prime divisors in $\widetilde{\hcoY}$ and $\Hes$ respectively,
it holds that $\widetilde{\phi}(F_{\rho})=\nS_3$.

\noindent {\bf Proof of (5) (a).}
Let $Q$ be a rank four quadric $Q$ in $\mP(V)$, 
{i.e., $[Q]\in\nS_4\setminus\nS_3$.} 
From (I.2) in Appendix$\,$\ref{sec:Appendix-A},
$\phi^{-1}([Q])$ consists of two points $[v,w]$
satisfying $v.w=\pm\sqrt{\det w}\,\id_{4}$. 
{Since} $\widetilde{\phi}(F_{\rho})=\nS_3$,
$\widetilde{\phi}^{-1}([Q])$ also consists of two points.



{We know now that} $\hcoY'\to \Hes$ is a finite morphism
of degree two, and its branch locus is contained in $\nS_3$.

\noindent {\bf Proof of a weaker assertion than (5) (c).}
Let $Q$ be a rank two quadric $Q$ in $\mP(V)$ and $w$ an associated 
symmetric matrix. We show that $\tilde{\phi}^{-1}([Q])$ contains a $\mP^1\times \mP^1$. 
Changing the coordinate of $V$ suitably, we may assume that
$[w]$ is given in the form $w_{0}=\left(\begin{smallmatrix}\begin{smallmatrix}0 & 1\\
1 & 0\end{smallmatrix} & O_{2}\\
O_{2} & O_{2}\end{smallmatrix}\right)$ with $O_{2}$ being the $2\times2$ zero matrix. Then by the properties
(I.4) and (I.2), we obtain $v=\left(\begin{smallmatrix}O_{2} & O_{2}\\
O_{2} & \begin{smallmatrix}v_{11} & v_{12}\\
v_{12} & v_{22}\end{smallmatrix}\end{smallmatrix}\right)$ . Now substituting $[v,w]=[v,tw_{0}]\,(t\not=0)$ into the equation
in the first line of (\ref{eq:Ivw}), we have \[
v_{11}v_{22}-v_{12}^{2}+t^{2}=0\;\;(t\not=0).\]
The closure $S$ of this locus in {$\hcoY_3=\rG(3,\wedge^2V)$} is 
isomorphic to $\mP^1\times \mP^1$. 
{Note that the restriction of the blow-up} $\widetilde\hcoY \to \hcoY_3$ 
over $S\subset \hcoY_3$ is the blow-up along the locus $t=0$. 
Hence the strict 
transform $S'$ of $S$ in $\widetilde{\hcoY}$
is also isomorphic to $\mP^1\times \mP^1$.
Note that $S'$ is contained in the fiber {of the restriction over $S$.}

\noindent {\bf Proof of a similar statement to (2) for 
$\widetilde{\hcoY}\to \hcoY'$.}
Since $\rho(\widetilde{\hcoY})=2$, we have 
$\rho(\widetilde{\hcoY}/\hcoY')\leq 1$.
Moreover, since the fiber over a rank two point is at least $2$-dimensional
and $\dim \nS_2=6$, $F_{\widetilde{\hcoY}}$ is a prime divisor.
We see that 
the contraction $\Lrho_{\widetilde{\hcoY}}$ is $K_{\widetilde{\hcoY}}$-negative by computing the intersection number between 
$K_{\widetilde{\hcoY}}$ and a ruling of $S'$. 
Thus $\hcoY'$ has only terminal singularities.

\noindent {\bf Proof of (1).}
$\hcoY'$ is Cohen-Macaulay since it is terminal
and hence
$\hcoY'\to \Hes$ is flat. Then its branch locus is 
empty or a divisor but the former case cannot occur since $\Hes=\mP({\ft S}^2 V^*)$ is simply connected. Therefore the branch locus of $\hcoY'\to \Hes$ coincides with $\nS_3$. Now we see that $\hcoY'\simeq \hcoY$ since
both $\hcoY'\to \Hes$ and $\hcoY\to \Hes$ are both flat, 
finite of degree two and
are branched along $\nS_3$.

\noindent {\bf Proof of (5) (b).}
Since $\nS_3$ is the branch locus, $F_{\rho}\to \nS_3$ is birational.
Therefore we see that the fiber over a rank three points consists of one point 
{as claimed.}


\vspace{5pt}

We have shown (1), (2), the {first} half of (3) and (5) (b).
The {second half} of (3) will follow from (5) (c).

We will show two resolutions $F_{\rho}\to \nS_3$ and 
$\widetilde{\nS}_3\to \nS_3$ coincides with each other.
First we note that $\rho(F_{\rho})=\rho(\widetilde{\nS}_3)=2$ and then
$\rho(F_{\rho}/\nS_3)=\rho(\widetilde{\nS}_3/\nS_3)=1$.
Since $\nS_3$ is $\mQ$-factorial, $F_{\rho}\to \nS_3$ is a divisorial contraction.
Let $G_1$ and $G_2$ be the exceptional divisors of 
$F_{\rho}\to \nS_3$ and $\widetilde{\nS}_3\to \nS_3$ respectively.
Since $\widetilde{\nS}_3\to \nS_3$ is a crepant resolution,
the valuation of $G_2$ in $k(\nS_3)$ is a unique crepant valuation.
If the discrepancy of $G_1$ is positive, then we see that
any exceptional valuation in $k(\nS_3)$ must have positive discrepancy by computing the discrepancies of
exceptional divisors over $F_{\rho}$, {which is} 
a contradiction to the existence of $G_2$.
Therefore $F_{\rho}\to \nS_3$ is crepant, and moreover
the valuations of $G_1$ and $G_2$ are the same by the uniqueness of 
the crepant valuation. In particular,  
$F_{\rho}\to \nS_3$ and $\widetilde{\nS}_3\to \nS_3$ are isomorphic in
codimension one.
Note that $-G_1$ and $-G_2$ are relatively ample over $\nS_3$.
Let $p\colon \Gamma\to F_{\rho}$ and $q\colon \Gamma\to \nS_3$
be a common resolution of $F_{\rho}$ and $\nS_3$.
Thus, by the standard argument using the negativity lemma,
we see that $p^*(-G_1)=q^*(-G_2)$.
This implies that  
two resolutions $F_{\rho}\to \nS_3$ and 
$\widetilde{\nS}_3\to \nS_3$ coincides with each other.

\noindent {\bf Proof of (5) (c).}
As we see above, 
the fiber over a rank two point $[Q]$
contains at least $S'\simeq \mP^1\times \mP^1$.
The fiber of $F_{\rho}\to \nS_3$ over
$[Q]$ is isomorphic to $\mP^1$ by the description of the fibers of 
$\widetilde{\nS}_3\to \nS_3$.
Thus the fiber $\tilde{\phi}^{-1}([Q])$
coincides with $S'$.

\noindent {\bf Proof of (4).}
We obtain the {claimed formula} 
by computing the intersection number between 
$K_{\widetilde{\hcoY}}$ and a ruling of $S'$. 

\noindent {\bf Proof of (5) (d).}
Let $Q$ be a rank one quadric in $\mP(V)$
and $w$ an associated symmetric matrix.
Then $w$ can be written as $(a_{k}a_{l})$ with some
$a\in V^{*}$. Then from (I.5) in Appendix, we see that
$\rank v\leq 1$. Writing $v_{ij}=x_{i}x_{j}$ with $x\in V$
and also solving (\ref{eq:Ivw}) we obtain
\begin{equation}
\label{eq:WPSeq}
\phi^{-1}([Q])=\left\{ [x_{i}x_{j},ta_{k}a_{l}]\mid a.x=0,t\not=0\right\} .
\end{equation}
The closure of this locus in $\hcoY_3$ is isomorphic to
the cone over $v_{2}(\mP^{2})\simeq\mP^{2}$
from the vertex $[0,a_{k}a_{l}]\in\mP(\ft{S}^{2} V\oplus\ft{S}^{2} V^{*})$,
which is isomorphic to $\mP(1^{3},2)$. 
Then we have the former assertion (5) (d) by a similar argument in case $\rank w=2$. The latter assertion is clear from the above description.

\end{proof}

\begin{rem}
\label{rem:WPS}
It is convenient to give a coordinate-free description of $\widetilde{\phi}^{-1}([Q])$ in case $\rank Q=1$. 
{Instead of $\widetilde{\phi}^{-1}([Q])$, 
we may} describe {its isomorphic} 
image $\Phi \subset \hcoY_3$ {under 
$\widetilde\hcoY=\hcoY_2\to\hcoY_3$}.  
Note that $\Phi$ is the closure {in $\hcoY_3$} of $\phi^{-1}([Q])$ 
and its equation is given by 
(\ref{eq:WPSeq}).
Let $Q=2\mP(V_3)$ as in Proposition \ref{prop:n=3fib} (5) (d).
The vertex of $\widetilde{\phi}^{-1}([Q])$ corresponds to
the $\sigma$-plane ${\rm P}_{V_3}=\{\mC^2\subset V_3\}$.
By the equation (\ref{eq:WPSeq}),
points $[{\rm P}_{V_1}]$ which correspond to $\rho$-planes and are contained in $\Phi$ satisfy $V_1\subset V_3$.
Since $\Phi$ is the cone over the Veronese surface $v_2(\mP(V_3))$,
it is swept out by lines joining $[{\rm P}_{V_3}]$ and $[{\rm P}_{V_1}]$ such that
$V_1\subset V_3$.
\end{rem}

\begin{prop}
\label{prop:conicquad}
For a $\tau$- or $\rho$-conic $q$,
$\Lrho_{\widetilde{\hcoY}}([q])$ is the point corresponding to
the quadric generated by lines which $q$ parameterizes.
For a $\sigma$-plane $\rm{P}_{V_3}$,
$\Lrho_{\widetilde{\hcoY}}([\rm{P}])$ is the point corresponding to
the rank one quadric $2\mP(V_3)$.
In particular,
the exceptional locus $F_{\widetilde{\hcoY}}$ consists of the points 
corresponding to $\tau$- or $\rho$-conics of rank at most two
or $\sigma$-planes,
and the image of $F_{\widetilde{\hcoY}}$ coincides with
$G_{\hcoY}$.
\end{prop}

\begin{proof}
We {have described} $\tau$-conics and $\sigma$-planes in
Examples \ref{ex:conics} and \ref{ex:ranktwo}
and Appendix \ref{sec:Appendix-B}.
The {assertions} for $\tau$-conics and $\sigma$-planes {follow from}
their descriptions and
direct computations
based on the results in Appendix \ref{sec:Appendix-A}. 
{For $\rho$-conics, the assertion follows from the isomorphism} 
$F_{\rho}\simeq \widetilde{\nS}_3$
as in the proof of Proposition \ref{prop:n=3fib}.
\end{proof}

\subsection{Divisorial contraction $\Lrho_{\widetilde{\hcoY}}\colon \widetilde{\hcoY}\to\hcoY$ for {$n\geq4$ ($\dim V\geq 5$)} \label{sub:tildeY-Y}}
{ Recall that we have the morphisms
\[
\hcoY_3\to \rG(n-3,V),\;\; \hcoY_2\to\rG(n-3,V) \;\text{ and } \; 
\UU \to \rG(n-3,V)
\]
from Definition \ref{def:Y2} and (\ref{eq:deftSr}) 
with $\UU:=\widetilde{\nS}_4.$ 
In this subsection, we consider the relative setting over $\rG(n-3,V)$ for 
$n\geq4$. Thus, for example, the geometry of $\hcoY_2$ is 
considered as the family of the blow-ups  
of $\rG(3,\wedge^2(V/V_{n-3}))$ along 
$\Prt_\rho\vert_{[V_{n-3}]}=v_2(\mP(V/V_{n-3}))$ for $[V_{n-2}]\in \rG(n-3,V)$. 
The results in the preceding subsection apply to each member of 
the family with the 4-dimensional vector space $V/V_{n-3}$. }

{
\begin{lem} 
There exists a morphism $\hcoY_2\to \UU$ defined over $\rG(n-3,V)$.
\end{lem}
}
{
\begin{proof}
Denote by $\hcoY_2\vert_{[V_{n-3}]}, \hcoY_3\vert_{[V_{n-3}]}, \UU\vert_{[V_{n-3}]}$ 
the restrictions to the fibers over $[V_{n-3}]\in \rG(n-3,V)$. 
Then $\hcoY_2\vert_{[V_{n-3}]}$ is the blow-up of $\hcoY_3\vert_{[V_{n-3}]}=
\rG(3,\wedge^2(V/V_{n-3}))$, as 
described above, and $\UU\vert_{[V_{n-3}]}=\mP(\ft{S}^2(V/V_{n-3})^*)$. The claimed morphism is the one described in Proposition \ref{prop:n=3fib} (1). 
\end{proof}
}

\begin{prop}
\label{prop:gendescr}
\begin{enumerate}[$(1)$]
\item There exists an extremal divisorial contraction
$\tLrho\,_{\hcoY_2}\colon \hcoY_2\to \widetilde{\hcoY}$ which is the blow-up 
along $G_{\rho}$ with the exceptional divisor $F_{\rho}$.
Any fiber of $F_{\rho}\to G_{\rho}$ is a copy of $\mP^{n-3}$
and is mapped to a fiber of $\hcoY_3\to \overline{\hcoY}'$ isomorphically.
\item There exists an extremal divisorial contraction 
$\Lrho\,_{\widetilde{\hcoY}}\colon \widetilde{\hcoY}\to \hcoY$.
\end{enumerate}
\end{prop}

\begin{proof}
We {reproduce} here a part of the diagram (\ref{eq:STcomm}):
\begin{equation}
\label{eq:STcommrev}
\begin{matrix}
\xymatrix{\hcoY_{\UU}=
\widetilde{\nT}_4\ar[r]^{\Lrho\,_{\widetilde{\nT}_4}}\ar[d]_{\Lp\,_{\widetilde{\nT}_4}} 
& \UU=\widetilde{\nS}_4\ar[d]_{\Lpi_{\widetilde{\nS}_4}}\\
\hcoY=\nT_4\ar[r]_{\Lrho\,_{\nT_4}} & \Hes=\nS_4.}
\end{matrix}
\end{equation}
By construction, we see 
that $\Lrho\,_{\widetilde{\nT}_4}\colon \hcoY_{\UU}\to \UU$ is
the family over $\rG(n-3,V)$
of the double covers $\nT_4\to \nS_4$ for 
$4$-dimensional vector spaces $V/V_{n-3}$.


{Consider the Stein factorization of the morphism $\hcoY_2\to\UU$. 
By the uniqueness of finite double cover, it is given by $\hcoY_2\to\hcoY_\UU
\to\UU$. Then the induced morphism 
$\hcoY_2\to \hcoY_{\UU}$ 
is the family over $\rG(n-3,V)$
of the divisorial contraction described in Proposition \ref{prop:n=3fib} (2) 
(for $4$-dimensional vector spaces $V/V_{n-3}$).}
In particular, a birational morphism $\hcoY_2\to \hcoY$ is induced.
By Proposition~\ref{pro:barYsing} and the definition of $\hcoY_2$,
a birational morphism $\hcoY_2\to \overline{\hcoY}'$ is also induced.
Therefore we obtain a map $\hcoY_2\to
\overline{\hcoY}'\times \hcoY$.
Let $\widetilde{\hcoY}'$ be the normalization of the image of this map.
We will show that $\hcoY_2\to \widetilde{\hcoY}'$ is non-trivial. 
Let $Q$ be a quadric in $\mP(V)$ of rank $3$ and $\mP(V_{n-2})$ 
its singular locus.
By Proposition \ref{prop:conicquad},
the fiber $\Gamma$ of $\hcoY_2\to \hcoY$ over $[Q]$
is isomorphic to $\rG(n-3,V_{n-2})$.
By Proposition~\ref{pro:barYsing} and the definition of $\hcoY_2$,
$\Gamma$ is also contracted by $\hcoY_2\to \overline{\hcoY}'$.
Therefore $\hcoY_2\to \widetilde{\hcoY}'$ is non-trivial.
{$\widetilde{\hcoY}'$ can not be isomorphic to $\overline{\hcoY}'$ nor 
$\hcoY$ since $\rho(\overline{\hcoY}')=\rho(\hcoY)=1$ and
$\overline{\hcoY}'\not \simeq \hcoY$.}
Therefore {
$\widetilde{\hcoY}'\to \overline{\hcoY}'$} is 
a small birational morphism.
By the uniqueness of the flip (cf.~\cite{KM}), 
we see that $\widetilde{\hcoY}'\simeq \widetilde{\hcoY}\ \text{or}\ \hcoY_3$.
There does not exist, however,
a contraction $\hcoY_3\to \hcoY$ since 
{$\rho(\hcoY_3)=2$ and there are }
two non-trivial contractions $\hcoY_3\to \rG(n-3,V)$ and
$\hcoY_3\to \overline{\hcoY}'$. {Therefore we must have} 
$\widetilde{\hcoY}'\simeq \widetilde{\hcoY}$.
Now extending (\ref{eq:STcommrev}), we have
\begin{equation}
\begin{matrix}
\xymatrix{\hcoY_2\ar[rr]^{/\rG(n-3,V)}\ar[d] & & \hcoY_{\UU}=\widetilde{\nT}_4\ar[rr]^{/\rG(n-3,V)}_{\Lrho\,_{\widetilde{\nT}_4}}\ar[d]_{\Lp\,_{\widetilde{\nT}_4}} & & \UU=\widetilde{\nS}_4\ar[d]_{\Lpi_{\widetilde{\nS}_4}}\\
\widetilde{\hcoY}\ar[rr]_{\Lrho_{\widetilde{\hcoY}}}& & \hcoY=\nT_4\ar[rr]_{\Lrho\,_{\nT_4}} & & \Hes=\nS_4.}
\end{matrix}
\end{equation}
Note that $\hcoY_2\to \hcoY_{\UU}$ and $\hcoY_{\UU}\to \hcoY$
are divisorial contractions.
Moreover, $\hcoY_2\to \widetilde{\hcoY}$ is also a divisorial contraction
contracting $F_{\rho}$ to $G_{\rho}$.
Therefore $\widetilde{\hcoY}\to \hcoY$ is a divisorial contraction,
and moreover its exceptional divisor $F_{\widetilde{\hcoY}}$ is 
the image of the exceptional divisor of 
$\hcoY_2\to \hcoY_{\UU}$. 
%

Finally {we show that
$\hcoY_2\to \widetilde{\hcoY}$ is the blow-up of $G_{\rho}$.
This morphism is given by forgetting the markings by 
$[V_{n-3}]$ in $\rG(n-3,V)$. 
%
But, since $G_{\rho}\simeq \mP({\ft S}^2 \sQ_{\rho}^*)$ 
(see Proposition~\ref{prop:Grho}), 
the markings by $[V_{n-2}]$ in $\rG(n-2,V)$ are retained.} 
Therefore the fiber of $\hcoY_2\to \widetilde{\hcoY}$
over a point $(q, [V_{n-2}])$ in $\mP({\ft S}^2 \sQ_{\rho}^*)$
is isomorphic to $\rG(n-3, V_{n-2})\simeq \mP^{n-3}$.
We may conclude that 
$\hcoY_2\to \widetilde{\hcoY}$ is the blow-up of $G_{\rho}$
by the same argument as in the proof of Proposition \ref{prop:tildeY}.
\end{proof}

\begin{rem}
\label{rem:blup}
In a similar way to the proof of Proposition \ref{prop:gendescr} (1),
we can show that $\hcoY_0\to \widetilde{\hcoY}$ is the blow-up along $G_{\sigma}$.
\end{rem}
By Propositions \ref{prop:conicquad} and \ref{prop:gendescr},
we have the following:

\begin{prop}
\label{prop:FYdisc}
For a $\tau$- or $\rho$-conic $q$,
$\Lrho_{\widetilde{\hcoY}}([q])$ is the point corresponding to
the quadric generated by {$\mP(V_{n-1})$'s} which $q$ parameterizes.
For a $\sigma$-plane {$\rm{P}_{V_{n-3}V_{n}}$,
$\Lrho_{\widetilde{\hcoY}}([\rm{P}_{V_{n-3}V_{n}}])$} is the point corresponding to
the rank one quadric $2\mP(V_{n})$.
In particular,
the exceptional locus $F_{\widetilde{\hcoY}}$ consists of the points 
corresponding to $\tau$- or $\rho$-conics of rank at most two
or $\sigma$-planes,
and the image of $F_{\widetilde{\hcoY}}$ coincides with
$G_{\hcoY}$.
\end{prop}

\def\FigHilbCoYs{\resizebox{10cm}{!}{\includegraphics{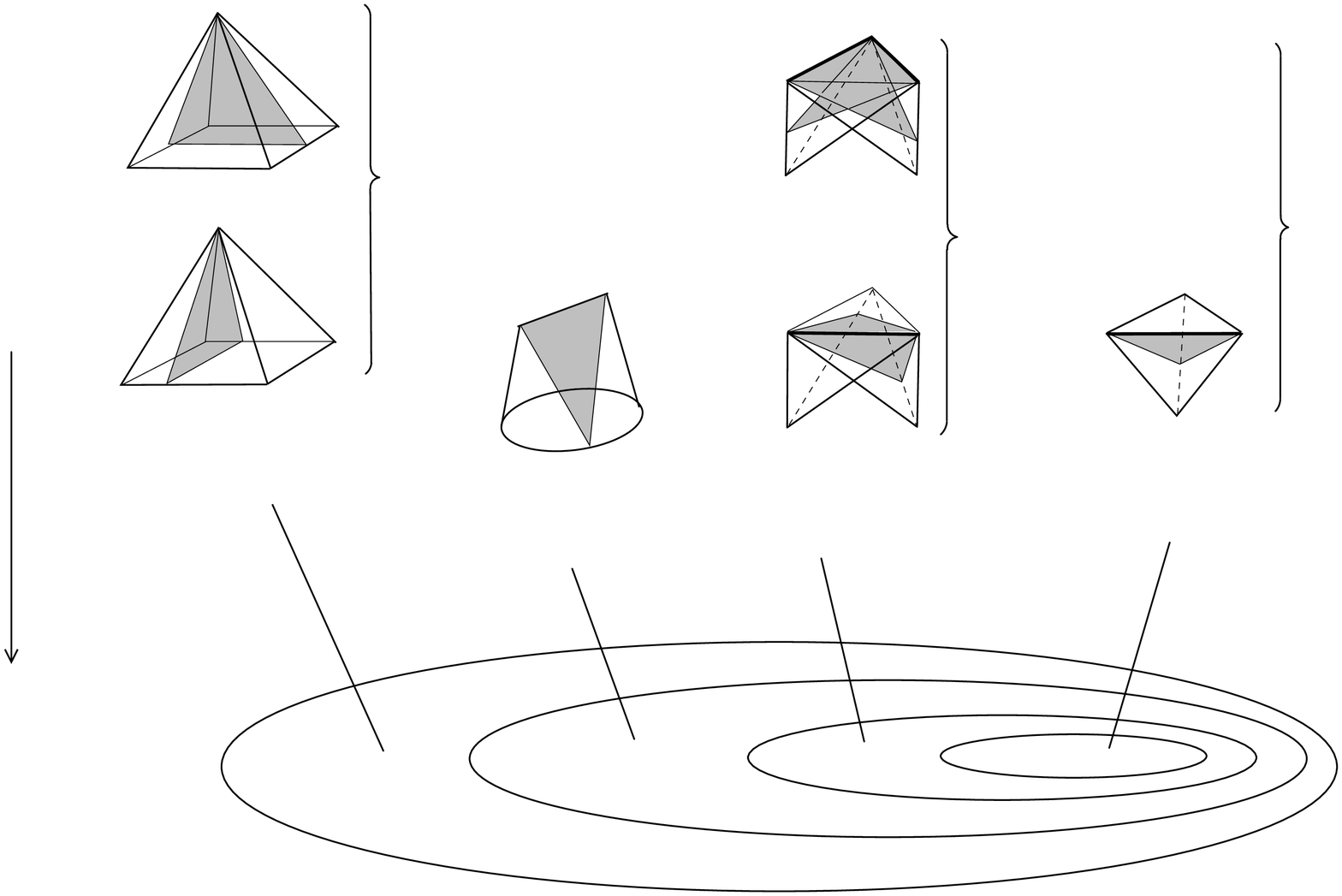}}} 
\def\upIN{\vcorr{90}{$\in$}} 
\def\xyFigCoYs{ 
\begin{xy} 
(0,0)*{\FigHilbCoYs}, 
(-49,13)*{\widetilde\hcoY}, 
(-49,-20)*{\Hes}, 
(-32,-28)*{\Hes}, 
(-10,-28)*{\nS_3}, 
( 14,-26.5)*{\nS_2=G_\hcoY}, 
(33,-26)*{\nS_1=G_\hcoY^1}, 
(-15,20)*{\tau\text{-conics}}, 
(-15,17)*{\,_{(\text{rk} \tau =3)}}, 
(-7,-3)*{\rho\text{-conic}}, 
(-7,-6)*{\,_{(\text{rk} \rho =3)}}, 
( 13, 18)*{\tau\text{-conics}}, 
( 13, 15)*{\,_{(\text{rk} \tau =2)}}, 
( 13,-1)*{\rho\text{-conics}}, 
( 13,-4)*{\,_{(\text{rk} \rho =2)}}, 
(38, 22)*{\tilde{\phi}^{-1}([Q])}, 
(38,16)*{\text{\upIN}}, 
(38,-1)*{\rho\text{-conics}}, 
(38,-4)*{\,_{(\text{rk} \rho=1)}}, 
( 56, 18)*{\text{double lines}}, 
( 52, 14)*{\text{and}}, 
( 53, 10)*{\sigma\text{-planes}}, 
(0,0)*{} \end{xy} }
\vspace{0.1cm}

\vbox{
\[
\xyFigCoYs\]
\vspace{0.0cm}
 \begin{fcaption} 
\item \textbf{Fig.2. The fibers of $\tilde{\phi}=
\Lrho\,_{\nT_4}\circ\Lrho\,_{\widetilde\hcoY}\colon \widetilde{\hcoY}\to\Hes$ 
when $n=4$.}
\end{fcaption} }


\section{{\bf Geometry of $F_{\widetilde{\hcoY}}\to G_{\hcoY}$ and flattening}}
\label{section:FY}

In this section, we determine the structure of $F_{\widetilde{\hcoY}}\to G_{\hcoY}$ and construct its flattening.

\subsection{Birational model $F^{(1)}/\mZ_2$ of $F_{\widetilde{\hcoY}}$}

From the description of the conics of rank two in Example \ref{ex:ranktwo} 
and Proposition \ref{prop:FYdisc}, we introduce the following $\mZ_{2}$-subvariety $F^{(1)}$
of $\mathrm{F}(n-2,n,V)^{\times 2}$ to study the exceptional locus $F_{\widetilde{\hcoY}} \subset \widetilde{\hcoY}$:
\begin{equation}
F^{(1)}:=\left\{ ([V_{n-2}],[V_{n-2}'];[V_{n}],[V_{n}'])\,\bigg\vert\;\begin{matrix}V_{n-2},V_{n-2}'\subset V_{n}\cap V_{n}'\\
\dim (V_{n-2}\cap V_{n-2}')\geq n-3\end{matrix}\right\},\label{eq:FsY}
\end{equation}
where $\mZ_{2}$ acts by the simultaneous exchanges $V_{n-2}\leftrightarrow V_{n-2}'$
and $V_{n}\leftrightarrow V_{n}'$. 
We set
\[
\widehat{G}:=\mP(V^{*})\times\mP(V^{*}),\, \Delta_{G}:=
\text{the diagonal of $\widehat{G}$}, 
\]
and note that the natural
projection
$F^{(1)}\to\widehat{G}$
is a $\mP^{n-2}\times \mP^{n-2}$-fibration outside $\Delta_G$.
Let $\oF{(1)}$ be the following open subset of $F^{(1)}$: 
\begin{equation}
\oF{(1)}:=\left\{ ([V_{n-2}],[V_{n-2}'];[V_{n}],[V_{n}'])\,\bigg\vert\;
V_{n}\not =V_{n}'
\right\} \subset F^{(1)}.\end{equation}
\begin{prop}
\label{prop:begin}
The natural map
$\oF{(1)}/\mZ_2\to (\widehat{G}\setminus \Delta_{G})/\mZ_2$
is isomorphic to
$F_{\widetilde{\hcoY}}\setminus \Lrho\,_{\widetilde{\hcoY}}^{-1}(G^1_{\hcoY})
\to G_{\hcoY}\setminus G^1_{\hcoY}$.
In particular,
$F_{\widetilde{\hcoY}}\setminus \Lrho\,_{\widetilde{\hcoY}}^{-1}(G^1_{\hcoY})
\to G_{\hcoY}\setminus G^1_{\hcoY}$
is a $\mP^{n-2}\times \mP^{n-2}$-fibration.
\end{prop}

\begin{proof}
First note that
$\widehat{G}/\mZ_2\simeq G_{\hcoY}$, 
$\Delta_G/\mZ_2\simeq G^1_{\hcoY}$ and hence 
$(\widehat{G}\setminus \Delta_{G})/\mZ_2\simeq 
G_{\hcoY}\setminus G^1_{\hcoY}$.

Let us note that 
$\oF{(1)}/\mZ_2$ parameterizes line pairs in $\mathrm{G}(n-1,n+1)$
which are reducible conics of rank two and not on $\sigma$-planes (see  
Example \ref{ex:ranktwo} for explicit descriptions). 
%
%
Therefore we have 
the unique injective morphism  $\oF{(1)}/\mZ_2\to \hcoY_0$ which is induced by
the universality of the Hilbert scheme $\hcoY_0$.
By Proposition \ref{prop:FYdisc},
the image of $\oF{(1)}/\mZ_2$ coincides with $F_{\widetilde{\hcoY}}\setminus \Lrho\,_{\widetilde{\hcoY}}^{-1}(G^1_{\hcoY})$,
and the map $\oF{(1)}/\mZ_2\to F_{\widetilde{\hcoY}}\setminus \Lrho\,_{\widetilde{\hcoY}}^{-1}(G^1_{\hcoY})$ induces the following commutative diagram:
\[
\xymatrix{\oF{(1)}/\mZ_2\ar[r]\ar[d] & F_{\widetilde{\hcoY}}\setminus \Lrho\,_{\widetilde{\hcoY}}^{-1}(G^1_{\hcoY})\ar[d]\\
(\widehat{G}\setminus \Delta_{G})/\mZ_2\ar[r]^{\simeq} & {G_{\hcoY}\setminus G^1_{\hcoY}}.}
\]
Note that 
$F_{\widetilde{\hcoY}}\setminus \Lrho\,_{\widetilde{\hcoY}}^{-1}(G^1_{\hcoY})$
is normal. Indeed, $F_{\widetilde{\hcoY}}$ satisfies the $S_2$ condition
since it is a divisor on a smooth variety.
It also satisfies the $R_1$ condition since, by considering the $\SL(V)$-action, its singular locus is at most the locus of $\rho$-conics of rank two which is codimension $n-2\geq 2$ in $F_{\widetilde{\hcoY}}$ if $n\geq 4$ (resp.~it is smooth
if $n=3$ by Proposition \ref{prop:n=3fib} (5)). Hence  
$F_{\widetilde{\hcoY}}\setminus \Lrho\,_{\widetilde{\hcoY}}^{-1}(G^1_{\hcoY})$
is normal.
Therefore the bijective morphism 
$\oF{(1)}/\mZ_2\to F_{\widetilde{\hcoY}}\setminus \Lrho\,_{\widetilde{\hcoY}}^{-1}(G^1_{\hcoY})$ is an isomorphism by the Zariski main theorem.

Finally, 
the natural map $\oF{(1)}\to \widehat{G}$ is obviously a $\mP^{n-2}\times \mP^{n-2}$-fibration, and then so is 
$\oF{(1)}/\mZ_2\to (\widehat{G}\setminus \Delta_{G})/\mZ_2$
since the $\mZ_2$-action interchanges the fibers over
$(x,y)$ and $(y,x)$ in $\widehat{G}\setminus \Delta_G$.
\end{proof}
%
%
The following corollary will be used in the companion paper \cite{DerSym}.

\begin{cor}
\label{cor:Steinn=any}
It holds that 
\begin{equation}
\label{eq:adjn=n}
K_{\widetilde{\hcoY}}=\Lrho_{\widetilde{\hcoY}}^{\;*}K_{\hcoY}+(n-2)F_{\widetilde{\hcoY}}. 
\end{equation}
\end{cor}

\begin{proof}
Let $a$ be the discrepancy of $F_{\widetilde{\hcoY}}$. We show $a=n-2$.
Let $\Gamma\simeq \mP^{n-2}\times \mP^{n-2}$ be a fiber of $F_{\widetilde{\hcoY}}\to G_{\hcoY}$ outside the diagonal of $G_{\hcoY}$ and 
$l$ a line in a ruling of $\Gamma\simeq \mP^{n-2}\times \mP^{n-2}$.
Since $K_{\Gamma}\cdot l=-(n-1)$ and $K_{\Gamma}=K_{F_{\widetilde{\hcoY}}}|_{\Gamma}=(a+1)F_{\widetilde{\hcoY}}|_{\Gamma}$, we have 
$(a+1)F_{\widetilde{\hcoY}}\cdot l=-(n-1)$.
Therefore we have only to show $F_{\widetilde{\hcoY}}\cdot l=-1$.
For this we take $l$ so that $l\cap G_{\rho}\not =\emptyset$.
Now we consider the diagram (\ref{eq:STcommrev}).
Since $\Gamma\cap G_{\rho}$ is the diagonal by Proposition \ref{prop:begin},
the strict transform $l'$ is a ruling of a fiber $\simeq \mP^1\times \mP^1$
of $\hcoY_2\to \hcoY_{\UU}$.
Therefore $F'_{\widetilde{\hcoY}}\cdot l'=-1$ where
$F'_{\widetilde{\hcoY}}$ is the strict transform of $F_{\widetilde{\hcoY}}$.
Since $G_{\rho}\not \subset F_{\widetilde{\hcoY}}$,
we have $F_{\widetilde{\hcoY}}\cdot l=F'_{\widetilde{\hcoY}}\cdot l'=-1$
as desired.
\end{proof}

By Proposition \ref{prop:begin},
we have a birational map
$F^{(1)}/\mZ_2\dashrightarrow F_{\widetilde{\hcoY}}$
extending the isomorphism
$\oF{(1)}/\mZ_2\simeq
F_{\widetilde{\hcoY}}\setminus \Lrho\,_{\widetilde{\hcoY}}^{-1}(G^1_{\hcoY})$.
In the sequel of this section,
we will give an explicit description of this birational map
using the minimal model theory,
which leads to a precise description of $F_{\widetilde{\hcoY}}$.
We summarize our description in the following diagram:

\begin{equation}
\begin{matrix}\xymatrix{ & {F}^{(3)}\ar[dl]\ar[dr]\\
{F}^{(2)}\ar[ddr]_{\;_{\text{\ensuremath{\mP^{n-2}\times\mP^{n-2}}-fib.}}}\ar@{-->}[rr]^{\;_{\text{(anti-)flip}}}\ar[dr] &  & {F}^{(4)}\ar[d]^{\;_{\text{div. cont.}}}\ar[dl]\\
 & {F}^{(1)}\ar[dr] & \widehat{F}\ar[d]\ar[r]_{\;_{\text{\ensuremath{\mZ_{2}}-quot.}}} & F_{\widetilde{\hcoY}}\ar[d]^{\Lrho_{\widetilde{\hcoY}}|_{F_{\widetilde{\hcoY}}}}\\
 & \widehat{G}'\ar[r]_{_{\text{{diag.blow up}}}} & \widehat{G}\ar[r]^{\;_{\text{\ensuremath{\mZ_{2}}-quot.}}} & {G}_{{\hcoY}}.}
\end{matrix}\label{eq:house}\end{equation}

\subsection{Small resolution and flip}

First we determine the singularities of $F^{(1)}$.
\begin{prop}
\label{pro:F2-proof-appendix-1} $F^{(1)}$ is singular along the
diagonal set
\begin{equation}
\Delta_{{F}^{(1)}}:=\{([V_{n-2}],[V_{n-2}];[V_{n}],[V_{n}])\mid V_{n-2}\subset V_{n}\}\simeq\mathrm{F}(n-2,n,V)\subset{F}^{(1)}.\label{eq:DelF}\end{equation}
 The singularity at each point on $\Delta_{F^{(1)}}$ is isomorphic
to the cone over the Segre variety $\mP^{1}\times\mP^{n-2}$. \end{prop}

\begin{proof}
Recall that $F^{(1)}$ is a subvariety of
$\rF(n-2,n,V)^{\times 2}$ and consider the first projection 
${F}^{(1)} \to \rF(n-2,n,V)$.
Let $\Gamma$ be a fiber of this projection over a point
$([V_{n-2}];[V_n])\in \rF(n-2,n,V)$. 
We consider $\Gamma$ as a subvariety of $\rF(n-2,n,V)$
parameterizing $V'_{n-2}\subset V'_n$
such that $V'_{n-2}\subset V_n$, $V_{n-2}\subset V'_n$ and $\dim (V_{n-2}\cap V'_{n-2})\geq n-3$.
To describe $\Gamma$, we choose a basis $\{{\bf e}_1,\dots,{\bf e}_{n+1}\}$
of $V$ so that
$V_{n-2}=\langle {\bf e}_1,\dots, {\bf e}_{n-2}\rangle$ and
$V_n=\langle {\bf e}_1,\dots, {\bf e}_n\rangle$.
An $(n-2)$-dimensional subspace $V'_{n-2}$ of $V_n$
with $\dim (V_{n-2}\cap V'_{n-2})\geq n-3$
is spanned by $n-3$ vectors in $V_{n-2}$ and a vector 
$b_1{\bf e}_1+\dots+b_n{\bf e}_n$ in $V_n$.
We arrange these vectors into an $(n-2)\times n$ matrix as
\begin{equation}
\label{eq:V'n-2}
\begin{pmatrix}
A & \bf{0} & \bf{0}\\
b_1\dots b_{n-2}& b_{n-1} & b_n
\end{pmatrix},
\end{equation}
where the row vectors of $A$ represents the $n-3$ vectors in $V_{n-2}$. 
  %
  %
We denote by $q_{ij}$ the Pl\"ucker coordinate of 
$V'_{n-2}$
given by the $(n-2)\times (n-2)$ minors of (\ref{eq:V'n-2}) with the 
$i$- and $j$-th columns omitted. 
  %
Denote by $x_1, \dots,x_{n+1}$, and $y_1,\dots,y_{n+1}$ 
the homogeneous coordinates of $\mP(V)$ and $\mP(V^*)$, respectively, 
associated to the basis $\{{\bf e}_1,\dots,{\bf e}_{n+1}\}$
and its dual basis.
An $n$-dimensional subspace $V'_n$ of $V$
containing $V_{n-2}$ is of the form
$\{c_{n-1} x_{n-1}+c_n x_n+c_{n+1} x_{n+1}=0\}$,
where we consider $(0,\dots,0, c_{n-1},c_n, c_{n+1})$ as 
the coordinates of $[V'_n]$ in $V^*$.
Therefore $V'_n$ contains $V'_{n-2}$ if and only if
$c_{n-1} b_{n-1}+c_n b_n=0$.
From the above considerations,
we can deduce that 
\[
\Gamma=
\left\{(q_{ij};y_1,\dots, y_{n+1}) \; \bigg| \;  \begin{array}{l} 
q_{ij}=0\, \text{for $1\leq i,j \leq n-2$}, \\
\rank \begin{pmatrix}
q_{1\,n} & q_{2\,n} & ...  & q_{n-2\,n} & -y_n \\
q_{1\,n-1} & q_{2\,n-1}& ... & q_{n-2\,n-1} & y_{n-1} 
\end{pmatrix}\leq 1 \end{array} \right\}.
\]
From this, it is easy to see the assertion.
\end{proof}
The cone over $\mP^1\times \mP^{n-2}$ has exactly two small resolutions; 
one of which has a $\mP^1$ as the exceptional set and
another has a $\mP^{n-2}$ as the exceptional set.
Corresponding to these, we have two small resolutions of $F^{(1)}$. 
One of them is given by the following variety $F^{(2)}$:
\[
\begin{aligned}
F^{(2)} & :=\rF(n-2,n-1,n,V)\times_{\rG(n-1,V)} \rF(n-2,n-1,n,V) \\
&=
\left\{ ([V_{n-2}],[V_{n-2}'];[V_{n-1}];[V_{n}],[V_{n}'])\,\bigg\vert\; V_{n-2},V_{n-2}'\subset V_{n-1}\subset V_{n}\cap V_{n}'\right\}.\end{aligned}
\]
We set
\[
\begin{aligned}
\hat{G}' &:=
\rF(n-1,n,V)\times_{\rG(n-1,V)} \rF(n-1,n,V) \\
&=\left\{ ([V_{n-1}];[V_{n}],[V_{n}'])\mid V_{n-1}\subset V_{n}\cap V_{n}'\right\}.
\end{aligned}
\]
$F^{(2)}$ has a $\mP^{n-2}\times \mP^{n-2}$-fibration $F^{(2)}\to\hat{G}'$.
We note that there is a morphism $\widehat{G}'\to\widehat{G}=\mP(V^{*})\times\mP(V^{*})$
defined by $([V_{n-1}];[V_{n}],[V_{n}'])\mapsto([V_{n}],[V_{n}'])$,
which is nothing but the blow-up of $\widehat{G}$ along the diagonal
$\Delta_G$.

\begin{prop}
\label{pro:F2-proof-appendix-2}\noindent $(1)$ $F^{(2)}$ is smooth.
The natural projection $F^{(2)}\to F^{(1)}$ is a small resolution 
with every non-trivial fiber $\gamma$
being isomorphic to $\mP^{1}$.

\noindent $(2)$
The normal bundle $\sN_{\gamma/F^{(2)}}$ of a non-trivial fiber $\gamma$ of $F^{(2)}\to F^{(1)}$
is isomorphic to
$\sO_{\mP^1}(-1)^{\oplus n-1}\oplus \sO_{\mP^1}^{\oplus 3n-4}$. 

\noindent $(3)$ There is another small resolution $F^{(4)}\to F^{(1)}$,
whose non-trivial fiber is isomorphic to $\mP^{n-2}$. $F^{(2)}$ and
$F^{(4)}$ fit into the following diagram$:$\def\Fdiagram{ 
\begin{xy}
(0,0)*+{F^{(3)}}="Fiii",
(-13,-10)*+{F^{(2)}}="Fii",
(13,-10)*+{F^{(4)}}="Fiv",
(0,-20)*+{F^{(1)},}="Fi",
\ar_p "Fiii";"Fii",
\ar "Fiii";"Fiv",
\ar "Fii";"Fi",
\ar "Fiv";"Fi",
\end{xy}}\begin{equation}
\begin{matrix}{\Fdiagram}\end{matrix}\label{eq:F-diagram}\end{equation}
where $p\colon F^{(3)}\to F^{(2)}$ is the blow-up along the exceptional locus
of $F^{(2)}\to F^{(1)}$, and $F^{(3)}\to F^{(4)}$ is the contraction
of the exceptional divisor of the blow-up $F^{(3)}\to F^{(2)}$ in
another direction.\end{prop}

\begin{proof}
(1) $F^{(2)}$ is smooth since it has a $\mP^{n-2}\times \mP^{n-2}$-fibration
over a smooth variety $\widehat{G}'$.
We show that ${F}^{(2)}\to{F}^{(1)}$
is a small resolution. For a point 
\[
([V_{n-2}],[V_{n-2}'];[V_{n-1}];[V_{n}],[V_{n}'])\in F^{(2)},
\]
$V_{n-1}=V_{n-2}+V_{n-2}'$ holds when $V_{n-2}\not=V_{n-2}'$, and also $V_{n-1}=V_{n}\cap V_{n}'$
when $V_{n}\not=V_{n}'$. Hence the morphism ${F}^{(2)}\to{F}^{(1)}$
is isomorphic outside the diagonal set $\Delta_{{F}^{(1)}}$.
The fiber of ${F}^{(2)}\to{F}^{(1)}$ over a point $([V_{n-2}],[V_{n-2}];[V_{n}],[V_{n}])\in\Delta_{{F}^{(1)}}$
is \[
\{([V_{n-2}],[V_{n-2}];[V_{n-1}];[V_{n}],[V_{n}])\mid[V_{n-1}]\in\mathrm{G}(n-1,V),V_{n-2}\subset V_{n-1}\subset V_{n}\}\simeq\mP^{1}.\]
We calculate the dimension of the exceptional set of ${F}^{(2)}\to{F}^{(1)}$
as $\dim\Delta_{{F}^{(1)}}+1=3n-3$. Hence ${F}^{(2)}\to{F}^{(1)}$ is
small since $\dim F^{(1)}=4n-4$. 

(2) The two small resolutions of $F^{(1)}$ locally coincide with
those of the cone over $\mP^1\times \mP^{n-3}$.
Therefore the description of the normal bundle of $\gamma$ follows
by that of a non-trivial fiber of the small resolutions of 
 the cone over $\mP^1\times \mP^{n-3}$.

(3) Let $D$ be the $p$-exceptional divisor.
Then any fiber of $D$ is $\mP^1\times \mP^{n-2}$ by
Proposition \ref{pro:F2-proof-appendix-1}. 
Let $\gamma\simeq \mP^1$ be a fiber of $F^{(2)}\to F^{(1)}$.
Since $K_{F^{(2)}}\cdot \gamma=n-3$ by (2), we see that $p^*K_{F^{(2)}}+(n-3)D$ is nef
and $(p^*K_{F^{(2)}}+(n-3)D)-K_{F^{(3)}}=-D$ is nef and big over $F^{(1)}$,
$p^*K_{F^{(2)}}+(n-3)D$ is semi-ample over $F^{(1)}$
by the Kawamata-Shokurov base point free theorem. Since $p^*K_{F^{(2)}}+D$ is 
numerically trivial for any fiber $\gamma'$ of 
$\mP^1\times \mP^{n-2}\to \mP^{n-2}$,
the birational morphism $F^{(3)}\to F^{(4)}$ over $F^{(1)}$ defined by a sufficiently high multiple of
$p^*K_{F^{(2)}}+(n-3)D$
contracts $\gamma'$. 
Since $-K_{F^{(3)}}\cdot \gamma'=1$ by (3), $F^{(4)}$ is smooth
and $F^{(3)}\to F^{(4)}$ is the blow-up along the image of $D$
(cf.~the proof of Proposition \ref{prop:tildeY} in case $n=3$).
\end{proof}

\subsection{Divisorial contraction}

Let $D^{(2)}$ be the inverse image in $F^{(2)}$ of the diagonal $\Delta_{G}$
of $\widehat{G}$, namely,
\[
D^{(2)}:=\rF(n-2,n-1,n,V)\times_{\rF(n-1,n,V)} \rF(n-2,n-1,n,V).
\] 
We denote by $D^{(1)}$ the image on $F^{(1)}$ of
$D^{(2)}$.
It is easy to verify the following properties:
\begin{lem}
\begin{enumerate}[$(1)$]
\item
$D^{(2)}$ is a prime divisor of $F^{(2)}$.
\item
The projection $D^{(2)}\to \rF(n-1,n,V)$ is a $\mP^{n-2}\times \mP^{n-2}$-fibration.
\item All the non-trivial fibers of
$F^{(2)}\to F^{(1)}$ are contained in $D^{(2)}$,
namely, they coincide with the fibers of
$D^{(2)}\to D^{(1)}$.
Therefore
$D^{(2)}\to D^{(1)}$ is birational with
any non-trivial fiber being a copy of $\mP^{1}$.
\end{enumerate}
\end{lem}

Now we set 
\begin{equation}
\label{eqn:FsY''}
\begin{aligned}
{D}^{(4)} & :=
\rF(n-3,n-2,n,V)\times_{\rF(n-3,n,V)} \rF(n-3,n-2,n,V) \\ 
&=
\left\{
([V_{n-3}];[V_{n-2}],[V'_{n-2}];[V_n],[V_n]) \, \big| \,
\begin{array}{l}
V_{n-3}\subset V_{n-2}\cap V'_{n-2}, \\
V_{n-2}, V'_{n-2} \subset V_n 
\end{array}
\right\}.
\end{aligned}
\end{equation}
Then we can deduce easily the following commutative diagram:
\begin{equation}
\label{eq:smallhouse}
\begin{matrix}
\xymatrix{{D}^{(2)}\ar[d]_{\;_{\text{$\mP^{n-2}\times \mP^{n-2}$-fib.}}}
\ar@{-->}[rr]\ar[dr] & &
{D}^{(4)}\ar[d]^{\;_{\text{$\mP^2\times \mP^2$-fib.}}}\ar[dl]& & \\ 
\rF(n-1,n,V)\ar[dr] & {D}^{(1)}\ar[d] & 
\rF(n-3,n,V)\ar[dl]\\ 
& \Delta_{G}, &}
\end{matrix}
\end{equation}
where
$D^{(4)}\to D^{(1)}$ is birational with
any non-trivial fiber being a copy of $\mP^{n-3}$.

\begin{lem}
\label{cla:D_4}
${D}^{(4)}$
is the strict transform on
${F}^{(4)}$ of
${D}^{(2)}$, and
the diagram
$(\ref{eq:smallhouse})$
follows from the restriction of 
$(\ref{eq:F-diagram})$.
\end{lem}

\begin{proof}
In a similar way to the case of $F^{(1)}$,
we may show that $D^{(1)}$
is singular along 
$\Delta_{{F}^{(1)}}$, and 
the singularity at each point on $\Delta_{F^{(1)}}$ is isomorphic
to the cone over the Segre variety $\mP^{1}\times\mP^{n-3}$ if $n\geq 4$
($D^{(1)}$ is smooth if $n=3$).
Moreover, 
by restricting  (\ref{eq:F-diagram}) to $D^{(1)}$ and its strict transforms, 
we have a similar diagram for $D^{(1)}$. 
In particular, the  restriction of 
(\ref{eq:F-diagram}) gives two small resolutions of $D^{(1)}$ if $n\geq 4$
(for $n=3$, the restriction of $F^{(2)}\to F^{(1)}$ is 
the blow-up along $\Delta_{{F}^{(1)}}$, and  
the restriction of $F^{(4)}\to F^{(1)}$ is
an isomorphism).
Let us define 
\begin{equation}
\label{eqn:FsY'''}
\begin{aligned}
{D}^{(3)} & := 
\rF(n\text{--}3,n\text{--}2,n\text{--}1,n,V)\times_{\rF(n\text{--}3,n\text{--}1,n,V)} 
\rF(n\text{--}3,n\text{--}2,n\text{--}1,n,V)\\
&=\left\{
([V_{n-3}];[V_{n-2}],[V'_{n-2}];[V_{n-1}];[V_n],[V_n]) \,\big|\,
\begin{array}{l}
V_{n-3}\subset V_{n-2},  \\ 
V'_{n-2} \subset V_{n-1} \subset V_n.
\end{array}
\right\} \\
\end{aligned}
\end{equation}
Then $D^{(1)},\dots, D^{(4)}$ fit into the following diagram
with the natural projections:
\def\Fdiagram'{ 
\begin{xy}
(0,0)*+{D^{(3)}}="Fiii",
(-13,-10)*+{D^{(2)}}="Fii",
(13,-10)*+{D^{(4)}}="Fiv",
(0,-20)*+{D^{(1)}.}="Fi",
\ar "Fiii";"Fii",
\ar "Fiii";"Fiv",
\ar "Fii";"Fi",
\ar "Fiv";"Fi",
\end{xy}}\begin{equation}
\begin{matrix}{\Fdiagram'}\end{matrix}\label{eq:F-diagram'}\end{equation}
By construction,
it is easy to see that $D^{(2)}\to D^{(1)}$ and $D^{(4)}\to D^{(1)}$ are
two small resolutions of $D^{(1)}$ if $n\geq 4$ (for $n=3$, 
$D^{(2)}\to D^{(1)}$ is the blow-up along 
$\Delta_{{F}^{(1)}}$ and $D^{(4)}\to D^{(1)}$ is an isomorphism).
Therefore the diagram 
(\ref{eq:F-diagram'}) coincides with the restriction of 
(\ref{eq:F-diagram}) considered above, and hence the assertions follow.
\end{proof}

\begin{prop}
\label{pro:Div-cont-Flat}
There exists a divisorial contraction 
${F}^{(4)}\to \widehat{F}$ over
$\widehat{G}$ which  
contracts the strict transform ${D}^{(4)}$ 
of ${D}^{(1)}$ to
the locus isomorphic to the flag variety $\mathrm{F}(n-3,n,V)$.
The discrepancy of 
${D}^{(4)}$ is two.
\end{prop}
\begin{proof}
Let $\Delta'_{\mP}$ be the inverse image 
in $\widehat{G}'$ of $\Delta_{G}$.
Note that $\Delta'_{\mP}\simeq \rF(n-1,n,V)$.
Let $\Gamma$ be a fiber of the $\mP^{n-2}\times \mP^{n-2}$-fibration 
${D}^{(2)}\to
\Delta'_{\mP}$.
Then $\Gamma$ intersects the flipping locus of $F^{(2)}\dashrightarrow F^{(4)}$
along the diagonal transversally.
Take a line $r\subset \mP^{n-2}\times \mP^{n-2}$ 
which is contained in a fiber of the second projection $\Gamma\to \mP^{n-2}$
and intersects the flipping locus.
$r$ is of the form with some fixed $V_{n-3}\subset V'_{n-2}\subset V_{n-1}\subset V_n$ and moving $V_{n-2}$ as follows:
\[
r:=\{([V_{n-2}], [V'_{n-2}];[V_{n-1}];[V_n],[V_n])\mid V_{n-3}\subset V_{n-2}\subset V_{n-1}\}.
\]
Then its strict transform $r'$ on 
${D}^{(4)}$
is contracted by the morphism 
${D}^{(4)}\to \rF(n-3,n,V)$.
Since ${F}^{(2)}\to
\widehat{G}'$ is a $\mP^{n-2}\times \mP^{n-2}$-fibration and
${D}^{(2)}$ is the pull-back of
$\Delta'_{\mP}$,
we see that $K_{{F}^{(2)}}\cdot r=-(n-1)$
and
${D}^{(2)}\cdot r=0$.
By the standard calculations of the changes of the intersection numbers 
by the flip,
we have 
$K_{{F}^{(4)}}\cdot r'=-(n-1)+(n-3)=-2$
and
${D}^{(4)}\cdot r'=0-1=-1$.
These equalities of the intersection numbers still hold
for any line in a ruling of a fiber of   
${D}^{(4)}\to \rF(n-3,n,V)$.

We show $-K_{F^{(4)}}+2D^{(4)}$ is relatively nef 
over $\widehat{G}$. Let $\gamma$ be a curve on 
$F^{(4)}$ which is contracted to a point $t$ on $\widehat{G}$.
If $t\not \in \Delta_{G}$,
then $(-K_{F^{(4)}}+2D^{(4)})\cdot \gamma>0$ since
$D^{(4)}\cap \gamma=\emptyset$ and $F^{(4)}\to \widehat{G}$
is a $\mP^{n-2}\times \mP^{n-2}$ fibration outside $\Delta_{G}$.
If $t\in \Delta_{G}$ and $\gamma$ is an exceptional curve of $F^{(4)}\to F^{(1)}$, then
$(-K_{F^{(4)}}+2D^{(4)})\cdot \gamma>0$
since $-K_{F^{(4)}}\cdot \gamma>0$ and $D^{(4)}\cdot \gamma>0$.
In the remaining cases, $t\in \Delta_{G}$ and $\gamma\subset D^{(4)}$.
Therefore we have only to consider the relative nefness of 
$(-K_{F^{(4)}}+2D^{(4)})|_{D^{(4)}}$ over $\Delta_{G}$. Now we take as $\gamma$
any line in a ruling of a fiber of   
${D}^{(4)}\to \rF(n-3,n,V)$. As we see in the first paragraph, 
$(-K_{F^{(4)}}+2D^{(4)})\cdot \gamma=0$.
Therefore $(-K_{F^{(4)}}+2D^{(4)})|_{D^{(4)}}$ is the pull-back of 
some divisor $D_F$ on $\mathrm{F}(n-3,n,V)$. It suffices to show 
$D_F$ is relatively nef over $\Delta_{G}$, which is true
since an exceptional curve of $D^{(4)}\to D^{(1)}$
is positive for $(-K_{F^{(4)}}+2D^{(4)})|_{D^{(4)}}$ as above and
is mapped to a curve on a fiber of  
$\rF(n-3,n,V)\to \Delta_{G}$.
Therefore 
$-K_{F^{(4)}}+2D^{(4)}$ is relatively nef 
over $\widehat{G}$.

Moreover, by this argument, we see that
$(-K_{F^{(4)}}+2D^{(4)})^{\perp}\cap \overline{\mathrm{NE}}(F^{(4)}/\widehat{G})$
is generated by the numerical class of the curves on fibers of
$D^{(4)}\to \rF(n-3,n,V)$. In particular,
$(-K_{F^{(4)}}+2D^{(4)})^{\perp}\cap \overline{\mathrm{NE}}(F^{(4)}/\widehat{G})\subset (K_{F^{(4)}})^{<0}$. Therefore, by Mori theory, 
there exists a contraction associated to
this extremal face, which is nothing but the divisorial contraction
contracting $D^{(4)}$ to $\rF(n-3,n,V)$
such that 
$-K_{F^{(4)}}+2D^{(4)}$ is the pull-back of $-K_{\widehat{F}}$.
Thus 
the discrepancy of 
${D}^{(4)}$ is two.
\end{proof}

\subsection{$\mZ_2$-quotient}
All the above constructions are $\mZ_2$-equivariant, 
hence we can take $\mZ_2$-quotient $\widehat{F}/\mZ_2$.
Comparing the morphisms
$a\colon F_{\widetilde{\hcoY}}\to G_{{\hcoY}}$
and
$b\colon \widehat{F}/\mZ_2
\to G_{{\hcoY}}$, we obtain

\begin{prop}
\label{prop:coincides}
$\widehat{F}/\mZ_2\simeq F_{\widetilde{\hcoY}}$ over $G_{\hcoY}$.
\end{prop}

\begin{lem}
\label{lem:dim}
The fiber of $F_{\widetilde{\hcoY}}\to G_{\hcoY}$
at any point of $G^1_{\hcoY}$
is of dimension at most $3n-6$.
In particular, codimension of the inverse image in 
$F_{\widetilde{\hcoY}}$ of $G^1_{\hcoY}$ is at least two. 
\end{lem}

\begin{proof}
We consider the diagram (\ref{eq:STcommrev}).
By Proposition \ref{prop:n=3fib} (5),
the fiber of $\hcoY_2\to \hcoY_{\UU}$ over
a rank one point in a fiber of $\hcoY_{\UU}\to \rG(n-3,V)$
is isomorphic to $\mP(1^3,2)$.
The fiber of $\hcoY_{\UU}\to \hcoY$ over a rank one point is
isomorphic to that of $\UU\to \Hes$ over a rank one point $[2V_n]\in \nS_1$, 
and hence is a copy of $\rG(n-3,V_n)$.
Therefore, 
the fiber of $F_{\widetilde{\hcoY}}\to G_{\hcoY}$
at any point of $G^1_{\hcoY}$
is of dimension at most $3+3(n-3)=3n-6$.
\end{proof}

\begin{proof}[{\bf Proof of Proposition $\ref{prop:coincides}$}]
Note that the morphisms $a$ and $b$ are isomorphic outside $G^1_{\hcoY}$
by Proposition \ref{prop:begin}.
Therefore,
by \cite[Lem.~5.5]{Tk} for example,
it suffices to check the following properties:
\begin{enumerate}
\item
The inverse images of $G^1_{\hcoY}$
by the morphisms $a$ and $b$ are of codimension at least two.
\item
Both $F_{\widetilde{\hcoY}}$ and 
$\widehat{F}/\mZ_2$
are normal.
\item
$-K_{F_{\widetilde{\hcoY}}}$ and
$-K_{\widehat{F}/\mZ_2}$
are $\mQ$-Cartier.
\item
$-K_{F_{\widetilde{\hcoY}}}$ is $a$-ample and
$-K_{\widehat{F}/\mZ_2}$ is $b$-ample.
\end{enumerate}
We show these in order. 

\noindent (1) The inverse image of $G^1_{\hcoY}$ by the morphism
$a$  
has codimension at least two
in $F_{\widetilde{\hcoY}}$ by Lemma \ref{lem:dim}
and 
the inverse image of $G^1_{\hcoY}$ by the morphism
$b$  
has codimension two
in $\widehat{F}/\mZ_2$ by the construction of
$\widehat{F}/\mZ_2$.

\noindent
(2) The variety $F_{\widetilde{\hcoY}}$ is normal.
Indeed, it satisfies the $S_2$ condition 
since it is a Cartier divisor on 
a smooth variety. It satisfies the $R_1$ condition
since it is a $\mP^{n-2}\times \mP^{n-2}$-fibration
outside the locus of codimension at least two by Proposition
\ref{prop:begin} and Lemma \ref{lem:dim}.
We see that the variety $\widehat{F}/\mZ_2$ is normal
by its explicit construction as above.
 
\noindent
(3), (4) The divisor $-K_{F_{\widetilde{\hcoY}}}$ is $\mQ$-Cartier
since ${F_{\widetilde{\hcoY}}}$ is a divisor on 
the smooth variety $\widetilde{\hcoY}$.
We see that $-K_{F_{\widetilde{\hcoY}}}$ is $a$-ample
since the relative Picard number $\rho(\widetilde{\hcoY}/\hcoY)$ is one
and $a$ is generically a $\mP^{n-2}\times \mP^{n-2}$-fibration.

Arguments for the morphism $b$ are similar.  Let us first show that 
$-K_{\widehat{F}/\mZ_2}$ is
$\mQ$-Cartier.
Indeed, by Lemma \ref{pro:Div-cont-Flat},
the discrepancy of ${D}^{(4)}$ is two. 
Then, by the Kawamata-Shokurov base point free theorem,
$-K_{{F}^{(4)}}-
2{D}^{(4)}$ is the pull-back of
a Cartier divisor on 
$\widehat{F}$,
which turns out to be the anti-canonical divisor 
$-K_{\widehat{F}}$.
Thus 
$-K_{\widehat{F}/\mZ_2}$ is
$\mQ$-Cartier.

To show 
$-K_{\widehat{F}/\mZ_2}$ is $b$-ample,
it suffices to see 
the relative Picard number $\rho((\widehat{F}/\mZ_2)/G_{{\hcoY}})$ is one because $b$ is generically a $\mP^{n-2}\times \mP^{n-2}$-fibration.
We compute $\rho((\widehat{F}/\mZ_2)/G_{{\hcoY}})$
using the above construction. 
The relative Picard number $\rho({F}^{(2)}/\widehat{G}')$ is 
two since ${F}^{(2)}
\to \widehat{G}'$ is  
a $\mP^{n-2}\times \mP^{n-2}$-fibration
and it is easy to see that
it is the composite of two $\mP^{n-2}$-fibrations.
Moreover we have 
$\rho^{\mZ_2}({F}^{(2)}/
\widehat{G}')=1$
since
rulings 
in two directions
of a fiber $\mP^{n-2}\times \mP^{n-2}$ of
${F}^{(2)}
\to \widehat{G}'$ are exchanged
by the $\mZ_2$-action.
Therefore
$\rho^{\mZ_2}({F}^{(2)})=3$ since 
$\rho^{\mZ_2}(\widehat{G}')=2$.
It holds that  
$\rho^{\mZ_2}({F}^{(4)})=3$ since
the flip preserves the Picard number
and the flip is $\mZ_2$-equivariant.
Since a divisorial contraction drops 
the Picard number at least by one,
we have
$\rho^{\mZ_2}(\widehat{F})\leq 2$.
Now we see that
$\rho((\widehat{F}/\mZ_2)/G_{{\hcoY}})$ is one
since 
$\rho(G_{{\hcoY}})=1$
and
the morphism $\widehat{F}/\mZ_2\to G_{{\hcoY}}$
is non-trivial.
Therefore we conclude $-K_{\widehat{F}/\mZ_2}$ is $b$-ample.
\end{proof}

\subsection{Flattening $F^{(3)}\to\widehat{G}'$ \label{sub:Flatenning-F-G}}

We describe the fibers of $F_{\widetilde{\hcoY}}\to G_{\hcoY}$
in the diagram (\ref{eq:house}).

\begin{prop}
\label{lem:inverse-rk1} 

There is a birational morphism $\mP(\sO_{\rG(n-2,V_{n})}\oplus\sU_{\rG(n-2,V_n)}^{*}(1))\to\Lrho_{\widetilde{\hcoY}}^{-1}([V_n])$
which contracts the divisor $\mP(\sU_{V_{n}}^{*}(1))$
to $\rG(n-3,V_n)$, where $\sU_{\rG(n-2,V_n)}$ is the universal subbundle
of the Grassmannian $\rG(n-2,V_{n})$.
\end{prop}

\begin{proof}
Since the fiber under consideration
is contained in the branched locus of 
$\widehat{F} \to {F}_{\widetilde{\hcoY}}$,
we have only to describe
the fiber $\Gamma$ of $\widehat{F}\to
\widehat{G}$ over $[V_n]$,
where we consider $[V_n]$ is a point of the diagonal
of $\widehat{G}$.
Let $\Gamma'$ be the restriction over $[V_n]$ of the exceptional locus
of ${F}^{(4)}\to
{F}^{(1)}$. 
Then the fiber $\Gamma$ is nothing but the image of $\Gamma'$ 
under the divisorial contraction $F^{(4)}\to \hat F$.
Since the fiber of
$\Delta_{{F}^{(1)}}\to \widehat{G}$
over $[V_n]$ is $\mathrm{G}(n-2,V_n)$,
the variety $\Gamma'$ is a $\mP^{n-2}$-bundle over $\mathrm{G}(n-2,V_n)$.
By the definition of ${D}^{(4)}$,
we see that ${D}^{(4)}|_{\Gamma'}=\mathrm{F}(n-3,n-2,V_n)$,
which is isomorphic to $\mP(\sU_{\rG(n-2,V_n)}^*(-1))$.
Therefore we may write $\Gamma'=\mP(\sA^*)$,
where $\sA$ is the locally free sheaf of rank $n-2$ on
$\mathrm{G}(n-2,V_n)$ with a surjection $\sA\to \sU_{\rG(n-2,V_n)}(1)$.
Now we show the kernel of $\sA\to \sU_{\rG(n-2,V_n)}(1)$ is 
$\sO_{\mathrm{G}(n-2,V_n)}(2)$.
Note that the image of $\mathrm{F}(n-3,n-2,V_n)$
by the divisorial contraction
${{F}^{(4)}}\to
{\widehat{F}}$ is $\mathrm{G}(n-3, V_n)$.
Therefore, since the discrepancy of ${D}^{(4)}$
for ${F}^{(4)}\to
\widehat{F}$ is two,
and $\sO_{\mP(\sU_{\rG(n-2,V_n)}^*(-1))}(1)$ is the pull-back of $\sO_{\mathrm{G}(n-3, V_n)}(1)$,
we see that ${D}^{(4)}|_{\Gamma'}=
H_{{\mP(\sA^*)}}-2L$,
where $L$ is the pull-back of $\sO_{\mathrm{G}(n-2,V_n)}(1)$.
Thus 
the kernel of $\sA\to \sU_{\rG(n-2,V_n)}(1)$ is $\sO_{\mathrm{G}(n-2,V_n)}(2)$.
Since the exact sequence
$0\to \sO_{\mathrm{G}(n-2,V_n)}(2)\to \sA\to \sU_{\rG(n-2,V_n)}(1)\to 0$ splits,
we have $\sA^*\simeq \sO_{\mathrm{G}(n-2,V_n)}(-2)\oplus \sU_{\rG(n-2,V_n)}^*(-1)
\simeq (\sO_{\mathrm{G}(n-2,V_n)}\oplus \sU^*_{\rG(n-2,V_n)}(1))\otimes \sO_{\mathrm{G}(n-2,V_n)}(-2)$.
\end{proof}

We have obtained the following diagram:
\begin{equation}
\begin{matrix}\xymatrix{F^{(3)}\;\ar[r]\ar[d] & \; F^{(4)}\;\ar[d]\ar[r]^{\;_{\text{{\rm div. cont.}}}} & \;\widehat{F}\;\ar[d]\ar[r]^{\;_{\mZ_{2}\text{{\rm -quot.}}}} & \; F_{\widetilde{\hcoY}}\ar[d]\\
\widehat{G}'\;\ar[r] & \;\widehat{G}\;\ar@{=}[r] & \;\widehat{G}\;\ar[r]^{\;_{\mZ_{2}\text{{\rm -quot.}}}} & \; G_{\hcoY}}.
\end{matrix}\label{eq:flat-diagram}
\end{equation}
We show that $F^{(3)}\to\widehat{G}'$
gives a flattening of the fibration $F_{\widetilde{\hcoY}}\to G_{\hcoY}$.

\begin{prop}
\label{pro:Components-A-B}
$F^{(3)}\to\widehat{G}'$ is flat. 
More precisely, 
the fiber $Fib^{(3)}(V_{n-1},V_{n},V_{n}')$
of $F^{(3)}\to\widehat{G}'$ over a point $([V_{n-1}];[V_{n}],[V_{n}'])$
have the following descriptions\,$:$

\begin{myitem2}\item[$(1)$] $Fib^{(3)}(V_{n-1},V_{n},V_{n}')\simeq\mP(V_{n-1}^{*})\times\mP(V_{n-1}^{*})$
if $V_{n}\not=V_{n}'$. 

\item[$(2)$] $Fib^{(3)}(V_{n-1},V_{n},V_{n})$
consists of two irreducible components $A$ and $B$ with \[
A=\mP(\sO_{\rG(n-2,V_{n})}\oplus\sU_{\rG(n-2,V_{n})}^{*}(1))\big\vert_{\rG(n-2,V_{n-1})},\; B=Bl_{\Delta}\mP(V_{n-1}^{*})\times\mP(V_{n-1}^{*}),\]
where $A$ is the restriction of the projective bundle as in Lemma
$\ref{lem:inverse-rk1}$ over $\rG(n-2,V_{n-1})\subset\rG(n-2,V_{n})$. 

\item[$(3)$] The intersection $E_{AB}:=A\cap B$ is given by $E_{AB}=\mP(\sU_{\rG(n-2,V_{n})}^{*}(1))\big\vert_{\rG(n-2,V_{n-1})}\simeq\mP(T_{\mP(V_{n-1}^{*})})$
in $A$. Also, $E_{AB}$ in $B$ is the exceptional divisor of $Bl_{\Delta}\mP(V_{n-1}^{*})\times\mP(V_{n-1}^{*})$. 

\end{myitem2}\end{prop}
\begin{proof}
Part (1) follows from the construction of
${F}^{(2)}\to \widehat{G}'$.

We show Part (2).
The fiber of ${F}^{(2)}\to \widehat{G}'$
over a point $([V_{n-1}];[V_n],[V_n])$ is $\mP(V_{n-1}^*)\times \mP(V_{n-1}^*)$.
The intersection of the fiber $\mP(V_{n-1}^*)\times \mP(V_{n-1}^*)$ with 
the exceptional locus of ${F}^{(2)}\to {F}^{(1)}$ is
\[
\{
([V_{n-2}],[V_{n-2}];[V_{n-1}];[V_n],[V_n])\mid 
V_{n-2} \subset V_{n-1}
\}
\simeq \mP^{n-2},
\]
which is nothing but the diagonal of 
$\mP(V_{n-1}^*)\times \mP(V_{n-1}^*)$. Therefore we have $B$ as 
an irreducible component of the fiber
of ${F}^{(3)}\to \widehat{G}'$
over the point $([V_{n-1}];[V_n],[V_n])$.

Another component $A$ is a $\mP^{n-2}$-bundle over 
the diagonal of $\mP(V_{n-1}^*)\times \mP(V_{n-1}^*)$
since the exceptional divisor of ${F}^{(3)}\to
{F}^{(2)}$ is a $\mP^{n-2}$-bundle
over the exceptional locus of ${F}^{(2)}\to
{F}^{(1)}$.
Since the image on ${F}^{(1)}$
 of the diagonal $\Delta_{V_{n-1}}$ 
of $\mP(V_{n-1}^*)\times \mP(V_{n-1}^*)$ is equal to
$\mathrm{G}(n-2,V_{n-1})=\mP(V_{n-1}^*)$ in $\mathrm{G}(n-2,V_n)$,
 the image of $A$ in ${F}^{(4)}$
is the restriction of $\mP(\sO_{\mathrm{G}(n-2,V_n)}\oplus \sU^*_{\rG(n-2,V_n)}(1))$
over $\rG(n-2,V_{n-1})$.
Therefore we obtain the description of $A$ as in the statement
since $\sU^*_{\rG(n-2,V_n)}|_{\mP(V_{n-1}^*)}\simeq T_{\mP(V_{n-1}^*)}(-1)$ 
and $\sN_{\Delta_{V_{n-1}}}\cong T_{\mP(V_{n-1}^*)}$ for the normal 
bundle $\sN_{\Delta_{V_{n-1}}}$ of the diagonal $\Delta_{V_{n-1}}$.

It is easy to see the assertion about $A\cap B$. 
\end{proof}

\begin{rem}
In \cite[Thm.~3.7]{CHK}, the authors studied
the relationship between the Hilbert scheme $\hcoY_0$ of conics in 
$\rG(n-1,V)$ and the stable map compactification of the space of smooth conics in $\rG(n-1, V)$, which we denote by $\hcoY_{\rm{st}}$.
We interpret this by our study of the birational geometry of $\hcoY_0$. 

By Remark \ref{rem:blup},
$\hcoY_0\to \widetilde{\hcoY}$ is the blow-up along $G_{\sigma}$. 
By the blow-up $\hcoY_0\to \widetilde{\hcoY}$,
the fiber $\Lrho_{\widetilde{\hcoY}}^{-1}([V_n])$ becomes
the $\mP^{n-2}$-bundle $\mP(\sO_{\rG(n-2,V_{n})}\oplus\sU_{\rG(n-2,V_{n})}^{*}(1))\to \rG(n-2,V_n)$
as in Proposition \ref{lem:inverse-rk1}. 
Therefore the strict transform $\Gamma$ of $\Lrho_{\widetilde{\hcoY}}^{-1}(G^1_{\hcoY})$
is a $\mP^{n-2}$-bundle to $\rF(n-2,n,V)$, where we note that
$\rF(n-2,n,V)$ is isomorphic to the Hilbert scheme of lines in $\rG(n-1,V)$.
Let $\widetilde{\hcoY}_0\to \hcoY_0$ be the blow-up along $\Gamma$.
Then $\hcoY_{\rm{st}}$ is obtained by contracting the exceptional divisor
over $\Gamma$ to a $\mP^2$-bundle over $\rF(n-2,n,V)$.
\end{rem}

\subsection{The component $A$ of the fiber $Fib^{(3)}(V_{n-1},V_{n},V_{n})$ }

Let us fix $V_{n-1}$ and $V_{n}$ such that $V_{n-1}\subset V_n$
and consider the exceptional set $A$
in the fiber 
\[
\text{$Fib^{(3)}(V_{n-1},V_{n},V_{n})\simeq A\cup B$ over $([V_{n-1}];[V_{n}],[V_{n}])\in\widehat{G}'$.}
\] Since $A$ is $\mZ_{2}$-invariant,
this determines the corresponding set $A_{\widetilde{\hcoY}}$ in
the fiber $F_{\widetilde{\hcoY}}\to G_{\hcoY}$ over $[V_{n}]$. 
We note that $A\simeq \mP(\sO_{\rG(n-2,V_{n-1})}\oplus\sU_{\rG(n-2,V_{n-1})}^{*}(1))\simeq\mP(\sO_{\mP(V_{n-1}^{*})}\oplus T_{\mP(V_{n-1}^{*})})$
by Proposition \ref{lem:inverse-rk1}. 

\begin{prop}
\label{pro:Prop-C-A}
Define $A_{\hcoY_{2}}$ to be the strict transform
of $A_{\widetilde{\hcoY}}\subset\widetilde{\hcoY}$ under $\hcoY_{2}\to\widetilde{\hcoY}$,
and $A_{\hcoY_{3}}$ by the image of $A_{\hcoY_{2}}$ under the morphism
$\hcoY_{2}\to\hcoY_{3}$. 

\begin{enumerate}[$(1)$]
\item The morphism $A\to A_{\widetilde{\hcoY}}$
contracts the divisor $E_{AB}=\mP(\sU_{\rG(n-2,V_{n-1})}^{*}(1))$ to
$\rG(n-3,V_{n-1})$. 
\item
The image $\rG(n-3,V_{n-1})$ of $E_{AB}$ on $A_{\widetilde{\hcoY}}$
is the locus of $\sigma$-planes.
The locus $s_A$ of $\rho$-conics in $A$ is a section 
of $A\to \rG(n-2,V_{n-1})$
corresponding to an injection 
$\sO_{\mP(V_{n-1}^{*})}\to \sO_{\mP(V_{n-1}^{*})}\oplus T_{\mP(V_{n-1}^{*})}$.
\item $A_{\hcoY_{2}}\to A_{\widetilde{\hcoY}}$ is the blow-up
along the image $\tilde{s}_{A}$ in $A_{\widetilde{\hcoY}}$ of the
section $s_{A}$.

\item Let $\widehat{A}:=Bl_{s_{A}} A$ be the blow-up $\hat{A}$
of $A$ along the section $s_{A}$. There exists a natural morphism
$\widehat{A}\to A_{\hcoY_{2}}$, which is the blow-up of $A_{\hcoY_{2}}$
along the singular locus of $A_{\hcoY_{2}}$. 

\item $A_{\hcoY_{3}}\simeq A_{\hcoY_{2}}$ and $\Lpi_{A_{3}}\colon A_{\hcoY_{3}}\to\rG(n-3,V_{n-1})$
is a quadric cone fibration,
where $\Lpi_{A_{3}}:=\Lpi_{\hcoY_{3}}\vert_{A_{\hcoY_{3}}}$.
\def\AAAA{
\begin{xy}
(0,0)*+{\widehat{A}}="Ah",
(-15,-10)*+{A_{\hcoY_2}}="Aii",
(15,-10)*+{A}="A",
(-30,-10)*+{A_{\hcoY_3}}="Aiii",
(-30,-25)*+{\rG(n-3,V_{n-1})}="Pv",
(0,-20)*+{A_{\widetilde{\hcoY}}}="At",
(15,-25)*+{\mP(V_{n-1}^*)}="Pvs",
\ar "Ah";"A",
\ar "Ah";"Aii",
\ar "A";"At",
\ar "A";"Pvs",
\ar "Aii";"At",
\ar_{\simeq} "Aii";"Aiii",
\ar_{\Lpi_{A_{3}}} "Aiii";"Pv",
\end{xy} }\[
\begin{matrix}{\AAAA}\end{matrix}\]

\end{enumerate}\end{prop}
\begin{proof}
(1) follow from Proposition \ref{lem:inverse-rk1}.
(4) is clear and (3) follows once we show (2) 
since $\hcoY_2\to \widetilde{\hcoY}$ is 
the blow-up along $G_{\rho}$ by Proposition \ref{prop:gendescr} (1) and
$\widetilde{s}_A=G_{\rho}\cap A_{\widetilde{\hcoY}}$.

To show (2) and (5),
as in the discussion of the subsections \ref{subsection:dimV=4}
and \ref{sub:tildeY-Y}, we first consider the case where $\dim V=4$ and
then use the results to the general cases.
In case $\dim V=4$, 
$A_{\hcoY_2}=A_{\widetilde{\hcoY}}$ is isomorphic to $\mP(1^2,2)$ by
Proposition \ref{pro:Components-A-B}.
Moreover, 
by Proposition \ref{prop:n=3fib} (5) (d),
the vertex corresponds to a $\sigma$-plane and
$A_{\widetilde{\hcoY}}\cap G_{\rho}$ is a $\mP^1$ which is the image of a section of
$A\simeq \mP(\sO_{\mP^1}\oplus \sO_{\mP^1}(2))$
associated to an injection $\sO_{\mP^1}\to\sO_{\mP^1}\oplus \sO_{\mP^1}(2)$.
Therefore, we also have $A_{\hcoY_3}\simeq A_{\hcoY_2}\simeq \mP(1^2,2)$. 
Now we have finished the proof in case $\dim V=4$.

We turn to the general cases.
First we immediately obtain (5) by the results in case $n=4$
since $\hcoY_3\to \rG(n-3,V)$ is the family of $\hcoY_3=\rG(3,\wedge^2 (V/V_{n-3}))$ for
$4$-dimensional spaces $V/V_{n-3}$.
By comparing the singularities between $A_{\hcoY_2}$ and $A_{\widetilde{\hcoY}}$, we see that the image of $E_{AB}$ is the locus of $\sigma$-planes.
Then the locus $s_A$ of $\rho$-conics in $A$ is disjoint from $E_{AB}$.
Since $s_A$ is a section of $A\to \rG(n-2,V_{n-1})$,
$s_A$ corresponds to an injection 
$\sO_{\mP(V_{n-1}^{*})}\to \sO_{\mP(V_{n-1}^{*})}\oplus T_{\mP(V_{n-1}^{*})}$.

Finally we show $\Prt_{\rho}\cap A_{\hcoY_{3}}\simeq\mP(\frQ_{V_{n-1}})$.
Note that $\Prt_{\rho}\cap A_{\hcoY_{3}}$ is isomorphic to the exceptional
divisor $G$ of $\widehat{A}\to A$, which we determine now. Let $\sI_{s_{A}}$
be the ideal sheaf of the section $s_{A}$ in $A$. Note that $\sO_{\mP(\sO_{\mP(V_{n-1}^{*})}\oplus T_{\mP(V_{n-1}^{*})})}(1)|_{s_{A}}=\sO_{s_{A}}$.
Tensoring $0\to\sI_{s_{A}}\to\sO_{A}\to\sO_{s_{A}}\to0$ with $\sO_{\mP(\sO_{\mP(V_{n-1}^{*})}\oplus T_{\mP(V_{n-1}^{*})})}(1)$
and pushing forward to $\mP(V_{n-1}^{*})$, we see that $\sI_{s_{A}}/\sI_{s_{A}}^{2}\simeq\Omega_{\mP(V_{n-1}^{*})}$.
Therefore $G$ is isomorphic to $\mP(T_{\mP(V_{n-1}^{*})})$. Since
$\mP(T_{\mP(V_{n-1}^{*})})$ is isomorphic to the incident variety $\{([V_{n-3}],[V_{n-2}])\mid V_{n-3}\subset V_{n-2}\}\subset\mP(V_{n-1})\times\mP(V_{n-1}^{*})$,
it follows that $\mP(T_{\mP(V_{n-1}^{*})})$ is isomorphic to $\mP(T_{\mP(V_{n-1})}(-1))$.
\end{proof}

\begin{rem}
\label{rem:freedescr}
Based on Remark \ref{rem:WPS} and Proposition \ref{pro:Prop-C-A},
we can obtain the following description of $A_{\hcoY_3}\to \rG(n-3,V_{n-1})$,
which follows by noting the fiber of $\hcoY_3\to \rG(n-3,V)$ over 
$[V_{n-3}]$ is isomorphic to $\rG(3,\wedge^2 (V/V_{n-3}))$:

Take a point $[V_{n-3}]\in \rG(n-3,V_{n-1})$ and 
let $\Gamma$ be the fiber of $A_{\hcoY_3}\to \rG(n-3,V_{n-1})$ over $[V_{n-3}]$.
The vertex of the quadric cone $\Gamma$ corresponds to
the $\sigma$-plane ${\rm P}_{V_n/V_{n-3}}=\{\mC^2\subset V_n/V_{n-3}\}$,
where we denote by ${\rm P}_{V_n/V_{n-3}}$ the $\sigma$-plane in
$\rG(3,\wedge^2 (V/V_{n-3}))$ corresponding to the $\sigma$-plane
${\rm P}_{V_{n-3}V_n}$.  
Points $[{\rm P}_{V_{n-2}/V_{n-3}}]$ which correspond to $\rho$-planes and are contained in $\Gamma$ satisfy $V_{n-3}\subset V_{n-2}$,
where we follows the same convention for $\rho$-planes as for $\sigma$-planes.
Since $\Gamma$ is the cone over the Veronese curve $v_2(\mP(V_{n-1}/V_{n-3}))$,
it is swept out by lines joining $[{\rm P}_{V_n/V_{n-3}}]$ and $[{\rm P}_{V_{n-2}/V_{n-3}}]$ such that
$V_{n-3}\subset V_{n-2}\subset V_{n-1}$.

By this description, we see that
$\Prt_{\rho}\cap A_{\hcoY_{3}}\simeq\mP(\frQ_{V_{n-1}})\simeq \rF(n-3,n-2,V_{n-1}),$
where $\frQ_{V_{n-1}}$ is the universal quotient bundle on $\rG(n-3,V_{n-1})$.
\end{rem}

\appendix

\section{\label{sec:Appendix-B}{\bf Proof of Proposition \ref{lem:appendixB-UU-solve}}}
\label{app:aU}

\begin{proof}[Proof of Proposition $\ref{lem:appendixB-UU-solve}$]
If $\dim a_{U}\geq n-3$, it is easy to see $\rank\varphi_{U}\leq1$
by writing down $U$ using a basis of $a_U$.
This shows one direction of (1).

We show the converse direction of (1).
If $\varphi_U=0$, then $\mP(U)$ is a plane contained in $\rG(n-1,V)$, and hence
is a $\rho$- or $\sigma$-plane. Therefore, we see that $\dim a_U\geq n-3$ holds
by (\ref{eq:xxxxx}). 
Now we assume that $\rank \varphi_U=1$. Then $q:=\hcoY_3\cap \mP(U)$ is the $\tau$-conic
which is the zero locus of $\varphi_U$. 
We will argue depending on the rank of the $\tau$-conic $q$.

Assume that $\rank q=3$. 
Note that the dual of the universal subbundle $\eS^{*}$ on $\mathrm{G}(n-1,V)$ restricts
as $\eS^{*}|_{q}\simeq\sO(1)_{\mP^{1}}^{\oplus 2}\oplus\sO_{\mP^{1}}^{\oplus n-3}$,
or $\sO_{\mP^{1}}(2)\oplus\sO_{\mP^{1}}^{\oplus n-2}$ since $\eS^{*}$
is generated by its global sections and $\deg\eS^{*}|_{q}=\deg \sO_{\rG(n-1,V)}(1)|_q=2$ since $q$ is a conic.
Let $Q$ be the image of $\mP({\eS}|_{q})$ under the natural map $\varphi_{\eS}\colon \mP(\eS)\to \mP(V)$.
Then there are two possibilities; (i) the degree of $\mP({\eS}|_{q})\to Q$
is two and $Q$ is a $(n-1)$-plane, i.e., a quadric of rank 1, or (ii)
the degree of $\mP({\eS}|_{q})\to Q$ is one and $Q$ is a quadric
of rank $4$ or $3$ depending on 
$\eS^{*}|_{q}\simeq\sO(1)_{\mP^{1}}^{\oplus 2}\oplus\sO_{\mP^{1}}^{\oplus n-3}$,
or $\sO_{\mP^{1}}(2)\oplus\sO_{\mP^{1}}^{\oplus n-2}$ respectively. 
The case (i) is
excluded since 
if $Q$ were a $(n-1)$-plane $\mP(V_{n})$, then $q\subset\{[U]\in\mathrm{G}(n-1,V)\mid U\subset V_{n}\}$
and $q$ would be a $\sigma$-conic
by definition, a contradiction. 
The case (ii) with 
$\eS^{*}|_{q}\simeq\sO_{\mP^{1}}(2)\oplus\sO_{\mP^{1}}^{\oplus n-2}$
also is excluded since if this happened, then $q$ would be a $\rho$-conic.
Therefore we have the case (ii) with  
$\eS^{*}|_{q}\simeq\sO(1)_{\mP^{1}}^{\oplus 2}\oplus\sO_{\mP^{1}}^{\oplus n-3}$.
Then we see that $q$ is a connected family of $(n-1)$-planes 
in the rank four quadric $Q$.
Since all the rank four quadrics are $\SL(V)$-equivalent,
we see that any rank three conic $q$ is also $\SL(V)$-equivalent.
Therefore we may assume that $q$ is of the form
as in Example \ref{ex:conics}. Then it is easy to see that 
$a_{U}=\langle {\bf e}_4,\dots, {\bf e}_n\rangle$ and hence $\dim a_U=n-3$.

Assume that $q$ is of rank two. Then $q$ is of the form as in
Example \ref{ex:ranktwo}. Since $q$ is a $\tau$-conic, 
$V_{n-2}\not =V'_{n-2}$ and $V_n\not =V'_n$.
Then it is easy to see that $a_U=V_{n-2}\cap V'_{n-2}$ and hence
$\dim a_U=n-3$.

Finally we assume that
$q$ is of rank one. Then the support of $q$ is a line $l$ and
$l$ is of the form as in Example \ref{ex:ranktwo}.
Let ${\bf e}_1,\dots,{\bf e}_{n-2}$ be a basis of $V_{n-2}$ and
${\bf e}_1,\dots,{\bf e}_n$ be a basis of $V_n$.
Then $l$ is spanned by ${\bf e}_1\wedge \dots\wedge {\bf e}_{n-2}\wedge {\bf e}_{n-1}$ and ${\bf e}_1\wedge \dots\wedge {\bf e}_{n-2}\wedge {\bf e}_{n}$.
Now we pass from $\wedge^{n-1} V$ to $\wedge^2 V^*$
and let $U'$ and $l'$ the $3$-plane in $\wedge^2 V^*$ and
the line in $\mP(\wedge^2 V^*)$.
Then $l'$ is spanned by ${\bf v}_1:={\bf e}^*_n\wedge {\bf e}^*_{n+1}$ and
${\bf v}_2:={\bf e}^*_{n-1}\wedge {\bf e}^*_{n+1}$. 
Let ${\bf w}:=\sum_{i<j} a_{ij} {\bf e}^*_i\wedge {\bf e}^*_j$ be a vector such that ${\bf v}_1, {\bf v}_2, {\bf w}$ span $U'$.
Then $\rG(2,V^*)\cap \mP(U')$ is a rank one conic.
Solving the equation 
\[
(\lambda_1 {\bf v}_1+\lambda_2 {\bf v}_2+\mu{\bf w})\wedge
(\lambda_1 {\bf v}_1+\lambda_2 {\bf v}_2+\mu{\bf w})=0,
\] 
we obtain the equation of $\rG(2,V^*)\cap \mP(U')$.
Thus 
$\rG(2,V^*)\cap \mP(U')$ is a rank one conic iff
${\bf v}_1\wedge {\bf w}={\bf v}_2\wedge {\bf w}=0$.
Therefore 
we have ${\bf w}=a_{n-1n}{\bf e}^*_{n-1}\wedge {\bf e}^*_n+
(\sum_{i\leq n-2} a_{in+1} {\bf e}^*_i)\wedge {\bf e}^*_{n+1}$.
Taking these back to $\wedge^{n-1} V$, we see that
$U$ is spanned by
${\bf e}_1\wedge \dots\wedge {\bf e}_{n-2}\wedge {\bf e}_{n-1}$ and ${\bf e}_1\wedge \dots\wedge {\bf e}_{n-2}\wedge {\bf e}_{n}$
and
${\bf w}=a_{n-1n}{\bf e}_1\wedge\dots \wedge {\bf e}_{n-2}+
\sum_{i\leq n-2} a_{in+1} {\bf e}_1\wedge \dots \wedge \check{\bf e}_i\wedge\dots
{\bf e}_{n}$, where $\check{\bf e}_i$ means that ${\bf e}_i$ is removed.
Therefore it is easy to see that
$a_U$ is spanned by vectors $\sum b_i {\bf e}_i$ with
$b_{n-1}=b_n=b_{n+1}=0$ and
$\sum (-1)^{n-i} a_{in+1} b_i=0$.
Therefore $\dim a_U\geq n-3$.
\end{proof}

\section{\label{sec:Appendix-A}{\bf The {}``double spin'' coordinates of $\mathrm{G}(3,6)$}}

In this appendix, we set $V_{4}=\mathbb{C}^{4}$ with the standard
basis. We can write the irreducible decomposition (\ref{eq:spin})
as \[
\wedge^{3}(\wedge^{2}V_{4})=\Sigma^{(3,1,1,1)}V_{4}\,\oplus\,\Sigma^{(2,2,2,0)}V_{4}\simeq\mathsf{S}^{2}V_{4}\,\oplus\,\mathsf{S}^{2}V_{4}^{*},\]
 where $\Sigma^{\beta}$ is the Schur functor. We define the projective
space $\mathbb{P}(\wedge^{3}(\wedge^{2}V_{4}))=\mathbb{P}(\mathsf{S}^{2}V_{4}\,\oplus\,\mathsf{S}^{2}V_{4}^{*})$.
The homogeneous coordinate of $\mathbb{P}(\mathsf{S}^{2}V_{4}\,\oplus\,\mathsf{S}^{2}V_{4}^{*})$
is naturally introduced by $[v_{ij},w_{kl}]$, where $v_{ij}$ and
$w_{kl}$ are entries of $4\times4$ symmetric matrices. Let $\mathcal{I}=\left\{ \{i,j\}\mid1\leq i<j\leq4\right\} $
the index set to write the standard basis of $\wedge^{2}V_{4}$, then
the homogeneous coordinate of $\mathbb{P}(\wedge^{3}(\wedge^{2}V_{4}))$
is naturally given by the $[p_{IJK}]$ where $p_{IJK}$ is totally
anti-symmetric for the indices $I,J,K\in\mathcal{I}.$ These two coordinates
are related by the above irreducible decomposition. Focusing on the
different symmetry properties of the Schur functors, it is rather
straightforward to decompose $p_{IJK}$ into the two components. When
we use the signature function defined by ${\bf e}_{i_{1}}\wedge{\bf e}_{i_{2}}\wedge{\bf e}_{i_{3}}\wedge{\bf e}_{i_{4}}=\epsilon^{i_{1}i_{2}i_{3}i_{4}}{\bf e}_{1}\wedge{\bf e}_{2}\wedge{\bf e}_{3}\wedge{\bf e}_{4}$
for a basis ${{\bf e}_{1},..,{\bf e}_{4}}$ of $V_{4}$, they are
given by \begin{equation}
v_{ij}=\frac{1}{6}\sum_{k,l,m,n}\epsilon^{klmn}p_{[ik][jl][mn]},\quad w_{kl}=\frac{1}{6}\sum_{a,b,c}\sum_{m,n,q}\epsilon^{kabc}\epsilon^{lmnq}p_{[am][bn][cq]},\label{eq:vw-general-fromula}\end{equation}
where the square brackets in $p_{[ij][kl][mn]}$ represents the anti-symmetric
extensions of the indices, i.e., $p_{[ij][J][K]}=p_{\{ij\}[J][K]}$
for $i<j$ while $p_{[ij][J][K]}=-p_{\{ji\}[J][K]}$ for $i\geq j$.
For convenience, we write them in the following (symmetric) matrices:
\begin{equation}
\begin{aligned}v=(v_{ij})=\left(\begin{matrix}2p_{{\bf 124}} & p_{{\bf 134}}+p_{{\bf 125}} & p_{{\bf 234}}+p_{{\bf 126}} & p_{{\bf 146}}-p_{{\bf 245}}\\
 & 2p_{{\bf 135}} & p_{{\bf 235}}+p_{{\bf 136}} & p_{{\bf 156}}-p_{{\bf 345}}\\
 &  & 2p_{{\bf 236}} & p_{{\bf 256}}-p_{{\bf 346}}\\
 &  &  & 2p_{{\bf 456}}\end{matrix}\right),\\
w=(w_{kl})=\left(\begin{matrix}2p_{{\bf 356}} & -p_{{\bf 346}}-p_{{\bf 256}} & p_{{\bf 345}}+p_{{\bf 156}} & p_{{\bf 235}}-p_{{\bf 136}}\\
 & 2p_{{\bf 246}} & -p_{{\bf 245}}-p_{{\bf 146}} & p_{{\bf 126}}-p_{{\bf 234}}\\
 &  & 2p_{{\bf 145}} & p_{{\bf 134}}-p_{{\bf 125}}\\
 &  &  & 2p_{{\bf 123}}\end{matrix}\right),\end{aligned}
\label{eq:vw-plucker}\end{equation}
 where we ordered the index set $\mathcal{I}$ as $\{{\bf 1},{\bf 2},...,{\bf 6}\}=\{\{1,2\},\{1,3\},\{2,3\},\{1,4\},$
$\{2,4\},$ $\{3,4\}\}$. Inverting the relations (\ref{eq:vw-plucker}),
we can write the Pl\"ucker relations among $p_{IJK}$ in terms of
the entries of $v$ and $w$. After some algebra, we find:
\begin{prop}
\label{prop:B1}
The Pl\"ucker ideal $I_{G}$ of
$\mathrm{G}(3,6)\subset\mathbb{P}(\wedge^{3}(\wedge^{2}V_{4}))$ is
generated by \begin{equation}
\begin{aligned}|v_{IJ}|-\epsilon_{I\check{I}}\epsilon_{J\check{J}}|w_{\check{I}\check{J}}|\qquad(I,J\in\mathcal{I}),\qquad\qquad\\
(v.w)_{ij},\;\;(v.w)_{ii}-(v.w)_{jj}\;\;(i\not=j,1\leq i,j\leq4),\end{aligned}
\label{eq:Ivw}\end{equation}
 where $\check{I}$ represents the complement of $I$, i.e., $x\in\mathcal{I}$
such that $x\cup I=\{1,2,3,4\}$ and similarly for $\check{J}$. $|v_{IJ}|$
and $|w_{IJ}|$ represent the $2\times2$ minors of $v$ and $w$,
respectively, with the rows and columns specified by $I$ and $J$.
$\epsilon_{I\check{I}}$ is the signature of the permutation of the
'ordered' union $I\cup\check{I}$. $(v.w)_{ij}$ is the $ij$-entry
of the matrix multiplication $v.w$. 
\end{prop}

For all $[v,w]\in V(I_{G})\simeq\mathrm{G}(3,6)$, we show the following
relations (I.1)-(I.5):

\vspace{0.2cm}

\noindent \textbf{(I.1)} $\det\, v=\det\, w$.

By the Laplace expansion of the determinant of $4\times4$ matrix
$v$, we have $\det\, v=\sum_{J\in\mathcal{I}}\epsilon_{J\check{J}}|v_{IJ}||v_{\check{I}\check{J}}|$.
Then, using the first relations of (\ref{eq:Ivw}), we obtain the
equality.

\vspace{0.2cm}

\noindent \textbf{(I.2)} $v.w=\pm\sqrt{\det\, w}\, id_{4}$, where
$id_{4}$ is the $4\times4$ identity matrix.

Note that the second line of (\ref{eq:Ivw}) may be written in a matrix
form $v.w=d\, id_{4}$ with $d=(v.w)_{11}=\cdots=(v.w)_{44}$. Then,
by {(I.1)}, we have $\det v\cdot w=(\det\, w)^{2}=d^{4}$ and hence
$d^{4}-(\det\, w)^{2}=(d^{2}-\det\, w)(d^{2}+\det\, w)=0$. We consider
a special case; $v=a\, id_{4}$, $w=a\, id_{4}$. Then $d=(v.w)_{11}=a^{2}$.
Therefore $d^{2}=a^{4}=\det\, w$ must holds for all since $V(I_{G})\simeq\mathrm{G}(3,6)$
is irreducible. Hence $d=\pm\sqrt{\det\, w}$ as claimed.

\vspace{0.2cm}

\noindent \textbf{(I.3)} ${\rm rk}\, w\not=3$ and also ${\rm rk}\, v\not=3$.

Assume ${\rm rk}\, w=3$, then from (I.2) we have $v.w=0$, which
implies ${\rm rk}\, v\leq1$. However, this contradicts the first
relations of (\ref{eq:Ivw}). Hence ${\rm rk}\, w\not=3$. By symmetry,
we also have ${\rm rk}\, v\not=3$.

\vspace{0.2cm}

\noindent \textbf{(I.4)} ${\rm rk}\, w=2\Leftrightarrow{\rm rk}\, v=2$.

When ${\rm rk}\, w=2$, we see ${\rm rk}\, v\geq2$ by the first relations
of (\ref{eq:Ivw}). From (I.1) and (I.3), we must have ${\rm rk}\, v=2$.
The converse follows in the same way.

\vspace{0.2cm}

\noindent \textbf{(I.5)} ${\rm rk}\, w\leq1\Leftrightarrow{\rm rk}\, v\leq1$.

This is immediate from the the first relations of (\ref{eq:Ivw}).

\vspace{0.5cm}




$\;$

\noindent {\footnotesize Department of Mathematics, Gakushuin University, 
Toshima-ku,Tokyo 171-8588,$\,$Japan }{\footnotesize \par}

\noindent {\footnotesize e-mail address: hosono@math.gakushuin.ac.jp} 

\vspace{5pt}

\noindent {\footnotesize Graduate School of Mathematical Sciences,
University of Tokyo, Meguro-ku,Tokyo 153-8914,$\,$Japan }{\footnotesize \par}

\noindent {\footnotesize e-mail address: takagi@ms.u-tokyo.ac.jp} 
\end{document}